\documentclass[a4paper]{article}

\usepackage[utf8]{inputenc} % to be able to add accents in the .tex file, etc.
\usepackage[T1]{fontenc} % for accents, etc., in the output document
\usepackage{amsmath,amsthm,latexsym} % symbols & theorem environments
\usepackage[fixamsmath]{mathtools}
\mathtoolsset{showmanualtags,mathic,centercolon} % number all equations
\usepackage{amssymb,amsfonts} % fonts & symbols
\usepackage{booktabs} % \toprule, \midrule, \bottomrule for proper table formatting
\usepackage[UKenglish]{babel} % force UK english language settings
\usepackage[svgnames]{xcolor} % more colours available (for alert macro)
\usepackage{algpseudocode} % algorithmicx
\usepackage{algorithm} % put algorithmic in a floating object
\usepackage{url}

\usepackage{bibentry}
\newcommand{\ignore}[1]{}
\newcommand{\nobibentry}[1]{[{\let\nocite\ignore\bibentry{#1}}]}

\usepackage[margin=1.25in]{geometry} % use small margins
\usepackage[labelfont=sl,textfont=sl]{subcaption} % newer version of subfig
\usepackage[font={small,sl},labelsep=period]{caption} % format figure & table captions

\numberwithin{equation}{section} %if want numbering by section # (i.e. 3.1, 3.2 in section 3)
\theoremstyle{plain}
\newtheorem{theorem}{Theorem}[section] %section = Theorem 1.7.4, chapter = Theorem 1.38, nothing = Theorem 15
\newtheorem{lemma}[theorem]{Lemma} %give everything the same numbering with the [theorem] tag
\newtheorem{corollary}[theorem]{Corollary}

\theoremstyle{definition}
\newtheorem{definition}[theorem]{Definition}

\newtheorem{assumption}[theorem]{Assumption}

\theoremstyle{remark}
\newtheorem{remark}[theorem]{Remark}

\newcommand{\appref}[1]{Appendix~\ref{#1}}
\newcommand{\secref}[1]{Section~\ref{#1}}
\newcommand{\defref}[1]{Definition~\ref{#1}}
\newcommand{\thmref}[1]{Theorem~\ref{#1}}
\newcommand{\lemref}[1]{Lemma~\ref{#1}}
\newcommand{\corref}[1]{Corollary~\ref{#1}}
\renewcommand{\algref}[1]{Algorithm~\ref{#1}}
\newcommand{\assref}[1]{Assumption~\ref{#1}}
\newcommand{\figref}[1]{Figure~\ref{#1}}
\newcommand{\subfigref}[1]{Figure~\ref{#1}}
\newcommand{\remref}[1]{Remark~\ref{#1}}
\newcommand{\tabref}[1]{Table~\ref{#1}}

\newcommand{\smartqed}{\hfill\qed}

\newcommand{\R}{\mathbb{R}} % Real numbers
\newcommand{\bigO}{\mathcal{O}} % big O notation
\DeclareMathOperator*{\argmin}{arg\,min} % argmin
\DeclareMathOperator*{\argmax}{arg\,max} % argmax
\newcommand{\pd}[2]{\frac{\partial #1}{\partial #2}} % partial derivative
\newcommand{\defeq}{:=} % defined equal
\newcommand{\grad}{\nabla} % gradient

\def\be{\begin{equation}}
\def\ee{\end{equation}}

\renewcommand{\b}[1]{\mathbf{#1}} % bold
\renewcommand{\t}[1]{\widetilde{#1}} % tilde
\newcommand{\bx}{\b{x}}
\newcommand{\by}{\b{y}}
\newcommand{\bs}{\b{s}}
\newcommand{\br}{\b{r}}
\newcommand{\bee}{\b{e}}
\newcommand{\bg}{\b{g}}
\newcommand{\bem}{\b{m}}

\algrenewcommand\algorithmicrequire{\textbf{Input:}}
\algrenewcommand\algorithmicensure{\textbf{Output:}}

\begin{document}
\title{A Derivative-Free Gauss-Newton Method}
\author{
	Coralia Cartis \thanks{Mathematical Institute, University of Oxford, Radcliffe Observatory Quarter, Woodstock Road, Oxford, OX2 6GG, United Kingdom (\texttt{cartis@maths.ox.ac.uk}).}
	\and
	Lindon Roberts \thanks{Mathematical Institute, University of
          Oxford, Radcliffe Observatory Quarter, Woodstock Road,
          Oxford, OX2 6GG, United Kingdom
          (\texttt{robertsl@maths.ox.ac.uk}). This work was supported
          by the EPSRC Centre For Doctoral Training in Industrially
          Focused Mathematical Modelling (EP/L015803/1) in
          collaboration with the Numerical Algorithms Group Ltd.}
}
\date{\today}
\maketitle
%\tableofcontents

\begin{abstract}
% \nobibliography* % full citation in abstract
We present DFO-GN, a derivative-free version of the Gauss-Newton method for solving nonlinear least-squares problems.
As is common in derivative-free optimization, DFO-GN uses
interpolation of function values to build a model of the objective,
which is then used within a trust-region framework to give a
globally-convergent algorithm requiring $\bigO(\epsilon^{-2})$
iterations to reach approximate first-order criticality within tolerance $\epsilon$.
This algorithm is a simplification of the method from [H. Zhang, A. R. Conn, and K. Scheinberg, A Derivative-Free Algorithm for Least-Squares Minimization, SIAM J. Optim., 20 (2010), pp. 3555--3576], where we replace quadratic models for each residual with linear models.
We demonstrate that DFO-GN performs comparably to the method of Zhang et al.~in terms of objective evaluations, as well as having a substantially faster runtime and improved scalability.
\end{abstract}

\\


% NOTE: ABSTRACT IN WRAPPER DOCUMENTS

\section{Introduction}
Over the last 15--20 years, there has been a resurgence and increased effort devoted to developing efficient methods for derivative-free optimization (DFO) --- that is, optimizing an objective using only function values.
These methods are useful to many applications \cite{Conn2009}, for instance, when the objective function is a black-box function or legacy code (meaning manual computation of derivatives or algorithmic differentiation is impractical), has stochastic noise (so finite differencing is inaccurate) or expensive to compute (so the evaluation of a full $n$-dimensional gradient is intractable).
There are several popular classes of DFO methods, such as direct and pattern search, model-based and evolutionary algorithms \cite{Powell1998, Kolda2003,Powell2007a,Custodio2017}.
Here, we consider model-based methods, which capture curvature in the objective well \cite{Custodio2017} and have been shown to have good practical performance \cite{More2009}.

Model-based methods typically use a trust-region framework for selecting new iterates, which ensures global convergence, provided we can build a sufficiently accurate model for the objective \cite{Conn2000}. 
The model-building process most commonly uses interpolation of quadratic functions, as originally proposed by Winfield \cite{Winfield1973} and later developed by Conn, Scheinberg and Toint \cite{Conn1996,Conn1997} and Powell \cite{Powell1994,Powell2002}. 
Another common choice for model-building is to use radial basis functions \cite{Wild2008,Oeuvray2009}.
Global convergence results exist in both cases \cite{Conn2009a,Conn2009,Wild2013}.
Several codes for model-based DFO are available, including those by Powell \cite{Zhang_URL} and others (see e.g.~\cite{Conn2009,Custodio2017} and references therein).

{\bf Summary of contributions} In this work, we consider nonlinear least-squares minimization, without constraints in the theoretical developments but allowing bounds in the implementation.
Model-based DFO is naturally suited to exploiting problem structure, and in this work we propose a method inspired by the classical Gauss-Newton method for derivative-based optimization (e.g.~\cite[Chapter 10]{Nocedal2006}). This method, which we call DFO-GN (Derivative-Free Optimization using Gauss-Newton), is a simplification of the method by Zhang, Conn and Scheinberg \cite{Zhang2010}. It constructs linear interpolants for each residual, requiring exactly $n+1$ points on each iteration, which is less than common proposals that generally insist on (partial or full)
quadratic local models for each residual. In addition to proving theoretical guarantees for DFO-GN in terms of global convergence and worst-case complexity, we provide an implementation that is
a modification of Powell's BOBYQA and that we extensively test and compare with existing state of the art DFO solvers. We show that little to nothing is lost by our simplified approach in terms of
algorithm performance on a given evaluation budget, when applied to smooth and noisy,  zero- and non-zero residual problems. Furthermore, significant gains are made in terms of reduced computational cost of the interpolation
problem (leading to a runtime reduction of at least a factor of $7$) and memory costs of storing the models. These savings result in a substantially faster runtime and improved scalability of DFO-GN compared to the implementation DFBOLS  in \cite{Zhang2010}.

%Because of the least-squares structure, we still have a quadratic model for the overall objective.

{\bf Relevant existing literature.} In \cite{Zhang2010}, each residual function is approximated by a quadratic interpolating model, using function values from $p\in[n+1, (n+1)(n+2)/2]$ points.
A quadratic (or higher-order) model for the overall least-squares objective is built from the models for each residual function, that takes into account full
quadratic terms in the models asymptotically but allows the use of simpler models early on in the run of the algorithm. The DFBOLS implementation in \cite{Zhang2010}
is shown to perform better than Powell's BOBYQA on a standard least-squares test set.
%A similar framework is used in \cite{Wild2017}. 
A similar derivative-free framework for nonlinear least-squares problems is POUNDERS by Wild \cite{Wild2017}:
it also constructs quadratic interpolation models for each residual, but takes them all into account in the objective model construction on each iteration.
%; it also differs from \cite{Zhang2010}  in the choice of interpolation points.
% and uses this to build a quadratic model for the full objective for which a trust region step is calculated.
%It is different in the way it constructs the full model Hessian (using all second-order information at all iterations) and chooses the interpolation points at each iteration.
In its implementation, it allows parallel computation of each residual component, and accepts previously-computed evaluations as an input providing extra information for the solver.
We also note the connection to \cite{Bergou2016}, which considers a Levenberg-Marquardt method for nonlinear least-squares when gradient evaluations are noisy; the framework is that
of probabilistic local models, and it uses a regularization parameter rather than trust region to ensure global convergence. The algorithm is applied and further developed for data assimilation problems,
with careful quantification of noise and algorithm parameters. Using linear vector models for objectives which are a composition of a (possibly nonconvex) vector function with a (possibly nonsmooth) convex function, such as a sum of squares, was also considered in \cite{Garmanjani2016}. There, worst-case complexity bounds for a general model-based trust-region DFO method applied to such objectives are established. Our approach differs in that it is designed specifically for nonlinear least-squares, and 
uses an algorithmic framework that is much closer to the software of Powell \cite{Powell2009}. Finally, we note a mild connection to the approach in \cite{AokiEtAl2014}, 
where multiple solutions to nonlinear inverse problems are sought by means of a two-phase method, where in the first phase,  low accuracy solutions are
obtained by building a linear regression model from a (large) cloud of points and moving each point to its corresponding, slightly perturbed, Gauss-Newton step.

%Firstly,
%a (large) cloud of points is generated from which a linear regression model is built; then, each point is moved to its corresponding, slightly perturbed, Gauss-Newton step. This produces a large %number of low-accuracy solutions to the system. The method then switches to Broyden's method for each point separately to get high accuracy solutions. 

{\bf Further details of contributions.} 
%In DFO-GN, w linear interpolants for each residual, requiring exactly $n+1$ points.
%Because of the least-squares structure, we still have a quadratic model for the overall objective.
%Aside from the model construction, the resulting algorithm is very similar to the one in \cite{Zhang2010}.
%However, 
In terms of theoretical guarantees, 
we extend the global convergence results in \cite{Zhang2010}  to allow inexact solutions to the trust-region subproblem (given by the usual Cauchy decrease condition), a simplification of the so-called `criticality phase' and `safety step', and prove first-order convergence of the whole sequence of iterates $\bx_k$ rather than a subsequence. We also provide a worst-case complexity analysis and show an iteration count which matches that of Garmanjani, J\'{u}dice and Vicente \cite{Garmanjani2016}, but with problem constants that correspond to second-order methods.
This reflects the fact that we capture some of the curvature in the objective (since linear models for residuals still give an approximate quadratic model for the least-squares objective), and so the complexity of DFO-GN sits between first- and second-order methods.

In the DFO-GN implementation, which is very much the focus of this work, the simplification from quadratic to linear models leads to a confluence of two approaches for analysing and improving the geometry of the interpolation set.
We compare DFO-GN  to Powell's general DFO solver BOBYQA and to  least-squares DFO solvers DFBOLS \cite{Zhang2010} (Fortran),
POUNDERS \cite{Wild2017} and our Python DFBOLS re-implementation Py-DFBOLS.
The primary test set is Mor\'e \& Wild \cite{More2009} where additionally, we also consider  noisy variants for each problem, perturbing the test set appropriately with
 unbiased Gaussian (multiplicative and additive),  and with additive $\chi^2$ noise;
we solve to low as well as high accuracy requirements for a given evaluation budget.
 We find --- and show by means of performance and data profiles --- that 
DFO-GN performs comparably well in terms of objective evaluations to the best of solvers, albeit with a small penalty for objectives with additive stochastic noise
and an even less penalty for nonzero-residuals. We then do a runtime comparison between DFO-GN and Py-DFBOLS on the same test set and settings, 
comparing like for like, and find that DFO-GN is at least $7$ times faster; see \tabref{tab_runtime} for details.  We further investigate scalability features of DFO-GN.
We compare memory requirements and runtime for DFO-GN and DFBOLS on a particular nonlinear equation problem from CUTEst with growing problem dimension $n$; we find that
both  of these increase much more rapidly for the latter than the former (for example, for $n=2500$ DFO-GN's runtime is 2.5 times faster than the Fortran DFBOLS' for $n=1400$). To further
illustrate that the improved scalability of DFO-GN does not come at the cost of performance, we compare evaluation performance of DFO-GN and DFBOLS on 60 medium-size least-squares problems 
from CUTEst and find similarly good behaviour of DFO-GN as on the Mor\'e \& Wild set.
%However, we observe that due to the simpler approach in DFO-GN, we see substantial reductions in runtime and better scalability.

%\alert{I will add more here on numerics, a little bit. So more later tonight here}

{\bf Implementation.} 
%Our Python implementation of DFO-GN is available on GitHub; details are provided in \secref{sec_conclusion}.
Our Python implementation of DFO-GN is available on GitHub\footnote{\url{https://github.com/numericalalgorithmsgroup/dfogn}}, and is released under the open-source GNU General Public License.

{\bf Structure of paper.} 
In \secref{sec_algo} we state the DFO-GN algorithm.
We prove its global convergence to first-order critical points and worst case complexity in \secref{sec_convergence}.
Then we discuss the differences between this algorithm and its software implementation in \secref{sec_implementation}.
Lastly, in \secref{sec_numerics}, we compare DFO-GN to other model-based derivative-free least-squares solvers on a selection of test problems, including noisy and higher-dimensional problems.
We draw our conclusions in \secref{sec_conclusion}.

\section{DFO-GN Algorithm} \label{sec_algo}
Here, our focus is unconstrained nonlinear least-squares minimization
\be \min_{\bx\in\R^n}f(\bx) \defeq \frac{1}{2}\|\br(\bx)\|^2 = \frac{1}{2}\sum_{i=1}^{m}r_i(\bx)^2, \label{eq_ls_definition} \ee
where $\br(\bx) \defeq \begin{bmatrix}r_1(\bx) & \cdots & r_m(\bx)\end{bmatrix}^{\top}$ maps $\R^n\to\R^m$ and is continuously differentiable with $m\times n$ Jacobian matrix $[J(\bx)]_{i,j}=\pd{r_i(\bx)}{x_j}$, although these derivatives are not available.
Typically $m\geq n$ in practice, but we do not require this for our method.
Throughout, $\|\cdot\|$ refers to the Euclidean norm for vectors or largest singular value for matrices, unless otherwise stated, and we define $B(\bx,\Delta)\defeq\{\by\in\R^n : \|\by-\bx\|\leq\Delta\}$ to be the closed ball of radius $\Delta>0$ about $\bx\in\R^n$.

In this section, we introduce the DFO-GN algorithm for solving \eqref{eq_ls_definition} using linear interpolating models for $\br$.

\subsection{Linear Residual Models} \label{sec_linear_models}
In the classical Gauss-Newton method, we approximate $\br$ in the neighbourhood of an iterate $\bx_k$ by its linearization: $\br(\by) \approx \br(\bx_k) + J(\bx_k)(\by-\bx_k)$, where $J(\bx)\in\R^{m\times n}$ is the Jacobian matrix of first derivatives of $\br$.
For DFO-GN, we use a similar approximation, but replace the Jacobian with an approximation to it calculated by interpolation.

Assume at iteration $k$ we have a set of $n+1$ interpolation points $Y_k\defeq \{\by_0,\ldots,\by_n\}$ in $\R^n$ at which we have evaluated $\br$.
This set always includes the current iterate; for simplicity of notation, we assume $\by_0= \bx_k$.
We then build the model
\be \br(\bx_k+\bs) \approx \bem_k(\bs) \defeq \br(\bx_k) + J_k\bs, \label{eq_linear_models} \ee
by finding the unique $J_k\in\R^{m\times n}$ satisfying the interpolation conditions
\be \bem_k(\by_t-\bx_k) = \br(\by_t), \qquad \text{for $t=1,\ldots,n$,} \label{eq_interp_conditions} \ee
noting that the other interpolation condition $\bem_k(\b{0})=\br(\bx_k)$ is automatically satisfied by \eqref{eq_linear_models}\footnote{We could have formulated a linear system to solve for the constant term in \eqref{eq_linear_models} as well as $J_k$, but this system becomes poorly conditioned as the algorithm progresses and the points $Y_k$ get closer together.}.
We can find $J_k$ by solving the $n\times n$ system
\be \begin{bmatrix} (\by_1-\bx_k)^{\top} \\ \vdots \\ (\by_n-\bx_k)^{\top}\end{bmatrix} \b{j}_{k,i} = \begin{bmatrix} r_i(\by_1) - r_i(\bx_k) \\ \vdots \\ r_i(\by_n) - r_i(\bx_k) \end{bmatrix}, \label{eq_linear_interp_system_Jonly} \ee
for each $i=1,\ldots,m$, where the rows of $J_k$ are $\b{j}_{k,i}^{\top}$.
This system is invertible whenever the set of vectors $\{\by_1-\bx_k,\ldots,\by_n-\bx_k\}$ is linearly independent.
We ensure this in the algorithm by routines which improve the geometry of $Y_k$ (in a specific sense to be discussed in \secref{sec_geometry}).

%\begin{remark}
%\end{remark}

Having constructed the linear models for each residual \eqref{eq_ls_definition}, we need to construct a model for the full objective $f$. 
To do this we simply take the sum of squares of the residual models, namely,
\be f(\bx_k+\bs) \approx m_k(\bs) \defeq \frac{1}{2}\|\bem_k(\bs)\|^2 = f(\bx_k) + \bg_k^{\top}\bs + \frac{1}{2}\bs^{\top} H_k \bs, \label{eq_gn_full_model_dfo} \ee
where $\bg_k \defeq J_k^{\top}\br(\bx_k)$ and $H_k\defeq J_k^{\top}J_k$.

\subsection{Trust Region Framework}
The DFO-GN algorithm is based on a trust-region framework \cite{Conn2000}. 
In such a framework, we use our model for the objective \eqref{eq_gn_full_model_dfo}, and maintain a parameter $\Delta_k>0$ which characterizes the region in which we `trust' our model to be a good approximation to the objective; the resulting `trust region' is $B(\bx_k,\Delta_k)$.
At each iteration, we use our model to find a new point where we expect the objective to decrease, by (approximately)
%\footnote{It is possible to solve \eqref{eq_tr_subproblem} to global optimality, for which efficient algorithms exist \cite{Conn2000}.} 
solving the `trust region subproblem'
\be \bs_k \approx \argmin_{\|\bs\|\leq\Delta_k} m_k(\bs). \label{eq_tr_subproblem} \ee
If this new point $\bx_k + \bs_k$ gives a sufficient objective reduction, we accept the step ($\bx_{k+1}\gets\bx_k+\bs_k$), otherwise we reject the step ($\bx_{k+1}\gets\bx_k$).
We also use this information to update the trust region radius $\Delta_k$.
The measure of `sufficient objective reduction' is the ratio
\be r_k = \frac{\text{actual reduction}}{\text{predicted reduction}} \defeq \frac{f(\bx_k) - f(\bx_k+\bs_k)}{m_k(\b{0}) - m_k(\bs_k)}. \label{eq_tr_ratio} \ee

This framework applies to both derivative-based and derivative-free settings.
However in a DFO setting, we also need to update the interpolation set $Y_k$ to incorporate the new point $\bx_k+\bs_k$, and have steps to ensure the geometry of $Y_k$ does not become degenerate (see \secref{sec_geometry}).

A minimal requirement on  the calculation of $\bs_k$ to ensure global convergence is the following.

\begin{assumption} \label{ass_cauchy_decrease}
	Our method for solving \eqref{eq_tr_subproblem} gives a step $\bs_k$ satisfying the sufficient (`Cauchy') decrease condition
	\be m_k(\b{0}) - m_k(\bs_k) \geq c_1 \|\bg_k\| \min\left(\Delta_k, \frac{\|\bg_k\|}{\max(\|H_k\|,1)}\right), \label{eq_cauchy_decrease} \ee
	for some $c_1\in[1/2, 1]$ independent of $k$.
\end{assumption}

This standard condition is not onerous, and can be achieved with $c_1=1/2$ by one iteration of steepest descent with exact linesearch applied to the model $m_k$ \cite{Conn2000}.
In Zhang et al.~\cite{Zhang2010}, convergence is only proven when \eqref{eq_tr_subproblem} is solved to full optimality; here we use this weaker assumption.

\subsection{Geometry Considerations} \label{sec_geometry}
It is crucial that model-based DFO algorithms ensure the geometry of $Y_k$ does not become degenerate; an example where ignoring geometry causes algorithm failure is given by Scheinberg and Toint \cite{Scheinberg2010}.

To describe the notion of `good' geometry, we need the Lagrange polynomials of $Y_k$.
In our context of linear approximation, the Lagrange polynomials are the basis $\{\Lambda_0(\bx), \ldots, \Lambda_n(\bx)\}$ for the $(n+1)$-dimensional space of linear functions on $\R^n$ defined by
\be \Lambda_l(\by_t) = \delta_{l,t}, \qquad \text{for all $l,t\in\{0,\ldots,n\}$.} \label{eq_lagrange_defn} \ee
Such polynomials exist whenever the matrix in \eqref{eq_linear_interp_system_Jonly} is invertible \cite[Lemma 3.2]{Conn2009}; when this condition holds, we say that $Y_k$ is \emph{poised} for linear interpolation.

The notion of geometry quality is then given by the following \cite{Conn2007}.

\begin{definition}[$\Lambda$-poised]
	Suppose $Y_k$ is poised for linear interpolation.
	Let $B\subset\R^n$ be some set, and $\Lambda\geq 1$.
	Then we say that $Y_k$ is $\Lambda$-poised in $B$ if $Y_k\subset B$  and
	\be \max_{t=0,\ldots,n} \: \max_{\bx\in B} |\Lambda_t(\bx)| \leq \Lambda, \label{eq_lambda_poised_definition} \ee
	where $\{\Lambda_0(\bx),\ldots,\Lambda_n(\bx)\}$ are the Lagrange polynomials for $Y_k$.
\end{definition}

In general, if $Y_k$ is $\Lambda$-poised with a small $\Lambda$, then $Y_k$ has `good' geometry, in the sense that linear interpolation using points $Y_k$ produces a more accurate model.
The notion of model accuracy we use is given in \cite{Conn2007,Conn2009a}:

\begin{definition}[Fully linear, scalar function] \label{def_fully_linear_scalar}
	A model $m_k\in C^1$ for $f\in C^1$ is \emph{fully linear} in $B(\bx_k,\Delta_k)$ if %$m_k$ has Lipschitz continuous gradient and
	\begin{align}
		|f(\bx_k+\bs) - m_k(\bs)| &\leq \kappa_{ef}\Delta_k^2, \label{eq_fully_linear_scalar_f} \\
		\|\grad f(\bx_k+\bs) - \grad m_k(\bs)\| &\leq \kappa_{eg}\Delta_k, \label{eq_fully_linear_scalar_g}
	\end{align}
	for all $\|\bs\|\leq\Delta_k$, where $\kappa_{ef}$ and $\kappa_{eg}$ are independent of $\bs$, $\bx_k$ and $\Delta_k$.
\end{definition}

In the case of a vector model, such as \eqref{eq_ls_definition}, we use an analogous definition as in \cite{Grapiglia2016}, which is equivalent, up to a change in constants, to the definition in \cite{Garmanjani2016}.

\begin{definition}[Fully linear, vector function] \label{def_fully_linear_vector}
	A vector model $\bem_k\in C^1$ for $\br\in C^1$ is fully linear in $B(\bx_k,\Delta_k)$ if
	\begin{align}
		\|\br(\bx_k+\bs) - \bem_k(\bs)\| &\leq \kappa_{ef}^r\Delta_k^2, \label{eq_fully_linear_vector_f} \\
		\|J(\bx_k+\bs) - J^m(\bs)\| &\leq \kappa_{eg}^r\Delta_k, \label{eq_fully_linear_vector_g}
	\end{align}
	for all $\|\bs\|\leq\Delta_k$, where $J^m$ is the Jacobian of $\bem_k$, and $\kappa_{ef}^r$ and $\kappa_{eg}^r$ are independent of $\bs$, $\bx_k$ and $\Delta_k$.
\end{definition}

In \secref{sec_interp_fully_linear}, we show that if $Y_k$ is $\Lambda$-poised, then $\bem_k$ \eqref{eq_ls_definition} and $m_k$ \eqref{eq_gn_full_model_dfo} are fully linear in $B(\bx_k,\Delta_k)$, with constants that depend on $\Lambda$.

\subsection{Full Algorithm Specification}
A full description of the DFO-GN algorithm is provided in \algref{alg_dfogn}.

In each iteration, if $\bg_k$ is small, we apply a `criticality phase'.
This ensures that $\Delta_k$ is comparable in size to $\|\bg_k\|$, which makes $\Delta_k$, as well as $\|\bg_k\|$,
 a good measure of progress towards optimality.
After computing the trust region step $\bs_k$, we then apply a `safety phase', also originally from Powell \cite{Powell2003}.
In this phase, we check if $\|\bs_k\|$ is too small compared to the lower bound $\rho_k$ on the trust-region radius, and if so we reduce $\Delta_k$ and improve the geometry of $Y_k$, without evaluating $\br(\bx_k+\bs_k)$.
The intention of this step is to detect situations where our trust region step will likely not provide sufficient function decrease without evaluating the objective, which would be wasteful.
If the safety phase is not called, we evaluate $\br(\bx_k+\bs_k)$ and determine how good the trust region step was, accepting any point which achieved sufficient objective decrease.
There are two possible causes for the situation $r_k<\eta_1$ (i.e.~the trust region step was `bad'): the interpolation set is not good enough, or $\Delta_k$ is too large.
We first check the quality of the interpolation set, and only reduce $\Delta_k$ if necessary.

An important feature of DFO-GN, due to Powell \cite{Powell2003}, is that it maintains not only the (usual)  trust region radius  $\Delta_k$ (used in \eqref{eq_tr_subproblem} and in checking $\Lambda$-poisedness), but also a lower bound on it, $\rho_k$.
This mechanism is useful when we reject the trust region step, but the geometry of $Y_k$ is not good (the `Model Improvement Phase').
In this situation, we do not want to shrink $\Delta_k$ too much, because it is likely that the step was rejected because of the poor geometry of $Y_k$, not because the trust region was too large.
The algorithm floors $\Delta_k$ at $\rho_k$, and only shrinks $\Delta_k$ when we reject the trust region step \emph{and} the geometry of $Y_k$ is good (so the model $m_k$ is accurate) --- in this situation, we know that reducing $\Delta_k$ will actually be useful.
%In our implementation, we potentially floor $\Delta_k$ at $\rho_k$ for several unsuccessful iterations before finally reducing $\rho_k$, giving us multiple chances to improve the geometry of $Y_k$ %(see \secref{sec_other_implementation_diffs}).

\begin{algorithm}
	\small{
	\begin{algorithmic}[1]
		\Require Starting point $\bx_0\in\R^n$ and initial trust region radius $\Delta_0^{init}>0$. 
		\Statex Parameters are $\Delta_{max}\geq\Delta_0^{init}$, criticality threshold $\epsilon_C>0$, criticality scaling $\mu>0$, trust region radius scalings $0<\gamma_{dec}<1<\gamma_{inc}\leq\overline{\gamma}_{inc}$ and $0<\alpha_1<\alpha_2<1$, acceptance thresholds $0 < \eta_1 \leq \eta_2 < 1$, safety reduction factor $0 < \omega_S < 1$, safety step threshold $0 < \gamma_S < 2c_1/(1+\sqrt{1+2c_1})$, poisedness constant $\Lambda\geq1$.
		\State Build an initial interpolation set $Y_0$ of size $n+1$, with $\bx_0\in Y_0$. Set $\rho_0^{init}=\Delta_0^{init}$.
		\For{$k=0,1,2,\ldots$} 
			\State Given $\bx_k$ and $Y_k$, solve the interpolation problem \eqref{eq_linear_interp_system_Jonly} and form $m_k^{init}$ \eqref{eq_gn_full_model_dfo}. \label{ln_loop}
			\If{$\|\bg_k^{init}\| \leq \epsilon_C$}
				\State \underline{Criticality Phase}: using \algref{alg_criticality} (Appendix B), modify $Y_k$ and find $\Delta_k\leq\Delta_k^{init}$ such that $Y_k$ is $\Lambda$-poised in $B(\bx_k,\Delta_k)$ and $\Delta_k\leq\mu\|\bg_k\|$, where $\bg_k$ is the gradient of the new $m_k$. Set $\rho_k= \min(\rho_k^{init}, \Delta_k)$.
			\Else
				\State Set $m_k=m_k^{init}$, $\Delta_k=\Delta_k^{init}$ and $\rho_k=\rho_k^{init}$.
			\EndIf
			\State Approximately solve the trust region subproblem \eqref{eq_tr_subproblem} to get step $\bs_k$ satisfying \assref{ass_cauchy_decrease}.
			\If{$\|\bs_k\| < \gamma_S \rho_k$}
				\State \underline{Safety Phase}: Set $\bx_{k+1}=\bx_k$ and $\Delta_{k+1}^{init} = \max(\rho_k, \omega_S \Delta_k)$, and form $Y_{k+1}$ by making $Y_k$ $\Lambda$-poised  in $B(\bx_{k+1},\Delta_{k+1}^{init})$.
				\State If $\Delta_{k+1}^{init} = \rho_k$, set $(\rho_{k+1}^{init}, \Delta_{k+1}^{init}) = (\alpha_1\rho_k, \alpha_2\rho_k)$, otherwise set $\rho_{k+1}^{init}=\rho_k$.
				\State \textbf{goto} line \ref{ln_loop}.
			\EndIf
			\State Calculate ratio $r_k$ \eqref{eq_tr_ratio}.
			\State Accept/reject step and update trust region radius: set
			\be \bx_{k+1} = \begin{cases}\bx_k + \bs_k, & r_k \geq \eta_1, \\ \bx_k, & r_k < \eta_1, \end{cases} \quad \text{and} \quad \Delta_{k+1}^{init} = \begin{cases}\min(\max(\gamma_{inc}\Delta_k, \overline{\gamma}_{inc}\|\bs_k\|), \Delta_{max}), & r_k \geq \eta_2, \\ \max(\gamma_{dec}\Delta_k, \|\bs_k\|, \rho_k), & \eta_1 \leq r_k < \eta_2, \\ \max(\min(\gamma_{dec}\Delta_k, \|\bs_k\|), \rho_k), & r_k < \eta_1. \end{cases} \ee
			\If{$r_k \geq \eta_1$}
				\State Form $Y_{k+1}=Y_k\cup\{\bx_{k+1}\}\setminus\{\by_t\}$ for some $\by_t\in Y_k$ and set $\rho_{k+1}^{init}=\rho_k$.
			\ElsIf{$Y_k$ is not $\Lambda$-poised in $B(\bx_k,\Delta_k)$}
				\State \underline{Model Improvement Phase}: Form $Y_{k+1}$ by making $Y_k$ $\Lambda$-poised in $B(\bx_{k+1}, \Delta_{k+1}^{init})$ and set $\rho_{k+1}^{init}=\rho_k$.
			\Else\:\:[$r_k < \eta_1$ \textbf{and} $Y_k$ is $\Lambda$-poised in $B(\bx_k,\Delta_k)$]
				\State \underline{Unsuccessful Phase}: Set $Y_{k+1}=Y_k$, and if $\Delta_{k+1}^{init} = \rho_k$, set $(\rho_{k+1}^{init}, \Delta_{k+1}^{init}) = (\alpha_1\rho_k, \alpha_2\rho_k)$, otherwise set $\rho_{k+1}^{init}=\rho_k$. \label{ln_rho_redn}
			\EndIf
		\EndFor
	\end{algorithmic}
	} % end font size
	\caption{DFO-GN: Derivative-Free Optimization using Gauss-Newton.}
	\label{alg_dfogn}
\end{algorithm}

\begin{remark} \label{rem_geom}
	There are two different geometry-improving phases in \algref{alg_dfogn}.
	The first modifies $Y_k$ to ensure it is $\Lambda$-poised in $B(\bx_k,\Delta_k)$, and is called in the safety and model improvement phases.
	This can be achieved by \cite[Algorithm 6.3]{Conn2009}, for instance, where the number of interpolation systems \eqref{eq_linear_interp_system_Jonly} to be solved depends only on $\Lambda$ and $n$ \cite[Theorem 6.3]{Conn2009}. 
	
	The second, called in the criticality phase, also ensures $Y_k$ is $\Lambda$-poised, but it also modifies $\Delta_k$ to ensure $\Delta_k\leq\mu\|\bg_k\|$. 
	This is a more complicated procedure \cite[Algorithm 10.2]{Conn2009}, as we have a coupling between $\Delta_k$ and $Y_k$: ensuring $\Lambda$-poisedness in $B(\bx_k, \Delta_k)$ depends on $\Delta_k$, but since $\bg_k$ depends on $Y_k$, there is a dependency of $\Delta_k$ on $Y_k$.
	Full details of how to achieve this are given in \appref{sec_criticality_geom} --- we show that this procedure terminates as long as $\|\grad f(\bx_k)\|\neq 0$.
		In addition, there, we also prove the bound
	\be \min\left(\Delta_k^{init}, \text{const}\cdot\|\grad f(\bx_k)\|\right) \leq \Delta_k \leq \Delta_k^{init}. \label{eq_criticality_delta_bound_main_text} \ee
	If the procedure terminates in one iteration, then $\Delta_k=\Delta_k^{init}$, and we have simply made $Y_k$ $\Lambda$-poised, just as in the model-improving phase. 
	Otherwise, we do one of these model-improving iterations, then several iterations where both $\Delta_k$ is reduced and $Y_k$ is made $\Lambda$-poised.
%, which when determining worst-case complexity we treat similar to unsuccessful iterations.
	The bound \eqref{eq_criticality_delta_bound_main_text} tells us that these unsuccessful-type iterations do not occur when $\Delta_k^{init}$ (but not $\nabla f(\bx_k)$) is sufficiently small.\footnote{The more common approach in the criticality phase (e.g.~\cite{Conn2009,Garmanjani2016,Zhang2010}) is to use an extra parameter $0<\beta<\mu$ and floor $\Delta_k$ at $\beta\|\bg_k\|$, maintaining full linearity with extra assumptions on $\kappa_{ef}$ and $\kappa_{eg}$ \cite[Lemma 3.2]{Conn2009a}, and requiring all fully linear models have Lipschitz continuous gradient with uniformly bounded Lipschitz constant.}
\end{remark}

\begin{remark}
	In \lemref{lem_fully_linear}, we show that if $Y_k$ is $\Lambda$-poised, then $m_k$ is fully linear with constants that depend on $\Lambda$. For the highest level of generality, one may replace `make $Y_k$ $\Lambda$-poised' with `make $m_k$ fully linear' throughout \algref{alg_dfogn}.
	Any strategy which achieves fully linear models would be sufficient for the convergence results in \secref{lem_fully_linear}.
\end{remark}

\begin{remark}
There are several differences between \algref{alg_dfogn} and its implementation, which we fully detail in \secref{sec_implementation}.
In particular, there is no criticality phase in the DFO-GN implementation as we found it is not needed, but the safety step is preserved to keep continuity with the BOBYQA framework\footnote{Note that the criticality and safety phases have similar aims, namely, to keep the  approximate gradient and the step proportional to $\Delta_k$. However, showing global convergence/complexity without a criticality step is unprecedented in the literature, and left for future work.};  
also,  the geometry-improving phases are replaced by a simplified calculation. 
\end{remark}

\section{Convergence and complexity results} \label{sec_convergence}
We first outline the connection between $\Lambda$-poisedness of $Y_k$ and fully linear models.
We then prove global convergence of \algref{alg_dfogn} (i.e.~convergence from any starting point $\bx_0$) to first-order critical points, and determine its worst-case complexity.

\subsection{Interpolation Models are Fully Linear} \label{sec_interp_fully_linear}
To begin, we require some assumptions on the smoothness of $\br$.

\begin{assumption} \label{ass_smoothness}
	The function $\br$ is $C^1$ and its Jacobian $J(\bx)$ is Lipschitz continuous in $\mathcal{B}$, the convex hull of $\cup_k B(\bx_k,\Delta_{max})$, with constant $L_J$. 
	We also assume that $\br(\bx)$ and $J(\bx)$ are uniformly bounded in the same region; i.e.~$\|\br(\bx)\|\leq r_{max}$ and $\|J(\bx)\| \leq J_{max}$ for all $\bx\in\mathcal{B}$.
\end{assumption}

\begin{remark}
	If the level set $\mathcal{L}\defeq\{\bx : f(\bx) \leq f(\bx_0)\}$ is bounded, which is assumed in \cite{Zhang2010}, then $\bx_k\in\mathcal{L}$ for all $k$, so $\mathcal{B}$ is compact, from which \assref{ass_smoothness} follows. 
\end{remark}

\begin{lemma}
	If \assref{ass_smoothness} holds, then $\grad f$ is Lipschitz continuous in $\mathcal{B}$ with constant
	\be L_{\grad f} \defeq r_{max}L_J + J_{max}^2. \label{eq_ls_lipschitz} \ee
\end{lemma}
\begin{proof}
	We choose $\bx,\by\in\mathcal{B}$ and use the Fundamental Theorem of Calculus to compute
	\be \|\br(\by)-\br(\bx)\| = \left\|\int_{0}^{1}J(\bx + \alpha(\by-\bx))(\by-\bx)d\alpha\right\| \leq J_{max}\|\by-\bx\|. \ee
% 	\be \|\br(\by)-\br(\bx)\|^2 = \sum_{i=1}^{m}\left[\int_{0}^{1}\grad r_i(\bx+\alpha(\by-\bx))^{\top}(\by-\bx) d\alpha\right]^2 \leq m J_{max}^2 \|\by-\bx\|^2, \ee
	Now we use this, the identity $\|A\|=\|A^{\top}\|$, and $\grad f(\bx) = J(\bx)^{\top}\br(\bx)$ to compute
	\begin{align}
		\|\grad f(\by)-\grad f(\bx)\| &\leq \|(J(\by)-J(\bx))^{\top}\br(\by)\| + \|J(\bx)^{\top}(\br(\by)-\br(\bx))\|, \\
		&\leq \|J(\by)-J(\bx)\|\cdot\|\br(\by)\| + \|J(\bx)\|\cdot\|\br(\by)-\br(\bx)\|, \\
		&\leq L_J \|\by-\bx\| \cdot r_{max} + J_{max} \cdot J_{max} \|\by-\bx\|,
	\end{align}
	from which we recover \eqref{eq_ls_lipschitz}.
\end{proof}

We now state the connection between $\Lambda$-poisedness of $Y_k$ and full linearity of the models $\bem_k$ \eqref{eq_ls_definition} and $m_k$ \eqref{eq_gn_full_model_dfo}.

\begin{lemma} \label{lem_fully_linear}
	Suppose \assref{ass_smoothness} holds and $Y_k$ is $\Lambda$-poised in $B(\bx_k, \Delta_k)$.
%, and let $W$ be the matrix from \eqref{eq_linear_interp_system_Jonly}.
	Then $\bem_k$ \eqref{eq_ls_definition} is a fully linear model for $\br$ in $B(\bx_k,\Delta_k)$ in the sense of \defref{def_fully_linear_vector} with constants
	\be \kappa_{ef}^r = \kappa_{eg}^r + \frac{L_J}{2} \qquad \text{and} \qquad 
\kappa_{eg}^r = \frac{1}{2}L_J\left(\sqrt{n}C+2\right), 
%\kappa_{eg}^r = \frac{1}{2}L_J\left(\sqrt{n}\|\hat{W}^{-1}\|+2\right), 
\label{eq_fully_linear_vector_consts} \ee
	in \eqref{eq_fully_linear_vector_f} and \eqref{eq_fully_linear_vector_g}, where $C=\bigO(\Lambda)$.
%$\hat{W}\defeq W/\Delta_k$ and $\|\hat{W}^{-1}\|=\bigO(\Lambda)$.
	Under the same hypotheses, $m_k$ \eqref{eq_gn_full_model_dfo} is a fully linear model for $f$ in $B(\bx_k,\Delta_k)$ in the sense of \defref{def_fully_linear_scalar} with constants
	\be \kappa_{ef} = \kappa_{eg} + \frac{L_{\grad f} + (\kappa_{eg}^r\Delta_{max} + J_{max})^2}{2} \:\: \text{and} \:\: \kappa_{eg} = L_{\grad f} + \kappa_{eg}^r r_{max} + (\kappa_{eg}^r\Delta_{max}+J_{max})^2, \label{eq_fully_linear_scalar_consts} \ee
	in \eqref{eq_fully_linear_scalar_f} and \eqref{eq_fully_linear_scalar_g}, where $L_{\grad f}$ is from \eqref{eq_ls_lipschitz}.
	We also have the bound $\|H_k\|\leq (\kappa_{eg}^r\Delta_{max} + J_{max})^2$, independent of $\bx_k$, $Y_k$ and $\Delta_k$.
\end{lemma}
\begin{proof}
	See \appref{sec_fully_linear_pf}.
\end{proof}

\subsection{Global Convergence of DFO-GN} \label{sec_algo_convergence}
We begin with some nomenclature to describe certain iterations: we call an iteration (for which the safety phase is not called)
\begin{itemize}
	\item `Successful' if $\bx_{k+1}=\bx_k+\bs_k$ (i.e.~$r_k\geq\eta_1$), and `very successful' if $r_k\geq\eta_2$. Let $\mathcal{S}$ be the set of successful iterations $k$;
	\item `Model-Improving' if $r_k<\eta_1$ and the model-improvement phase is called (i.e.~$Y_k$ is not $\Lambda$-poised in $B(\bx_k,\Delta_k)$); and
	\item `Unsuccessful' if $r_k<\eta_1$ and the model-improvement phase is not called.
\end{itemize}
The results below are largely based on corresponding results in \cite{Zhang2010,Conn2009}.

\begin{assumption} \label{ass_bdd_hess}
	We assume that $\|H_k\|\leq\kappa_H$ for all $k$, for some $\kappa_H\geq1$\footnote{\lemref{lem_fully_linear} ensures \assref{ass_bdd_hess} holds whenever $Y_k$ is $\Lambda$-poised in $B(\bx_k,\Delta_k)$, but we need it to hold on all iterations. However, most of our analysis holds if this assumption is removed --- see \secref{sec_bdd_hess} for details.}.
\end{assumption}

\begin{lemma} \label{lem_min_delta_successful_dfo}
	Suppose \assref{ass_cauchy_decrease} holds.
	If the model $m_k$ is fully linear in $B(\bx_k,\Delta_k)$ and 
	\be \Delta_k \leq \min\left(\frac{c_1(1-\eta_2)\|\bg_k\|}{2\kappa_{ef}}, \frac{\|\bg_k\|}{\max(\|H_k\|,1)}\right), \label{eq_delta_lower_bound_dfo} \ee
	then either the $k$-th iteration is very successful or the safety phase is called.
\end{lemma}
\begin{proof}
	We compute
	\begin{align}
		|r_k-1| &= \left|\frac{(f(\bx_k)-f(\bx_k+\bs_k))-(m_k(\b{0})-m_k(\bs_k))}{m_k(\b{0})-m_k(\bs_k)}\right|, \\
		&\leq \frac{|f(\bx_k+\bs_k) - m_k(\bs_k)|}{|m_k(\b{0})-m_k(\bs_k)|} + \frac{|f(\bx_k)-m_k(\b{0})|}{|m_k(\b{0})-m_k(\bs_k)|}.
	\end{align}
	By assumption, $\Delta_k \leq \|\bg_k\|/\max(\|H_k\|,1)$. 
	Applying this to \eqref{eq_cauchy_decrease}, we have
	\be m_k(\b{0}) - m_k(\bs_k) \geq c_1\|\bg_k\|\Delta_k. \ee
	Using this and fully linearity \eqref{eq_fully_linear_scalar_f}, we get
	\be |r_k-1| \leq 2\left(\frac{\kappa_{ef}\Delta_k^2}{c_1 \|\bg_k\|\Delta_k}\right) \leq 1-\eta_2. \ee
	Thus $r_k\geq\eta_2$ and the iteration is very successful if $\|\bs_k\| \geq \gamma_S\rho_k$, otherwise the safety phase is called.
\end{proof}

The next result provides a lower bound on the size of the trust region step $\|\bs_k\|$, which we will later use to determine that the safety phase is not called when $\|\bg_k\|$ is bounded away from zero and $\Delta_k$ is sufficiently small.

\begin{lemma} \label{lem_cauchy_step_size}
	Suppose \assref{ass_cauchy_decrease} holds.
	Then the step $\bs_k$ satisfies
	\be \|\bs_k\| \geq \frac{2c_1}{1+\sqrt{1+2c_1}}\min\left(\Delta_k, \frac{\|\bg_k\|}{\max(\|H_k\|,1)}\right). \label{eq_trs_step_bound_fixed} \ee
\end{lemma}
\begin{proof}
	For convenience of notation, let $h_k\defeq\max(\|H_k\|,1)\geq 1$.
	Since $m_k(\b{0})-m_k(\bs_k)\geq 0$ from \eqref{eq_cauchy_decrease}, we have
	\be m_k(\b{0}) - m_k(\bs_k) = |m_k(\b{0})-m_k(\bs_k)| = \left|\bg_k^{\top}\bs_k + \frac{1}{2}\bs_k^{\top}H_k\bs_k\right| \leq \|\bs_k\|\cdot\|\bg_k\| + \frac{h_k}{2}\|\bs_k\|^2.\ee
	Substituting this into \eqref{eq_cauchy_decrease}, we get
	\be \frac{1}{2}\|\bs_k\|^2 + \frac{\|\bg_k\|}{h_k}\cdot\|\bs_k\| - c_1\frac{\|\bg_k\|}{h_k}\min\left(\Delta_k, \frac{\|\bg_k\|}{h_k}\right) \geq 0. \ee
	For this to be satisfied, we require
	\begin{align}
		\|\bs_k\| &\geq \sqrt{\frac{\|\bg_k\|^2}{h_k^2} + 2c_1 \frac{\|\bg_k\|}{h_k}\min\left(\Delta_k, \frac{\|\bg_k\|}{h_k}\right)} - \frac{\|\bg_k\|}{h_k}, \\
		&= \frac{2c_1 \frac{\|\bg_k\|}{h_k}\min\left(\Delta_k, \frac{\|\bg_k\|}{h_k}\right)}{\sqrt{\frac{\|\bg_k\|^2}{h_k^2} + 2c_1 \frac{\|\bg_k\|}{h_k}\min\left(\Delta_k, \frac{\|\bg_k\|}{h_k}\right)} + \frac{\|\bg_k\|}{h_k}}, \\
		&\geq \frac{2c_1 \frac{\|\bg_k\|}{h_k}\min\left(\Delta_k, \frac{\|\bg_k\|}{h_k}\right)}{\sqrt{\frac{\|\bg_k\|^2}{h_k^2} + 2c_1 \frac{\|\bg_k\|}{h_k}\left(\frac{\|\bg_k\|}{h_k}\right)} + \frac{\|\bg_k\|}{h_k}},
	\end{align}
	from which we recover \eqref{eq_trs_step_bound_fixed}.
\end{proof}

\begin{lemma} \label{lem_true_model_gradient}
	In all iterations, $\|\bg_k\| \geq \min(\epsilon_C, \Delta_k/\mu)$.
	Also, if $\|\grad f(\bx_k)\|\geq\epsilon>0$ then 
	\be \|\bg_k\|\geq \epsilon_g \defeq \min\left(\epsilon_C, \frac{\epsilon}{1+\kappa_{eg}\mu}\right) > 0. \label{eq_normg_lower} \ee
\end{lemma}
\begin{proof}
	Firstly, if the criticality phase is not called, then we must have $\|\bg_k\|=\|\bg_k^{init}\|>\epsilon_C$.
	Otherwise, we have $\Delta_k\leq\mu\|\bg_k\|$.
	Hence $\|\bg_k\| \geq \min(\epsilon_C, \Delta_k/\mu)$.
	
	To show \eqref{eq_normg_lower}, first suppose $\|\bg_k^{init}\|\geq\epsilon_C$.
	Then $\bg_k=\bg_k^{init}$ and \eqref{eq_normg_lower} holds.
	Otherwise, the criticality phase is called and $m_k$ is fully linear in $B(\bx_k,\Delta_k)$ with $\Delta_k\leq\mu\|\bg_k\|$.
	In this case, we have
	\be \epsilon \leq \|\grad f(\bx_k)\| \leq \|\grad f(\bx_k) - \bg_k\| + \|\bg_k\| \leq \kappa_{eg}\mu\|\bg_k\| + \|\bg_k\|, \ee
	and so $\|\bg_k\| \geq \epsilon/(1 + \kappa_{eg}\mu)$ and \eqref{eq_normg_lower} holds.
\end{proof}

\begin{lemma} \label{lem_lower_delta_bound_dfo}
	Suppose Assumptions \ref{ass_cauchy_decrease}, \ref{ass_smoothness} and \ref{ass_bdd_hess} hold.
	If $\|\grad f(\bx_k)\| \geq \epsilon > 0$ for all $k$, then $\rho_k\geq \rho_{min} > 0$ for all $k$, where
	\be \rho_{min} \defeq \min\left(\Delta_0^{init}, \frac{\omega_C\epsilon}{\kappa_{eg}+1/\mu}, \: \frac{\alpha_1 \epsilon_g}{\kappa_H}, \: \alpha_1\left(\kappa_{eg} + \frac{2\kappa_{ef}}{c_1(1-\eta_2)}\right)^{-1}\epsilon\right). \label{eq_rho_min}\ee
\end{lemma}
\begin{proof}
	From \lemref{lem_true_model_gradient}, we also have $\|\bg_k\|\geq\epsilon_g>0$ for all $k$.
	To find a contradiction, let $k(0)$ be the first $k$ such that $\rho_k<\rho_{min}$.
	That is, we have
	\be \rho_0^{init} \geq \rho_0 \geq \rho_1^{init} \geq \rho_1 \geq \cdots \geq \rho_{k(0)-1}^{init} \geq \rho_{k(0)-1} \geq \rho_{min} \qquad \text{and} \qquad \rho_{k(0)} < \rho_{min}. \ee
	We first show that 	\be \rho_{k(0)}=\rho_{k(0)}^{init}<\rho_{min}. \label{eq_rho_reduced_in_k0m1} \ee
From \algref{alg_dfogn}, we know that either $\rho_{k(0)}=\rho_{k(0)}^{init}$ or $\rho_{k(0)}=\Delta_{k(0)}$.
	Hence we must either have $\rho_{k(0)}^{init}<\rho_{min}$ or $\Delta_{k(0)}<\rho_{min}$. In the former case, there is nothing to prove; in the latter, 
using \lemref{lem_criticality_complexity}, we have that 
	\be \rho_{min} > \Delta_{k(0)} \geq \min\left(\Delta_{k(0)}^{init}, \frac{\omega_C \epsilon}{\kappa_{eg}+1/\mu}\right) \geq \min\left(\rho_{k(0)}^{init}, \frac{\omega_C \epsilon}{\kappa_{eg}+1/\mu}\right). \label{eq_deltak_lower} \ee
	Since $\rho_{min} \leq \omega_C \epsilon / (\kappa_{eg}+1/\mu)$, we therefore conclude that \eqref{eq_rho_reduced_in_k0m1} holds.

	Since $\rho_{min}\leq\Delta_0^{init}=\rho_0^{init}$, we therefore have $k(0)>0$ and $\rho_{k(0)-1} \geq \rho_{min} > \rho_{k(0)}^{init}$. 
	This reduction in $\rho$ can only happen from a safety step or an unsuccessful step, and we must have $\rho_{k(0)}^{init}=\alpha_1\rho_{k(0)-1}$, so $\rho_{k(0)-1} \leq \rho_{min}/\alpha_1$.
	If we had a safety step, we know $\|\bs_{k(0)-1}\| \leq \gamma_S \rho_{k(0)-1}$, but if we had an unsuccessful step, we must have $\gamma_{dec}\|\bs_{k(0)-1}\| \leq \min(\gamma_{dec}\Delta_{k(0)-1}, \|\bs_{k(0)-1}\|) \leq \rho_{k(0)-1}$.
	Hence in either case, we have
	\be \|\bs_{k(0)-1}\| \leq \min(\gamma_S, \gamma_{dec}^{-1})\rho_{k(0)-1} \leq \frac{1}{\alpha_1}\min(\gamma_S, \gamma_{dec}^{-1})\rho_{min} = \frac{\gamma_S}{\alpha_1}\rho_{min}, \ee
	since $\gamma_S<1$ and $\gamma_{dec}<1$.
	Hence by \lemref{lem_cauchy_step_size} we have
	\be c_2 \min\left(\Delta_{k(0)-1}, \frac{\epsilon_g}{\kappa_H}\right) \leq \|\bs_{k(0)-1}\| \leq \frac{\gamma_S}{\alpha_1}\rho_{min}, \label{eq_delta_small_1}\ee
	where $c_2\defeq 2c_1/(1+\sqrt{1+2c_1})$.
	Note that $\rho_{min} \leq \alpha_1 \epsilon_g / \kappa_H < (\alpha_1 c_2 \epsilon_g)/(\gamma_S \kappa_H)$, where in the last inequality we used
the choice of $\gamma_S$ in  \algref{alg_dfogn}. This inequality and the choice of $\gamma_S$, together with \eqref{eq_delta_small_1}, also imply 
	\be \Delta_{k(0)-1} \leq \frac{\gamma_S \rho_{min}}{\alpha_1 c_2} < \frac{\rho_{min}}{\alpha_1} \leq \min\left(\frac{\epsilon_g}{\kappa_H}, \left(\kappa_{eg} + \frac{2\kappa_{ef}}{c_1(1-\eta_2)}\right)^{-1}\epsilon\right). \label{eq_delta_small} \ee
	Then since $\Delta_{k_{(0)}-1} \leq \epsilon_g/\kappa_H$, \lemref{lem_cauchy_step_size} gives us $\|\bs_{k_{(0)}-1}\| \geq c_2\Delta_{k_{(0)}-1} > \gamma_S\rho_{k_{(0)}-1}$ and the safety phase is not called.

	If $m_k$ is not fully linear, then we must have either a successful or model-improving iteration, so $\rho_{k_{(0)}}^{init}=\rho_{k_{(0)}-1}$, contradicting \eqref{eq_rho_reduced_in_k0m1}.
	Thus $m_k$ must be fully linear.
	Now suppose that
	\be \Delta_{k_{(0)}-1} > \frac{c_1(1-\eta_2)\|\bg_{k_{(0)}-1}\|}{2\kappa_{ef}}. \label{eq_delta_large_assum} \ee
	Then using full linearity, we have
	\be \epsilon \leq \|\grad f(\bx_{k_{(0)}-1})\| \leq \kappa_{eg}\Delta_{k(0)-1} + \|\bg_{k_{(0)}-1}\| < \left(\kappa_{eg} + \frac{2\kappa_{ef}}{c_1(1-\eta_2)}\right)\Delta_{k_{(0)}-1}. \ee
	contradicting \eqref{eq_delta_small}.
	That is, \eqref{eq_delta_large_assum} is false and so together with \eqref{eq_delta_small}, we have \eqref{eq_delta_lower_bound_dfo}.
	Hence \lemref{lem_min_delta_successful_dfo} implies iteration $(k_0-1)$ was very successful (as we have already established the safety phase was not called), so $\rho_{k_{(0)}}^{init}=\rho_{k_{(0)}-1}$, contradicting \eqref{eq_rho_reduced_in_k0m1}.
\end{proof}

%\begin{remark}
%	In the last part of the above proof, we could have shown \eqref{eq_delta_lower_bound_dfo} more directly by replacing \eqref{eq_delta_small} with
%	\be \Delta_{k(0)-1} < \frac{\rho_{min}}{\alpha_1} \leq \min\left(\frac{\epsilon_g}{\kappa_H}, \frac{c_1(1-\eta_2)\epsilon_g}{2\kappa_{ef}}\right), \ee
%	however this would require $\rho_{min} = \bigO(\kappa_{ef}^{-1}\epsilon_g) = \bigO(\kappa_{ef}^{-1}\kappa_{eg}^{-1}\epsilon)$, which would give us complexity results with a worse dependency on $n$.
%	Our argument is based on \cite[Lemma 3.2]{Garmanjani2016}.
%\end{remark}

Our first convergence result considers the case where we have finitely-many successful iterations.

\begin{lemma} \label{lem_finitely_many_conv}
	Suppose Assumptions \ref{ass_cauchy_decrease}, \ref{ass_smoothness} and \ref{ass_bdd_hess} hold.
	If there are finitely many successful iterations, then $\lim_{k\to\infty}\Delta_k=\lim_{k\to\infty}\rho_k=0$ and $\lim_{k\to\infty}\|\grad f(\bx_k)\|=0$.
\end{lemma}
\begin{proof}
	Let $k_0=\max(\mathcal{S})$ be the last successful iteration, after which $\Delta_k$ is never increased.
	For any $k>k_0$, we possibly call the criticality phase, and then have either a safety phase, model-improving phase, or an unsuccessful step.
	
	If the model is not fully linear, then either it is made fully linear by the criticality phase, or we have a safety or model-improving step.
	In the first case, the model is made fully linear at iteration $k$; in the second and third, it is fully linear at iteration $k+1$. 
	That is, there is at most 1 iteration until the model is fully linear again. Therefore there are infinitely many $k>k_0$ where $m_k$ is fully 
linear and we have either a safety phase or an unsuccessful step.
	In both of these cases, $\Delta_k$ is reduced by a factor of at least $\max(\gamma_{dec}, \alpha_2, \omega_S)<1$, so $\Delta_k\to 0$ as $k\to\infty$.
	Since $\rho_k\leq\Delta_k$ at all iterations, we must also have $\rho_k\to 0$.
	
	For each $k>k_0$, let $j_k$ be the first iteration after $k$ where the model is fully linear.
	Then from the above discussion we know $0\leq j_k-k\leq 1$, and hence $\|\bx_{j_k}-\bx_k\|\leq \Delta_k\to 0$.
		We now compute
	\be \|\grad f(\bx_k)\| \leq \|\grad f(\bx_k) - \grad f(\bx_{j_k})\| + \|\grad f(\bx_{j_k}) - \bg_{j_k}\| + \|\bg_{j_k}\|. \label{eq_full-lin}\ee
	As $k\to\infty$, the first term of the right-hand side of \eqref{eq_full-lin} is bounded by $L_{\grad f} \Delta_k \to 0$, while the
 second term  is bounded by $\kappa_{eg}\Delta_{j_k}\to 0$; thus it remains to show that the last term goes to zero.
	
	By contradiction, suppose there exists $\epsilon>0$ and a subsequence $k_i$ such that $\|\bg_{j_{k_i}}\|\geq \epsilon > 0$.
%	For simplicity of notation, let $j_i\defeq j_{k_i}$.
	Then \lemref{lem_min_delta_successful_dfo} implies that for sufficiently small $\Delta_{j_{k_i}}$ (valid since $\Delta_{j_{k_i}}\to 0$), we get a very successful iteration or a safety step.
	Since $j_{k_i}\geq k_i > k_0$, this must mean we get a safety step.
	However, \lemref{lem_cauchy_step_size} implies that for sufficiently large $i$, we have $\|\bs_{j_{k_i}}\| \geq c_2 \Delta_{j_{k_i}} > \gamma_S \rho_{j_{k_i}}$, so the safety step cannot be called, a contradiction.
\end{proof}

\begin{lemma} \label{lem_delta_conv_dfo}
	Suppose Assumptions \ref{ass_cauchy_decrease}, \ref{ass_smoothness} and \ref{ass_bdd_hess} hold.
	Then $\lim_{k\to\infty}\Delta_k=0$ and so $\lim_{k\to\infty}\rho_k=0$.
\end{lemma}
\begin{proof}
	If $|\mathcal{S}|<\infty$, the proof of \lemref{lem_finitely_many_conv} gives the result.
	Thus, suppose there are infinitely many successful iterations (i.e.~$|\mathcal{S}|=\infty$).
	
	For any $k\in\mathcal{S}$, we have
	\be f(\bx_k) - f(\bx_{k+1}) \geq \eta_1\left(m_k(\b{0}) - m_k(\bs_k)\right) \geq \eta_1 c_1 \|\bg_k\| \min\left(\frac{\|\bg_k\|}{\kappa_H}, \Delta_k\right) > 0. \ee
	But since $\|\bg_k\| \geq \min(\epsilon_C, \Delta_k/\mu)$ (see \lemref{lem_true_model_gradient}), this means that
	\be f(\bx_k) - f(\bx_{k+1}) \geq \eta_1 c_1 \min(\epsilon_C, \mu^{-1}\Delta_k) \min\left(\frac{\min(\epsilon_C, \mu^{-1}\Delta_k)}{\kappa_H}, \Delta_k\right) > 0. \ee
	If we were to sum over all $k\in\mathcal{S}$, the left-hand side must be finite as it is bounded above by $f(\bx_0)$, remembering that $f \geq 0$ for least-squares objectives.
	The right-hand side is only finite if $\lim_{k\in\mathcal{S}\to\infty}\Delta_k=0$.
	The only time $\Delta_k$ is increased is if $k\in\mathcal{S}$, when it is increased by a factor of at most $\overline{\gamma}_{inc}$.
	For any given $k\notin\mathcal{S}$, let $j_k\in\mathcal{S}$ be the last successful iteration before $k$ (which exists whenever $k$ is sufficiently large).
	Then $\Delta_k \leq \overline{\gamma}_{inc} \Delta_{j_k}\to 0$.
	Lastly, $\rho_k\to 0$ since $\rho_k\leq\Delta_k$ throughout the algorithm.
\end{proof}

\begin{theorem} \label{thm_liminf}
	Suppose Assumptions \ref{ass_cauchy_decrease}, \ref{ass_smoothness} and \ref{ass_bdd_hess} hold.
	Then 
	\be \liminf_{k\to\infty}\|\grad f(\bx_k)\|=0. \ee
\end{theorem}
\begin{proof}
	If $|\mathcal{S}|<\infty$, then this follows from \lemref{lem_finitely_many_conv}.
	Otherwise, it follows from \lemref{lem_delta_conv_dfo} and \lemref{lem_lower_delta_bound_dfo}.
\end{proof}

\begin{theorem} \label{thm_lim}
	Suppose Assumptions \ref{ass_cauchy_decrease}, \ref{ass_smoothness} and \ref{ass_bdd_hess} hold.
	Then $\lim_{k\to\infty}\|\grad f(\bx_k)\|=0$.
\end{theorem}
\begin{proof}
	If $|\mathcal{S}|<\infty$, then the result follows from \lemref{lem_finitely_many_conv}.
	Thus, suppose there are infinitely many successful iterations (i.e.~$|\mathcal{S}|=\infty$).
	
	To find a contradiction, suppose there is a subsequence of successful iterations $t_j$ with $\|\grad f(\bx_{t_j})\|\geq \epsilon_0$ for some $\epsilon_0>0$ (note: we do not consider any other iteration types as $\bx_k$ does not change for these).
	Hence by \lemref{lem_true_model_gradient}, we must have $\|\bg_{t_j}\|\geq\epsilon>0$ for some $\epsilon$, where without loss of generality we assume that
	\be \epsilon < \min\left(\epsilon_C, \frac{\epsilon_0}{2+\kappa_{eg}\mu}\right). \ee
	
	Let $\ell_j$ be the first iteration $\ell_j>t_j$ such that $\|\bg_{\ell_j}\|<\epsilon$, which is guaranteed to exist by \thmref{thm_liminf}. 
	That is, there exist subsequences $t_j<\ell_j$ satisfying
	\be \|\bg_k\| \geq \epsilon\quad \text{for $k=t_j,\ldots,\ell_j-1$, and} \quad \|\bg_{\ell_j}\| < \epsilon. \ee
	Now consider the iterations $\mathcal{K} \defeq \cup_{j\geq 0}\{t_j,\ldots,\ell_j-1\}$.
	
	Since $\|\bg_k\|\geq\epsilon$ and $\Delta_k\to 0$ (\lemref{lem_delta_conv_dfo}), \lemref{lem_min_delta_successful_dfo} implies that for sufficiently large $k\in\mathcal{K}$, there can be no unsuccessful steps. 
	That is, iteration $k$ is a safety step, or if not it must be successful or model-improving.
	By the same reasoning as in the proof of \lemref{lem_finitely_many_conv}, since $\|\bg_k\|\geq\epsilon$, for $k\in\mathcal{K}$ sufficiently large, \lemref{lem_cauchy_step_size} implies that $\|\bs_k\|\geq c_2\Delta_k>\gamma_S\rho_k$, so the safety step is never called.
	
	For each successful iteration $k\in \mathcal{K}\cap\mathcal{S}$, we have from \eqref{eq_cauchy_decrease}
	\be f(\bx_k) - f(\bx_{k+1}) \geq \eta_1 (m_k(\b{0}) - m_k(\bs_k)) \geq \eta_1 c_1\epsilon\min\left(\frac{\|\bg_k\|}{\kappa_H}, \Delta_k\right) > 0, \label{eq_monotone_lim} \ee
	and for $k$ sufficiently large (so that $\Delta_k \leq \epsilon/\kappa_H$), we get
	\be \Delta_k \leq \frac{f(\bx_k) - f(\bx_{k+1})}{\eta_1 c_1 \epsilon}. \ee
	
	Since for $k\in\mathcal{K}$ sufficiently large, we either have successful or model-improving steps, and of these $\bx_k$ is only changed on successful steps, we have (for $j$ sufficiently large)
	\be \|\bx_{\ell_j} - \bx_{t_j}\| \leq \sum_{k=t_j, k\in\mathcal{K}\cap\mathcal{S}}^{\ell_j-1} \|\bx_k - \bx_{k+1}\| \leq \sum_{k=t_j, k\in\mathcal{K}\cap\mathcal{S}}^{\ell_j-1} \Delta_k \leq \frac{f(\bx_{t_j}) - f(\bx_{\ell_j})}{\eta_1 c_1 \epsilon}. \ee
	Since $\{f(\bx_k) : k\in\mathcal{K}\}$ is a monotone decreasing sequence by \eqref{eq_monotone_lim}, and bounded below (as $f\geq0$ for least-squares problems), it must converge.
	Thus $f(\bx_{t_j})-f(\bx_{\ell_j})\to 0$, and hence $\|\bx_{\ell_j}-\bx_{t_j}\|\to 0$ as $j\to\infty$.
	
	Now, we compute
	\be \|\grad f(\bx_{t_j})\| \leq \|\grad f(\bx_{t_j}) - \grad f(\bx_{\ell_j})\| + \|\grad f(\bx_{\ell_j}) - \bg_{\ell_j}\| + \|\bg_{\ell_j}\|. \ee
	Similarly to \lemref{lem_finitely_many_conv}, the first term goes to zero as $j\to\infty$ since $\grad f$ is continuous and $\|\bx_{\ell_j}-\bx_{t_j}\|\to 0$.
	Since $\|\bg_{\ell_j}\|<\epsilon < \epsilon_C$, the criticality step is called for iteration $\ell_j$, so $m_{\ell_j}$ is fully linear on $B(\bx_{\ell_j}, \Delta_{\ell_j})$ for $\Delta_{\ell_j}\leq\mu\|\bg_{\ell_j}\|$.
	Hence the second term is bounded by $\kappa_{eg}\Delta_{\ell_j}\leq \kappa_{eg}\mu\epsilon$.
	Lastly, the third term is bounded by $\epsilon$ by definition of $\ell_j$.
	
	All together, this means that for sufficiently large $j$,
	\be \|\grad f(\bx_{t_j})\| \leq \epsilon + \kappa_{eg}\mu\epsilon + \epsilon = (2+\kappa_{eg}\mu)\epsilon < \epsilon_0, \ee
	and we have our contradiction.
\end{proof}

\subsection{Worst-Case Complexity}
Next, we bound the number of iterations and objective evaluations until $\|\grad f(\bx_k)\|<\epsilon$.
We know such a bound exists from \thmref{thm_liminf}.
Let $i_{\epsilon}$ be the last iteration before $\|\grad f(\bx_{i_{\epsilon}+1})\|<\epsilon$ for the first time.

\begin{lemma} \label{lem_num_successful_steps}
	Suppose Assumptions \ref{ass_cauchy_decrease}, \ref{ass_smoothness} and \ref{ass_bdd_hess} hold.
	Let $|\mathcal{S}_{i_{\epsilon}}|$ be the number of successful steps up to iteration $i_{\epsilon}$.
%that \algref{alg_dfogn} takes before $\|\grad f(\bx_k)\|<\epsilon$ for the first time.
	Then
	\be |\mathcal{S}_{i_{\epsilon}}| \leq \frac{f(\bx_0)}{\eta_1 c_1} \max\left(\kappa_H \epsilon_g^{-2}, \epsilon_g^{-1} \rho_{min}^{-1}\right), \label{eq_num_successful_steps} \ee
where $\epsilon_g$ is defined in \eqref{eq_normg_lower}, and $\rho_{min}$ in \eqref{eq_rho_min}.
\end{lemma}
\begin{proof}
	For all $k\in \mathcal{S}_{i_{\epsilon}}$, we have the sufficient decrease condition
	\be f(\bx_k) - f(\bx_{k+1}) \geq \eta_1\left(m_k(\b{0}) - m_k(\bs_k)\right) \geq \eta_1 c_1 \|\bg_k\| \min\left(\frac{\|\bg_k\|}{\kappa_H}, \Delta_k\right). \ee
	Since $\|\bg_k\|\geq\epsilon_g$ from \lemref{lem_true_model_gradient} and $\Delta_k\geq\rho_k\geq\rho_{min}$ from \lemref{lem_lower_delta_bound_dfo}, this means
	\be f(\bx_k) - f(\bx_{k+1}) \geq \eta_1 c_1 \epsilon_g \min\left(\frac{\epsilon_g}{\kappa_H}, \rho_{min}\right). \label{eq_succ_decrease} \ee
	Summing \eqref{eq_succ_decrease} over all $k\in \mathcal{S}_{i_{\epsilon}}$, and noting that $0\leq f(\bx_k)\leq f(\bx_0)$, we get
	\be f(\bx_0) \geq |\mathcal{S}_{i_{\epsilon}}| \eta_1 c_1 \epsilon_g \min\left(\frac{\epsilon_g}{\kappa_H}, \rho_{min}\right), \ee
	from which \eqref{eq_num_successful_steps} follows.
\end{proof}

We now need to count the number of iterations of \algref{alg_dfogn} which are not successful.
Following \cite{Garmanjani2016}, we count each iteration of the loop inside the criticality phase (\algref{alg_criticality}) as a separate iteration --- in effect, one `iteration' corresponds to one construction of the model $m_k$ \eqref{eq_gn_full_model_dfo}.
We also consider separately the number of criticality phases for which $\Delta_k$ is not reduced (i.e.~$\Delta_k=\Delta_k^{init}$). 
Counting until iteration $i_{\epsilon}$ (inclusive), we let
%Counting until $\|\grad f(\bx_k)\|<\epsilon$ for the first time, let:
\begin{itemize}
	\item $\mathcal{C}^M_{i_{\epsilon}}$ be the set of criticality phase iterations $k\leq i_{\epsilon}$ for which $\Delta_k$ is not reduced (i.e.~the first iteration of every call of \algref{alg_criticality} --- see \remref{rem_geom} for further details);
	\item $\mathcal{C}^U_{i_{\epsilon}}$ be the set of criticality phase iterations $k\leq i_{\epsilon}$  where $\Delta_k$ is reduced (i.e.~all iterations except the first for every call of \algref{alg_criticality});
	\item $\mathcal{F}_{i_{\epsilon}}$ be the set of iterations  where the safety phase is called;
	\item $\mathcal{M}_{i_{\epsilon}}$ be the set of iterations where the model-improving phase is called; and
	\item $\mathcal{U}_{i_{\epsilon}}$ be the set of unsuccessful iterations\footnote{Note that the analysis in \cite{Grapiglia2016} bounds the number of outer iterations of \algref{alg_dfogn}; i.e.~excluding $\mathcal{C}^M_{i_{\epsilon}}$ and $\mathcal{C}^U_{i_{\epsilon}}$.
Instead, they prove that while $\|\grad f(\bx_k)\|\geq\epsilon$, the criticality phase requires at most $|\log\epsilon|$ iterations.
Thus their bound on the number of objective evaluations is a factor $|\log\epsilon|$ larger than in \cite{Garmanjani2016} and than here.}.
\end{itemize}

\begin{lemma} \label{lem_num_other_steps}
	Suppose Assumptions \ref{ass_cauchy_decrease}, \ref{ass_smoothness} and \ref{ass_bdd_hess} hold.
	Then we have the bounds
	\begin{align}
		|\mathcal{C}^U_{i_{\epsilon}}| + |\mathcal{F}_{i_{\epsilon}}| + |\mathcal{U}_{i_{\epsilon}}| &\leq |\mathcal{S}_{i_{\epsilon}}|\cdot\frac{\log \overline{\gamma}_{inc}}{|\log \alpha_3|} + \frac{1}{|\log \alpha_3|}\log\left(\frac {\Delta_0^{init}}{\rho_{min}}\right), \label{eq_delta_reduction_count} \\
		|\mathcal{C}^M_{i_{\epsilon}}| &\leq |\mathcal{F}_{i_{\epsilon}}| + |\mathcal{S}_{i_{\epsilon}}| + |\mathcal{U}_{i_{\epsilon}}|, \label{eq_crit_count} \\
		|\mathcal{M}_{i_{\epsilon}}| &\leq |\mathcal{C}^M_{i_{\epsilon}}| + |\mathcal{C}^U_{i_{\epsilon}}| + |\mathcal{F}_{i_{\epsilon}}| + |\mathcal{S}_{i_{\epsilon}}| + |\mathcal{U}_{i_{\epsilon}}|, \label{eq_model_improving_count}
	\end{align}
	where $\alpha_3\defeq\max(\omega_C, \omega_S, \gamma_{dec}, \alpha_2)<1$ and $\rho_{min}$ is defined in \eqref{eq_rho_min}.
\end{lemma}
\begin{proof}
	On each  iteration  $k\in\mathcal{C}^U_{i_{\epsilon}}$, we reduce $\Delta_k$ by a factor of $\omega_C$.
	Similarly, on each iteration  $k\in\mathcal{F}_{i_{\epsilon}}$ we reduce $\Delta_k$ by a factor of at least $\max(\omega_S, \alpha_2)$, and for iterations in $\mathcal{U}_{i_{\epsilon}}$ by a factor of at least $\max(\gamma_{dec},\alpha_2)$.
	On each successful iteration, we increase $\Delta_k$ by a factor of at most $\overline{\gamma}_{inc}$, and on all other iterations, $\Delta_k$ is either constant or reduced.
	Therefore, we must have
	\begin{align}
		\rho_{min} &\leq \Delta_{i_{\epsilon}} \leq \Delta_0^{init} \cdot \omega_C^{|\mathcal{C}^U_{i_{\epsilon}}|} \cdot \max(\omega_S, \alpha_2)^{|\mathcal{F}_{i_{\epsilon}}|} \cdot \max(\gamma_{dec}, \alpha_2)^{|\mathcal{U}_{i_{\epsilon}}|} \cdot \overline{\gamma}_{inc}^{|\mathcal{S}_{i_{\epsilon}}|}, \\
		&\leq \Delta_0^{init} \cdot \alpha_3^{|\mathcal{C}^U_{i_{\epsilon}}| + |\mathcal{F}_{i_{\epsilon}}| + |\mathcal{U}_{i_{\epsilon}}|} \cdot \overline{\gamma}_{inc}^{|\mathcal{S}_{i_{\epsilon}}|},
	\end{align}
	from which \eqref{eq_delta_reduction_count} follows.
	
	After every call of the criticality phase, we have either a safety, successful or unsuccessful step, giving us \eqref{eq_crit_count}.
	Similarly, after every model-improving phase, the next iteration cannot call a subsequent model-improving phase, giving us \eqref{eq_model_improving_count}.
\end{proof}

\begin{assumption} \label{ass_epsc_bound}
	The algorithm parameter $\epsilon_C \geq c_3\epsilon$ for some constant $c_3>0$.
\end{assumption}
Note that Assumption \ref{ass_epsc_bound} can be easily satisfied by appropriate parameter choices in \algref{alg_dfogn}.

\begin{theorem} \label{thm_iters_bound}
	Suppose Assumptions \ref{ass_cauchy_decrease}, \ref{ass_smoothness}, \ref{ass_bdd_hess} and \ref{ass_epsc_bound} hold.
	Then the number of iterations  $i_{\epsilon}$ (i.e.~the number of times a model $m_k$ \eqref{eq_gn_full_model_dfo} is built) until  $\|\grad f(\bx_{i_{\epsilon}+1})\|< \epsilon$ is at most
	\begin{align}
		\left\lfloor\frac{4f(\bx_0)}{\eta_1 c_1}\left(1 + \frac{\log \overline{\gamma}_{inc}}{|\log \alpha_3|}\right)\max\left(\kappa_H c_4^{-2}\epsilon^{-2}, c_4^{-1}c_5^{-1}\epsilon^{-2}, c_4^{-1}(\Delta_0^{init})^{-1}\epsilon^{-1}\right)\right. \nonumber \\
%		\quad + \frac{4}{|\log \alpha_3|}|\log\epsilon| + \frac{4\log\left(\max\{1,c_5/\Delta_0^{init}\}\right)}{|\log \alpha_3|}, \label{eq_num_iters}
\left.\quad + \frac{4}{|\log \alpha_3|}\max\left(0,\log \left(\Delta_0^{init}c_5^{-1}\epsilon^{-1}\right)\right)\right\rfloor \label{eq_num_iters}
	\end{align}
% 	\be \left[\frac{4f(\bx_0)}{\eta_1 c_1}\left(1 + \frac{\log \overline{\gamma}_{inc}}{|\log \alpha_3|}\right)\max\left(\kappa_H c_4^{-2}, c_4^{-1}c_5^{-1}\right)\right]\epsilon^{-2} + \left[\frac{4}{|\log \alpha_3|}\right]|\log\epsilon| + \left[\frac{4\left|\log\left(c_5/\Delta_0^{init}\right)\right|}{|\log \alpha_3|}\right], \label{eq_num_iters} \ee
	where $c_4 \defeq \min\left(c_3, (1 + \kappa_{eg}\mu)^{-1}\right)$ and
	\be c_5 \defeq \min\left(\frac{\omega_C}{\kappa_{eg}+1/\mu}, \frac{\alpha_1 c_4}{\kappa_H}, \alpha_1\left(\kappa_{eg} + \frac{2\kappa_{ef}}{c_1(1-\eta_2)}\right)^{-1}\right). \ee
% 	\be c_4 \defeq \min\left(c_3, \frac{1}{1 + \kappa_{eg}\mu}\right) \quad \text{and} \quad c_5 \defeq \min\left(\frac{\omega_C}{\kappa_{eg}+1/\mu}, \frac{\alpha_1 c_4}{\kappa_H}, \alpha_1\left(\kappa_{eg} + \frac{2\kappa_{ef}}{c_1(1-\eta_2)}\right)^{-1}\right). \ee
\end{theorem}
\begin{proof}
	From \assref{ass_epsc_bound} and \lemref{lem_true_model_gradient}, we have $\epsilon_g = c_4\epsilon$.
	Similarly, from \lemref{lem_lower_delta_bound_dfo} we have $\rho_{min}=\min(\Delta_0^{init}, c_5\epsilon)$.
	Thus using \lemref{lem_num_other_steps}, we can bound the total number of iterations by
	\begin{align}
		&|\mathcal{C}^M_{i_{\epsilon}}| + |\mathcal{C}^U_{i_{\epsilon}}| + |\mathcal{F}_{i_{\epsilon}}| + |\mathcal{S}_{i_{\epsilon}}| + |\mathcal{M}_{i_{\epsilon}}| + |\mathcal{U}_{i_{\epsilon}}| \\
		&\qquad\qquad\qquad \leq 4|\mathcal{S}_{i_{\epsilon}}| + 4\left(|\mathcal{C}^U_{i_{\epsilon}}| + |\mathcal{F}_{i_{\epsilon}}| + |\mathcal{U}_{i_{\epsilon}}|\right), \\
		&\qquad\qquad\qquad \leq 4|\mathcal{S}_{i_{\epsilon}}|\left(1 + \frac{\log \overline{\gamma}_{inc}}{|\log \alpha_3|}\right) + \frac{4}{|\log \alpha_3|}\log\left(\frac {\Delta_0^{init}}{\rho_{min}}\right),
	\end{align}
	and so \eqref{eq_num_iters} follows from this and \lemref{lem_num_successful_steps}.
\end{proof}

We can summarize our results as follows:

\begin{corollary} \label{cor_complexity}
	Suppose Assumptions \ref{ass_cauchy_decrease}, \ref{ass_smoothness}, \ref{ass_bdd_hess} and \ref{ass_epsc_bound} hold.
	Then for $\epsilon\in(0,1]$, the number of iterations  $i_{\epsilon}$ (i.e.~the number of times a model $m_k$ \eqref{eq_gn_full_model_dfo} is built) 
until  $\|\grad f(\bx_{i_{\epsilon}+1})\|< \epsilon$
 is at most $\bigO(\kappa_H \kappa_d^2 \epsilon^{-2})$, and the number of objective evaluations until $i_{\epsilon}$ is at most $\bigO(\kappa_H \kappa_d^2 n \epsilon^{-2})$, where $\kappa_d\defeq\max(\kappa_{ef},\kappa_{eg})=\bigO(n L_J^2)$.
\end{corollary}
\begin{proof}
	From \thmref{thm_iters_bound}, we have $c_4^{-1}=\bigO(\kappa_{eg})$ and so 
	\be c_5^{-1}=\bigO(\max(\kappa_{eg}, \kappa_H c_4^{-1}, \kappa_{ef}+\kappa_{eg}))=\bigO(\kappa_H \kappa_d). \ee
	To leading order, the number of iterations is 
	\be \bigO(\max(\kappa_H c_4^{-2}, c_4^{-1}c_5^{-1})\epsilon^{-2})=\bigO(\kappa_H \kappa_d^2 \epsilon^{-2}), \ee
	as required.
	In every type of iteration, we change at most $n+1$ points, and so require no more than $n+1$ evaluations.
	The result $\kappa_d=\bigO(n L_J^2)$ follows from \lemref{lem_fully_linear}.
\end{proof}

\begin{remark}
	\thmref{thm_iters_bound} gives us a possible termination criterion for \algref{alg_dfogn} --- we loop until $k$ exceeds the value \eqref{eq_num_iters} or until $\rho_k \leq \rho_{min}$.
	However, this would require us to know problem constants $\kappa_{ef}$, $\kappa_{eg}$ and $\kappa_H$ in advance, which is not usually the case. Moreover, \eqref{eq_num_iters} is a worst-case
bound and so unduly pessimistic. 
%we expect the true gradient to become small much earlier.
% in the run of the algorithm.
\end{remark}

\begin{remark}
	In \cite{Garmanjani2016}, the authors propose a different criterion to test whether the criticality phase should be entered: $\|\bg_k^{init}\| \leq \Delta_k/\mu$ rather than $\|\bg_k^{init}\|\leq \epsilon_C$ as found here and in \cite{Conn2009}.
	We are able to use our criterion because of \assref{ass_epsc_bound}.
	If this did not hold, we would have $\epsilon_g \ll \epsilon$ and so $\rho_{min}\ll \epsilon$, which would worsen the result in \thmref{thm_iters_bound}.
	In practice, \assref{ass_epsc_bound} is reasonable, as we would not expect a user to prescribe a criticality tolerance much smaller than their desired solution tolerance.
\end{remark}

The standard complexity bound for first-order methods is $\bigO(\kappa_H \kappa_d^2 \epsilon^{-2})$ iterations and $\bigO(\kappa_H \kappa_d^2 n \epsilon^{-2})$ evaluations \cite{Garmanjani2016}, where $\kappa_d=\bigO(\sqrt{n})$ and $\kappa_H=1$.
\corref{cor_complexity} gives us the same count of iterations and evaluations, but the worse bounds $\kappa_d=\bigO(n)$ and $\kappa_H=\bigO(\kappa_d)$, coming from the least-squares structure (\lemref{lem_fully_linear}).

However, our model \eqref{eq_gn_full_model_dfo} is better than a simple linear model for $f$, as it captures some of the curvature information in the objective via the term $J_k^TJ_k$.
%This is reflected in \cite{Zhang2012}, where the algorithm from Zhang et al.~\cite{Zhang2010} is shown to have a quadratic asymptotic convergence rate for zero-residual problems, similar to the %classical Gauss-Newton method.
This means that DFO-GN produces models which are between fully linear and fully quadratic \cite[Definition 10.4]{Conn2009}, which is the requirement for convergence of second-order methods.
It therefore makes sense to also compare the complexity of DFO-GN with the complexity of second-order methods.

Unsurprisingly, the standard bound for second-order methods is worse in general, than for first-order methods, namely, $\bigO(\max(\kappa_H \kappa_d^2, \kappa_d^3) \epsilon^{-3})$ iterations and $\bigO(\max(\kappa_H \kappa_d^2, \kappa_d^3) n^2 \epsilon^{-3})$ evaluations \cite{Judice2015}, where $\kappa_d = \bigO(n)$, to achieve second-order criticality for the given objective.
Note that here $\kappa_d\defeq\max(\kappa_{ef}, \kappa_{eg}, \kappa_{eh})$ for fully quadratic models.
If $\|\grad^2 f\|$ is uniformly bounded, then we would expect $\kappa_H=\bigO(\kappa_{eh})=\bigO(\kappa_d)$. 

Thus DFO-GN has the iteration and evaluation complexity of a first-order method, but the problem constants (i.e.~dependency on $n$) of a second-order method.
That is, assuming $\kappa_H=\bigO(\kappa_d)$ (as suggested by \lemref{lem_fully_linear}), DFO-GN requires $\bigO(n^3 \epsilon^{-2})$ iterations and $\bigO(n^4 \epsilon^{-2})$ evaluations, compared to $\bigO(n\epsilon^{-2})$ iterations and $\bigO(n^2\epsilon^{-2})$ evaluations for a first-order method, and $\bigO(n^3\epsilon^{-3})$ iterations and $\bigO(n^5\epsilon^{-3})$ evaluations for a second-order method. 
%%This can be explained by the choice of model which includes via $J_k$, approximate curvature information. 

\begin{remark}
	%\alert{[New comment]}
	In \lemref{lem_fully_linear}, we used the result $C=\bigO(\Lambda)$ whenever $Y_k$ is $\Lambda$-poised, and wrote $\kappa_{eg}$ in terms of $C$; see Appendix A for details on the provenance of $C$ with respect to the interpolation system \eqref{eq_interp_conditions}.
	Our approach here matches the presentation of the first- and second-order complexity bounds from \cite{Garmanjani2016,Judice2015}.
	However, \cite[Theorem 3.14]{Conn2009} shows that $C$ may also depend on $n$.
%, which is not included in the $n$-dependence of the complexity bounds above.
	Including this dependence, we have $C=\bigO(\sqrt{n}\:\Lambda)$ for DFO-GN and general first-order methods, and $C=\bigO(n^2 \Lambda)$ for general second-order methods (where $C$  is now adapted for quadratic interpolation).
	This would yield the alternative bounds $\kappa_d=\bigO(n)$ for first-order methods, $\bigO(n^2)$ for DFO-GN and $\bigO(n^3)$ for second-order methods\footnote{For second-order methods, the fully quadratic bound is $\kappa_d=\bigO(nC)$.}.
	Either way, we conclude that the complexity of DFO-GN lies between first- and second-order methods.
\end{remark}

\subsubsection{Discussion of \assref{ass_bdd_hess}} \label{sec_bdd_hess}
It is also important to note that when $m_k$ is fully linear, we have an explicit bound $\|H_k\|\leq \t{\kappa}_H=\bigO(\kappa_d)$ from \lemref{lem_fully_linear}.
This means that \assref{ass_bdd_hess}, which typically necessary for first-order convergence (e.g.~\cite{Conn2009,Garmanjani2016}), is not required for \thmref{thm_liminf} and our complexity analysis.
To remove the assumption, we need to change \algref{alg_dfogn} in two places:
\begin{enumerate}
	\item Replace the test for entering the criticality phase with
	\be \min\left(\|\bg_k^{init}\|, \frac{\|\bg_k^{init}\|}{\max(\|H_k^{init}\|,1)}\right) \leq \epsilon_C; \qquad \text{and}\ee 
	\item Require the criticality phase to output $m_k$ fully linear and $\Delta_k$ satisfying 
	\be \Delta_k \leq \mu \min\left(\|\bg_k\|, \frac{\|\bg_k\|}{\max(\|H_k\|,1)}\right). \ee
\end{enumerate}
With these changes, the criticality phase still terminates, but instead of \eqref{eq_criticality_delta_bound} we have
\be \min\left(\Delta_k^{init}, \frac{\omega_C\epsilon}{\kappa_{eg}+1/\mu}, \frac{\omega_C\epsilon}{\kappa_{eg}+\t{\kappa}_H/\mu}\right) \leq \Delta_k \leq \Delta_k^{init}. \ee
We can also augment \lemref{lem_true_model_gradient} with the following, which can be used to arrive at a new value for $\rho_{min}$.

\begin{lemma}
	In all iterations, $\|\bg_k\|/\max(\|H_k\|,1) \geq \min(\epsilon_C, \Delta_k/\mu)$.
	If $\|\grad f(\bx_k)\|\geq\epsilon>0$ then 
	\be \frac{\|\bg_k\|}{\max(\|H_k\|,1)} \geq \epsilon_H \defeq \min\left(\epsilon_C, \frac{\epsilon}{(1+\kappa_{eg}\mu)\t{\kappa}_H}\right) > 0.\ee
\end{lemma}

Ultimately, we arrive at complexity bounds which match \corref{cor_complexity}, but replacing $\kappa_H$ with $\t{\kappa}_H$.
However, \assref{ass_bdd_hess} is still necessary for \thmref{thm_lim} to hold.

\section{Implementation} \label{sec_implementation}
In this section, we describe the key differences between \algref{alg_dfogn} and its software implementation DFO-GN.
These differences largely come from Powell's implementation of BOBYQA \cite{Powell2009} and are also features of DFBOLS, the implementation of the algorithm from Zhang et al.~\cite{Zhang2010}.
We also obtain a unified approach for analysing and improving the geometry of the interpolation set  due to our particular choice of local Gauss-Newton-like models.

\subsection{Geometry-Improving Phases}
In practice, DFO algorithms are generally not run to very high tolerance levels, and so the asymptotic behaviour of such algorithms is less important than for other optimization methods.
To this end, DFO-GN, like BOBYQA and DFBOLS, does not implement a criticality phase; but the safety step is implemented to encourage convergence.

In the geometry phases of the algorithm, we check the $\Lambda$-poisedness of $Y_k$ by calculating all the Lagrange polynomials for $Y_k$ (which are linear), then maximizing the absolute value of each in $B(\bx_k,\Delta_k)$.
To modify $Y_k$ to make it $\Lambda$-poised, we can repeat the following procedure \cite[Algorithm 6.3]{Conn2009}:
\begin{enumerate}
	\item Select the point $\by_t\in Y_k$ ($\by_t\neq\bx_k$) for which $\max_{\by\in B(\bx_k,\Delta_k)} |\Lambda_t(\by)|$ is maximized (c.f.~\eqref{eq_lambda_poised_definition});
	\item Replace $\by_t$ in $Y_k$ with $\by^+$, where 
	\be \by^+ = \argmax_{\by\in B(\bx_k,\Delta_k)} |\Lambda_t(\by)|, \label{eq_geom_improvement}\ee
\end{enumerate}
until $Y_k$ is $\Lambda$-poised in $B(\bx_k,\Delta_k)$.
This procedure terminates after at most $N$ iterations, where $N$ depends only on $\Lambda$ and $n$ \cite[Theorem 6.3]{Conn2009}, and in particular does not depend on $\bx_k$, $Y_k$ or $\Delta_k$.

In DFO-GN, we follow BOBYQA and replace these geometry-checking and improvement algorithms (which are called in the safety and model-improvement phases of \algref{alg_dfogn}) with simplified calculations.
Firstly, instead of checking for the $\Lambda$-poisedness of $Y_k$, we instead check if all interpolation points are within some distance of $\bx_k$, typically a multiple of $\Delta_k$.
If any point is sufficiently far from $\bx_k$, the geometry of $Y_k$ is improved by selecting the point $\by_t$ furthest from $\bx_k$, and moving it to $\by^+$ satisfying \eqref{eq_geom_improvement}.
That is, we effectively perform one iteration of the full geometry-improving procedure.

\subsection{Model Updating} \label{sec_model_updating}
In \algref{alg_dfogn}, we only update $Y_{k+1}$, and hence $\bem_k$ and $m_k$, on successful steps.
However, in our implementation, we always try to incorporate new information when it becomes available, and so we update $Y_{k+1}=Y_k\cup\{\bx_k+\bs_k\}\setminus\{\by_t\}$ on all iterations except when the safety phase is called (since in the safety phase we never evaluate $\br(\bx_k+\bs_k)$).

Regardless of how often we update the model, we need some criterion for selecting the point $\by_t\in Y_k$ to replace with $\by^+\defeq \bx_k+\bs_k$.
There are three common reasons for choosing a particular point to remove from the interpolation set:
\begin{description}
	\item[Furthest Point:] It is the furthest away from $\bx_k$ (or $\bx_{k+1}$);
	\item[Optimal $\Lambda$-poisedness:] Replacing it with $\by^+$ would give the maximum improvement in the $\Lambda$-poisedness of $Y_k$.
	That is, choose the $t$ for which $|\Lambda_t(\by^+)|$ is maximized;
	\item[Stable Update:] Replacing it with $\bx_k+\bs_k$ would induce the most stable update to the interpolation system \eqref{eq_linear_interp_system_Jonly}. 
	As introduced by Powell \cite{Powell2004a} for quadratic models, moving $\by_t$ to $\by^+$ induces a low-rank update of the matrix $W \to W_{new}$ in the interpolation system, here \eqref{eq_linear_interp_system_Jonly}.
	From the Sherman-Morrison-Woodbury formula, this induces a low-rank update of $H=W^{-1}$, which has the form
	\be H_{new} \gets H + \frac{1}{\sigma_t}\left[A_t B_t^{\top}\right], \label{eq_inv_update} \ee
	for some $\sigma_t\neq 0$ and low rank $A_t B_t^{\top}$.
	Under this measure, we would want to replace a point in the interpolation set when the resulting $|\sigma_t|$ is maximal; i.e.~the update \eqref{eq_inv_update} is `stable'.
	In \cite{Powell2004a}, it is shown that for underdetermined quadratic interpolation, $\sigma_t\geq \Lambda_t(\by^+)^2$.
\end{description}
Two approaches for selecting $\by_t$ combine two of these reasons into a single criterion.
Firstly in BOBYQA, the point $t$ is chosen by combining the `furthest point' and `stable update' measures:
\be t = \argmax_{j=0,\ldots,n} \left\{|\sigma_j| \max\left(\frac{\|\by_j-\bx_k\|^4}{\Delta_k^4}, 1\right)\right\}. \label{eq_pt_update_bobyqa} \ee
Alternatively, Scheinberg and Toint \cite{Scheinberg2010} combine the `furthest point' and `optimal $\Lambda$-poisedness' measures:
\be t = \argmax_{j=0,\ldots,n} \left\{|\Lambda_j(\by^+)| \: \|\by_j-\bx_k\|^2\right\}. \ee
In DFO-GN, we use the BOBYQA criterion \eqref{eq_pt_update_bobyqa}.
However, as we now show, in DFO-GN, the two measures `optimal $\Lambda$-poisedness' and `stable update' coincide, meaning our framework allows a unification of the perspectives from \cite{Powell2009} and \cite{Scheinberg2010}, rather than having the indirect relationship via the bound $\sigma_t \geq \Lambda_t(\by^+)^2$.

To this end, define $W$ as the matrix in \eqref{eq_linear_interp_system_Jonly}, and let $H \defeq W^{-1}$.
The Lagrange polynomials for $Y_k$ can then be found by applying the interpolation conditions \eqref{eq_lagrange_defn}.
That is, we have
\be \Lambda_t(\by) = 1 + \bg_t^{\top}(\by-\by_t), \ee
where $\bg_t$ solves
\be W\bg_t = \begin{bmatrix} \Lambda_t(\by_1) - \Lambda_t(\bx_k) \\ \vdots \\ \Lambda_t(\by_n) - \Lambda_t(\bx_k) \end{bmatrix} = \begin{cases}\bee_t, & \text{if $\by_t\neq\bx_k$,} \\ -\bee, & \text{if $\by_t=\bx_k$,} \end{cases} \ee
where $\bee_t$ is the usual coordinate vector in $\R^n$ and $\bee \defeq [1 \: \cdots \: 1]^{\top}\in\R^n$.
This gives us the relations
\be \Lambda_t(\by^+) = \begin{cases}1 + (H\bee_t)^{\top}(\by^+-\by_t), & \text{if $\by_t\neq\bx_k$,} \\ 1 -(H\bee)^{\top}(\by^+-\bx_k), & \text{if $\by_t=\bx_k$.} \end{cases} \ee

Now, we consider the `stable update' measure.
We will update the point $\by_t$ to $\by^+$, which will give us a new matrix $W_{new}$ with inverse $H_{new}$.
This change induces a rank-1 update from $W$ to $W_{new}$, given by
\be W_{new} = W + \begin{cases}\bee_t(\by^+-\by_t)^{\top}, & \text{if $\by_t\neq\bx_k$,} \\ \bee(\bx_k-\by^+)^{\top}, & \text{if $\by_t=\bx_k$.} \end{cases} \ee
By the Sherman-Morrison formula, this induces a rank-1 update from $H$ to $H_{new}$, given by
\be H_{new} = H - \frac{1}{\sigma_t}\begin{cases}H\bee_t(\by^+-\by_t)^{\top}H, & \text{if $\by_t\neq\bx_k$,} \\ H\bee(\bx_k-\by^+)^{\top}H, & \text{if $\by_t=\bx_k$.} \end{cases} \ee
For a general rank-1 update $W_{new}=W+\b{u}\b{v}^{\top}$, the denominator is $\sigma=1+\b{v}^{\top}W^{-1}\b{u}$, and so here we have
\be \sigma_t = \begin{cases} 1 + (\by^+-\by_t)^{\top}H\bee_t, & \text{if $\by_t\neq\bx_k$,} \\ 1 + (\bx_k-\by^+)^{\top}H\bee, & \text{if $\by_t=\bx_k$,} \end{cases} \ee
and hence $\sigma_t=\Lambda_t(\by^+)$, as expected.

\subsection{Termination Criteria}
The specification in \algref{alg_dfogn} does not include any termination criteria.
In the implementation of DFO-GN, we use the same termination criteria as DFBOLS \cite{Zhang2010}, namely terminating whenever any of the following are satisfied:
\begin{itemize}
	\item Small objective value: since $f\geq 0$ for least-squares problems, we terminate when 
	\be f(\bx_k) \leq \max\{10^{-12}, 10^{-20}f(\bx_0)\}.\ee
	For nonzero residual problems (i.e.~where $f(\bx^*)>0$ at the true minimum $\bx^*$), it is unlikely that termination will occur by this criterion;
	\item Small trust region: $\rho_k$, which converges to zero as $k\to\infty$ from \lemref{lem_delta_conv_dfo}, falls below a user-specified threshold; and
	\item Computational budget: a (user-specified) maximum number of evaluations of $\br$ is reached.
\end{itemize}

\subsection{Other Implementation Differences} \label{sec_other_implementation_diffs}

\paragraph{Addition of bound constraints} This is allowed in the implementation of DFO-GN as it is important to practical applications.
%Perhaps the most important difference between the implementation of DFO-GN and \algref{alg_dfogn} is that the implementation also adds bound constraints.
That is, we solve \eqref{eq_ls_definition} subject to $\b{a} \leq \bx \leq \b{b}$ for given bounds $\b{a},\b{b}\in\R^n$.
This requires no change to the logic as specified in \algref{alg_dfogn}, but does require the addition of the same bound constraints in the algorithms for the trust region subproblem \eqref{eq_tr_subproblem} and calculating geometry-improving steps \eqref{eq_geom_improvement}.
For the trust-region subproblem, we use the routine from DFBOLS, which is itself a slight modification of the routine from BOBYQA (which was specifically designed to produce feasible iterates in the presence of bound constraints).
Calculating geometry-improving steps \eqref{eq_geom_improvement} is easier, since the Lagrange polynomials are linear rather than quadratic.
That is, we need to maximize a linear objective subject to Euclidean ball and bound constraints
In this case, we use our own routine, given in \appref{sec_geom_step_algo}, which handles the bound constraints via an active set method.

\paragraph{Internal representation of interpolation points} In the interpolation system \eqref{eq_linear_interp_system_Jonly}, we are required to calculate the vectors $\by_t-\bx_k$.
As the algorithm progresses, we expect $\|\by_t-\bx_k\| = \bigO(\Delta_k) \to 0$ as $k\to\infty$, and hence the evaluation $\by_t-\bx_k$ to be sensitive to rounding errors, especially if $\|\bx_k\|$ is large.
To reduce the impact of this, internally the points are stored with respect to a `base point' $\bx_b$, which is periodically updated to remain `close' to the points in $Y_k$.
That is, we actually maintain the set $Y_k = \bx_b + \{(\by_0-\bx_b), \ldots, (\by_n-\bx_b)\}$.
This means that in \eqref{eq_linear_interp_system_Jonly} we are computing $\|(\by_t-\bx_b) - (\bx_k-\bx_b)\|$, where $\|\by_t-\bx_b\|$ and $\|\bx_k-\bx_b\|$ are much smaller than $\|\by_t\|$ and $\|\bx_k\|$, so the calculation is less sensitive to cancellation.

Once we have determined $J_k$ by solving \eqref{eq_linear_interp_system_Jonly}, the model $\bem_k$ \eqref{eq_ls_definition} is internally represented as
\be \br(\by) \approx \t{\bem}_k(\by-\bx_b) \defeq \b{c}_k + J_k(\by-\b{x}_b), \ee
where $\b{c}_k = \b{r}(\b{x}_k) - J_k(\b{x}_k-\b{x}_b)$, to ensure that $\t{\bem}_k(\bx_k-\bx_b)=\br(\bx_k)$.
Note that we take the argument of $\t{\bem}_k$ to be $\by-\bx_b$, rather than $\by-\bx_k$ as in \eqref{eq_ls_definition}.

When we update $\bx_b$, say to $\bx_b+\Delta\b{b}$, we rewrite $\t{\bem}_k$ as
\be \t{\bem}_k(\by-\bx_b) = [\b{c}_k + J_k\Delta\b{b}] + J_k(\by-\bx_b-\Delta\b{b}), \ee
and so we update $\bx_b \gets \bx_b+\Delta\b{b}$ and $\b{c}_k\gets \b{c}_k+J_k\Delta\b{b}$.
Lastly, the interpolation points have their representation changed to $Y_k \gets (\bx_b+\Delta\b{b}) + \{(\by_0-\bx_b-\Delta\b{b}),\ldots, (\by_n-\bx_b-\Delta\b{b})\}$.

\paragraph{Other differences}
The following changes, which are from BOBYQA, are present in the implementation of DFO-GN:
\begin{itemize}
	\item We accept any step (i.e.~set $\bx_{k+1}=\bx_k+\bs_k$) where we see an objective reduction --- that is, when $r_k>0$. In fact, we always update $\bx_k$ to be the best value found so far, even if that point came from a geometry-improving phase rather than a trust region step;
	\item The reduction of $\rho_k$ in an unsuccessful step (line \ref{ln_rho_redn}) only occurs when $r_k<0$;
	\item Since we update the model on every iteration, we only reduce $\rho_k$ after 3 consecutive unsuccessful iterations; i.e.~we only reduce $\rho_k$ when $\Delta_k$ is small and after the model has been updated several times (reducing the likelihood of the unsuccessful steps being from a bad interpolating set);
	\item The method for reducing $\rho_k$ is usually given by $\rho_{k+1}=\alpha_1\rho_k$, but it changed when $\rho_k$ approaches $\rho_{end}$:
	\be \rho_{k+1} = \begin{cases} \alpha_1 \rho_k, & \text{if $\rho_k > 250\rho_{end}$,} \\ \sqrt{\rho_k \rho_{end}}, & \text{if $16\rho_{end} < \rho_k \leq 250 \rho_{end}$, } \\ \rho_{end}, & \text{if $\rho_k \leq 16 \rho_{end}$.} \end{cases} \ee
	\item In some calls of the safety phase, we only reduce $\rho_k$ and $\Delta_k$, without improving the geometry of $Y_k$.
\end{itemize}

\subsection{Comparison to DFBOLS}
As has been discussed at length, there are many similarities between DFO-GN and DFBOLS from Zhang et al.~\cite{Zhang2010}.
Although the algorithm described in \cite{Zhang2010} allows Gauss-Newton type models in principle, in practice DFO-GN is simpler in several respects:
\begin{itemize}
	\item The use of linear models for each residual \eqref{eq_ls_definition} means we require only $n+1$ interpolation points. In DFBOLS, quadratic models for each $r_i$ are used, which requires between $n+2$ and $(n+1)(n+2)/2$ points, where the remaining degrees of freedom are taken up by minimizing the change in the model Hessian. This results in both a more complicated interpolation problem compared to our system \eqref{eq_linear_interp_system_Jonly}, and a larger startup cost (where an initial $Y_0$ of the correct size is constructed, and $\br$ evaluated at each of these points);
	\item As a result of using linear models, there is no ambiguity in how to construct the full model $m_k$ \eqref{eq_gn_full_model_dfo}. In DFBOLS, simply taking a sum of squares of each residual's model gives a quartic.
	The authors drop the cubic and quartic terms, and choose the quadratic term from one of three possibilities, depending on the sizes of $\|\bg_k\|$ and $f(\bx_k)$.
	This requires the introduction of three new algorithm parameters, each of which may require calibration.
	\item DFO-GN's method for choosing a point to replace when doing model updating, as discussed in \secref{sec_model_updating}, yields a unification of the geometric (`optimal $\Lambda$-poisedness') and algebraic (`stable update') perspectives on this update.
	In DFBOLS, the connection exists but is less direct, as it uses the same method as BOBYQA \eqref{eq_pt_update_bobyqa} with $\sigma_t \geq \Lambda_t(\by^+)^2$.
	As discussed in \cite{Powell2009}, this bound may sometimes be violated as a result of rounding errors, and thus requires an extra geometry-improving routine to `rescue' the algorithm from this problem.
	DFO-GN does not need or have this routine.
\end{itemize}
The first of these points also applies to Wild's POUNDERS \cite{Wild2017}, which builds quadratic models for each residual, and constructs a model for the full objective which is equivalent to a full Newton model (i.e.~taking all available second-order information).

It is also important to note that neither \cite{Zhang2010} nor \cite{Wild2017} test the use of DFO-GN type linear models for each residual in practice.

\section{Numerical Results} \label{sec_numerics}
%So far, we have introduced DFO-GN, and demonstrated how the algorithm is largely a simplification of the method from \cite{Zhang2010}.
%In this chapter, we compare the practical performance of DFO-GN to two versions of DFBOLS \cite{Zhang2010}, and show that our simplification leads to comparable performance.
Now we compare the practical performance of DFO-GN to two versions of DFBOLS \cite{Zhang2010}, and show that DFO-GN has comparable budget performance and significantly faster runtime.
The two versions of DFBOLS are the original Fortran implementation from \cite{Zhang2010}, and the other is our own Python implementation (which we will call `Py-DFBOLS'), designed to be as similar as possible to the implementation of DFO-GN.
Both DFO-GN and Py-DFBOLS use Python 3.5.2\footnote{With NumPy 1.12.1. Linear solves use LAPACK's LU decomposition routines, wrapped by SciPy 0.19.0.}.
We also compare with the general-objective solver BOBYQA \cite{Powell2009} and POUNDERS \cite{Wild2017}, another least-squares DFO code which uses quadratic interpolation models for each residual.

The parameter values used for DFO-GN are: $\Delta_{max}=10^{10}$, $\gamma_{dec}=0.5$, $\gamma_{inc}=2$, $\overline{\gamma}_{inc}=4$, $\eta_1=0.1$, $\eta_2=0.7$, $\alpha_1=0.1$, $\alpha_2=0.5$, $\omega_S=0.1$ and $\gamma_S=0.5$.
For all solvers, we use an initial trust region radius of $\rho_0=\Delta_0=0.1\max(\|\bx\|_{\infty},1)$ and final trust region radius $\rho_{end}=10^{-10}$ where possible\footnote{POUNDERS does not have this as a user input. Instead we set all gradient tolerances to zero.}, to avoid this being the termination condition as often as possible.

We tested BOBYQA and (Py-)DFBOLS with $n+2$, $2n+1$ and $(n+1)(n+2)/2$ interpolation points.
In the results which follow, for simplicity we show the $n+2$ and $2n+1$ cases for DFBOLS and the $(n+1)(n+2)/2$ case for Py-DFBOLS.
These were chosen because Py-DFBOLS performs very similarly to DFBOLS in the case of $n+2$ and $2n+1$ points, and outperforms DFBOLS in the $(n+1)(n+2)/2$ case.
Similarly, we show the best-performing $2n+1$ and $(n+1)(n+2)/2$ cases for BOBYQA.

\subsection{Test Problems and Methodology}
We tested the solvers on the test suite from Mor\'e and Wild \cite{More2009}, a collection of 53 unconstrained nonlinear least-squares problems with dimension $2\leq n\leq 12$ and $2\leq m \leq 65$.
For each problem, we optionally allowed evaluations of the residuals $r_i$ to have stochastic noise.
Specifically, we allowed the following noise models:
\begin{itemize}
	\item Smooth (noiseless) function evaluations;
	\item Multiplicative unbiased Gaussian noise: we evaluate $\t{r}_i(\bx)=r_i(\bx)(1+\epsilon)$, where $\epsilon\sim N(0,\sigma^2)$ i.i.d.~for each $i$ and $\bx$;
	\item Additive unbiased Gaussian noise: we evaluate $\t{r}_i(\bx)=r_i(\bx)+\epsilon$, where $\epsilon\sim N(0,\sigma^2)$ i.i.d.~for each $i$ and $\bx$; and
	\item Additive $\chi^2$ noise: we evaluate $\t{r}_i(\bx)=\sqrt{r_i(\bx)^2+\epsilon^2}$, where $\epsilon\sim N(0,\sigma^2)$ i.i.d.~for each $i$ and $\bx$.
\end{itemize}

To compare solvers, we use data and performance profiles \cite{More2009}.
First, for each solver $\mathcal{S}$, each problem $p$ and for an accuracy level $\tau\in(0,1)$, we determine the number of function evaluations $N_p(\mathcal{S};\tau)$ required for a problem to be `solved':
\be N_p(\mathcal{S}; \tau) \defeq \text{\# objective evals required to get $f(\b{x}_k) \leq \mathbb{E}[f^* + \tau(f(\b{x}_0) - f^*)]$,} \label{eq_solved_threshold} \ee
where $f^*$ is an estimate of the true minimum\footnote{Note that in \cite{More2009}, and subsequent other papers such as \cite{Zhang2010}, the value $f^*$ is usually taken to be the smallest objective value achieved by any of the solvers under consideration within a fixed budget. The main motivation in \cite{More2009} for this choice is for when $f$ is expensive, and so we have small computational budgets and it is possible that no solver converges. In our setting, this is not the case, so we use our (stronger) choice of $f^*$, which comes from \cite{More1981} or the results of running these and other (derivative-based) solvers.} $f(\b{x}^*)$.
A full list of the values used is provided in \appref{sec_mw_problems}. %\tabref{tab_more_wild_values}.
We define $N_p(\mathcal{S}; \tau)=\infty$ if this was not achieved in the maximum computational budget allowed.

We can then compare solvers by looking at the proportion of test problems solved for a given computational budget.
For \emph{data profiles}, we normalize the computational effort by problem dimension, and plot (for solver $\mathcal{S}$, accuracy level $\tau\in(0,1)$ and problem suite $\mathcal{P}$)
\be d_{\mathcal{S}, \tau}(\alpha) \defeq \frac{|\{p\in\mathcal{P} : N_p(\mathcal{S};\tau) \leq \alpha(n_p+1)\}|}{|\mathcal{P}|}, \qquad \text{for $\alpha\in[0,N_g]$,} \ee
where $N_g$ is the maximum computational budget, measured in simplex gradients (i.e.~$N_g(n_p+1)$ objective evaluations are allowed for problem $p$).

For \emph{performance profiles}, we normalize the computational effort by the minimum effort needed by any solver (i.e.~by problem difficulty).
That is, we plot
\be \pi_{\mathcal{S},\tau}(\alpha) \defeq \frac{|\{p\in\mathcal{P} : N_p(\mathcal{S};\tau) \leq \alpha N_p^*(\tau)\}|}{|\mathcal{P}|}, \qquad \text{for $\alpha\geq 1$,} \ee
where $N_p^*(\tau) \defeq \min_{\mathcal{S}} N_p(\mathcal{S};\tau)$ is the minimum budget required by any solver.

For test runs where we added stochastic noise, we took average data and performance profiles over multiple runs of each solver; that is, for each $\alpha$ we take an average of $d_{\mathcal{S},\tau}(\alpha)$ and $\pi_{\mathcal{S},\tau}(\alpha)$.
When plotting performance profiles, we took $N_p^*(\tau)$ to be the minimum budget required by any solver in any run.

\subsection{Test Results} \label{sec_results}
For our testing, we used a budget of $N_g=200$ gradients (i.e.~$200(n+1)$ objective evaluations) for each problem, noise level $\sigma=10^{-2}$, and took 10 runs of each solver\footnote{Scheduled using \cite{Tange2011}.}.
Most results use an accuracy level of $\tau=10^{-5}$ in \eqref{eq_solved_threshold}.

\begin{figure}
	\centering
	\begin{subfigure}[b]{0.48\textwidth}
		\includegraphics[width=\textwidth]{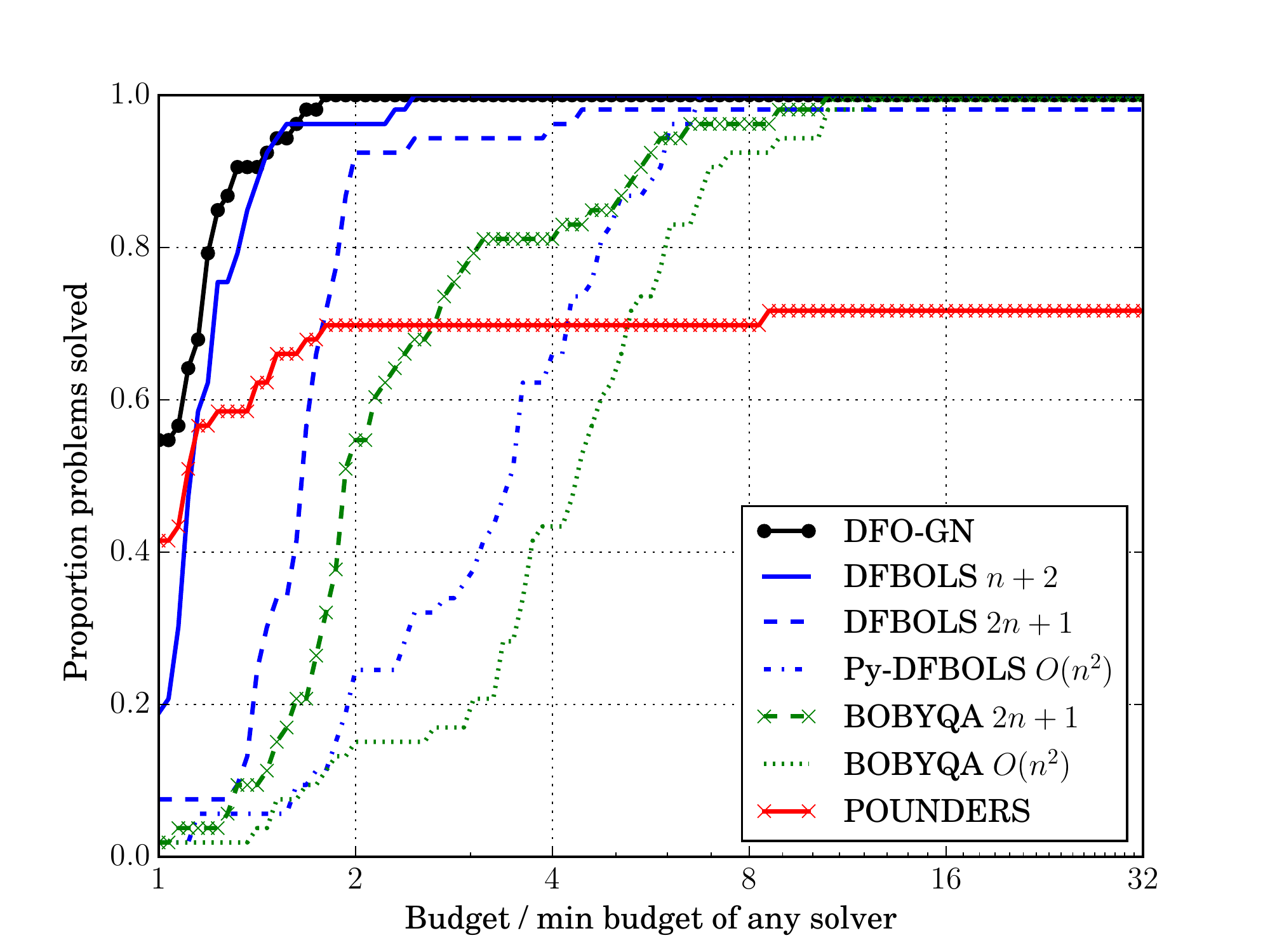}
		\caption{Smooth objective}
		\label{fig_tau1_smooth}
	\end{subfigure}
	~
	\begin{subfigure}[b]{0.48\textwidth}
		\includegraphics[width=\textwidth]{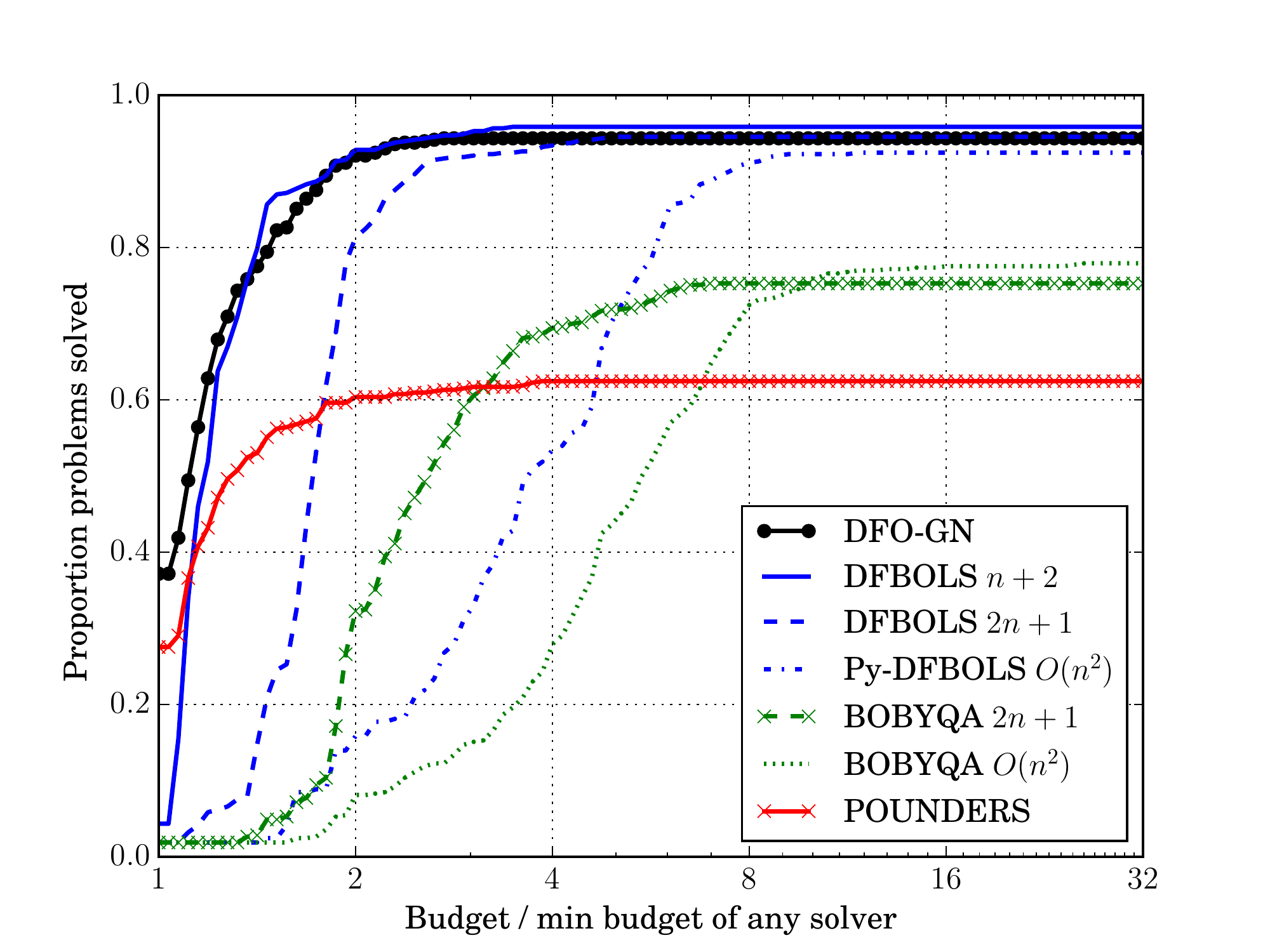}
		\caption{Mult.~Gaussian noise $\sigma=10^{-2}$}
		\label{fig_tau1_ubgsn}
	\end{subfigure}
	\caption{Performance profile comparison of DFO-GN with BOBYQA, DFBOLS and POUNDERS for low accuracy $\tau=10^{-1}$. For the BOBYQA and DFBOLS runs, $n+2$, $2n+1$ and $\bigO(n^2)=(n+1)(n+2)/2$ are the number of interpolation points. For \subfigref{fig_tau1_ubgsn}, results shown are an average of 10 runs for each solver.}
	\label{fig_tau1}
\end{figure}

Firstly, \figref{fig_tau1} shows two performance profiles under the low accuracy requirement $\tau=10^{-1}$. 
Here we see an important benefit of DFO-GN compared to BOBYQA, DFBOLS and POUNDERS --- the smaller interpolation set means that it can begin the main iteration and make progress sooner.
This is reflected in \figref{fig_tau1}, where DFO-GN is the fastest solver more frequently than any other, both with smooth and noisy objective evaluations.
We also note that POUNDERS is the faster solver more frequently than DFBOLS at this accuracy level.
However, after the full budget of 200 gradients, POUNDERS is unable to solve a high proportion of problems to this accuracy --- this holds both here and for the remainder of the results.
We believe this is because it uses a built-in termination condition for sufficiently small trust region radius, which is set too high to achieve \eqref{eq_solved_threshold} for many problems.
In line with the results from \cite{Zhang2010}, BOBYQA does not perform as well as DFBOLS or DFO-GN, as it does not exploit the least-squares problem structure.

Next, \figref{fig_smooth} shows results for accuracy $\tau=10^{-5}$ and smooth objective evaluations.
% The most important point to note is that our simplification from quadratic to linear residual models has not led to a loss of performance at obtaining high accuracy solutions.
It is important to note is that our simplification from quadratic to linear residual models has not led to a loss of performance at obtaining high accuracy solutions, and produces essentially identical long-budget performance.
At this level, the advantage from the smaller startup cost is no longer seen, but particularly in the performance profiles, we can still see the substantially higher startup cost of using $(n+1)(n+2)/2$ interpolation points.

\begin{figure}
	\centering
	\begin{subfigure}[b]{0.48\textwidth}
		\includegraphics[width=\textwidth]{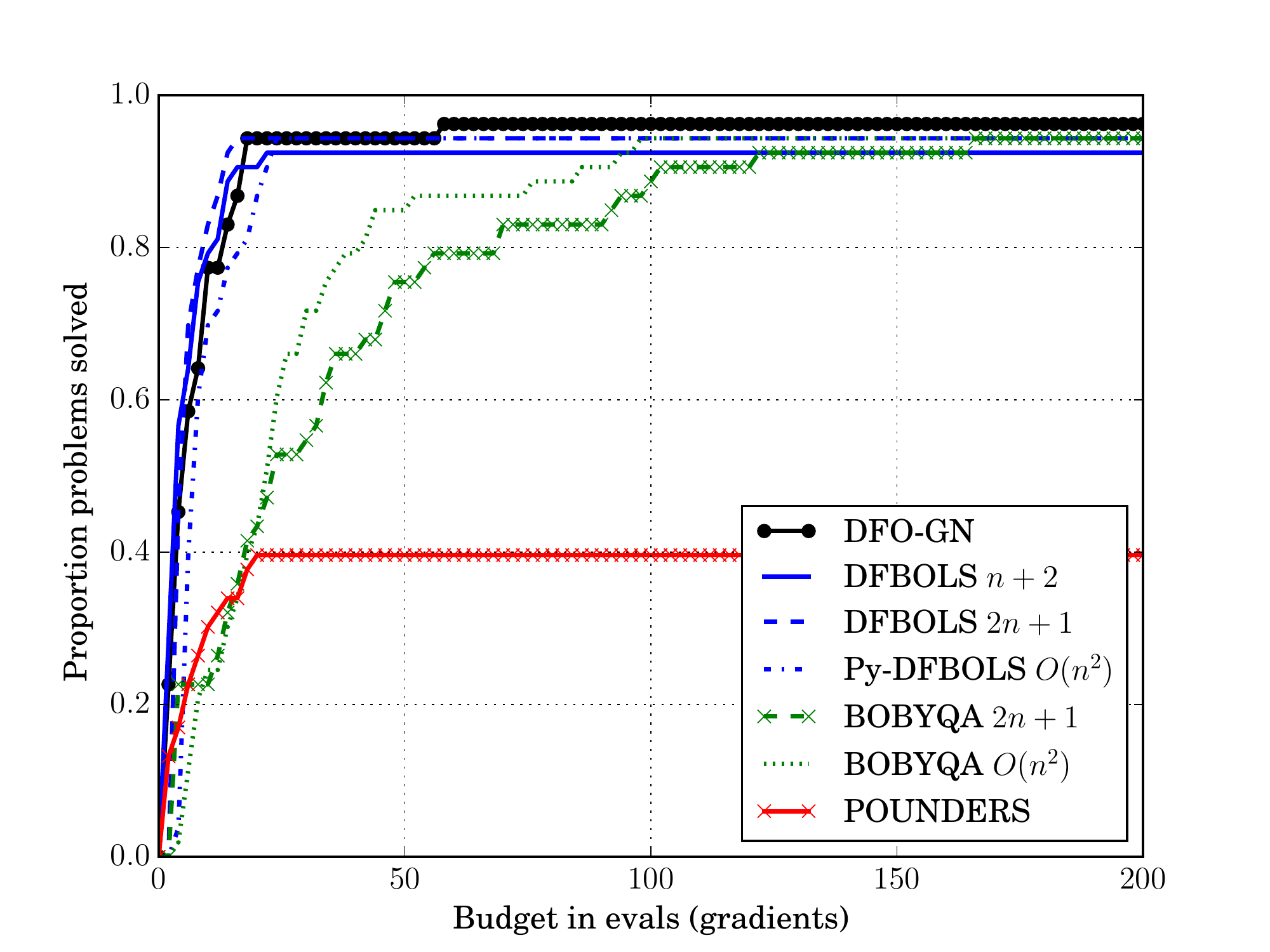}
		\caption{Data Profile}
		\label{fig_smooth_data}
	\end{subfigure}
	~
	\begin{subfigure}[b]{0.48\textwidth}
		\includegraphics[width=\textwidth]{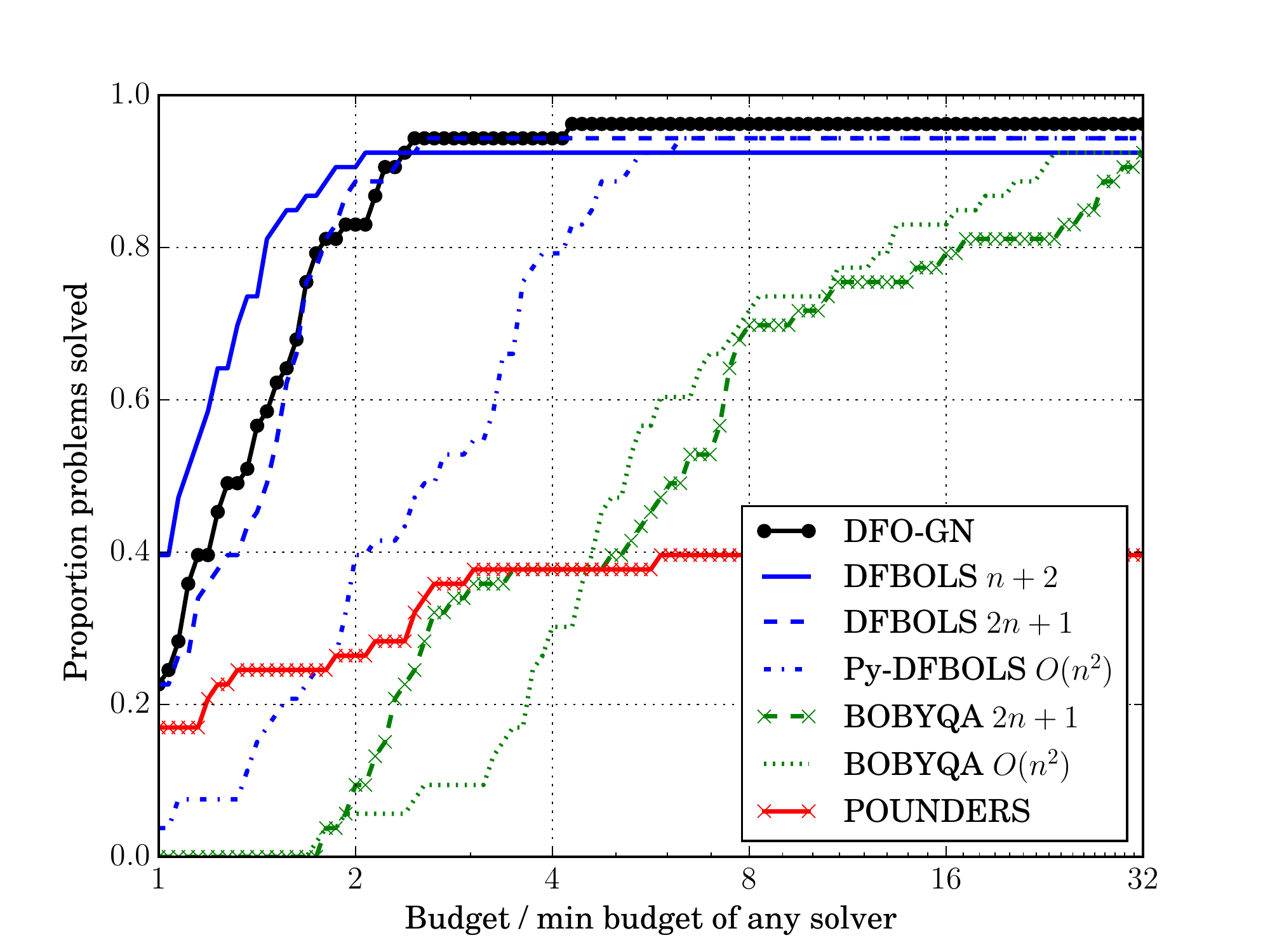}
		\caption{Performance Profile}
		\label{fig_smooth_perf}
	\end{subfigure}
	\caption{Comparison of DFO-GN with BOBYQA, DFBOLS and POUNDERS for smooth objectives, to accuracy $\tau=10^{-5}$. For the BOBYQA and DFBOLS runs, $n+2$, $2n+1$ and $\bigO(n^2)=(n+1)(n+2)/2$ are the number of interpolation points.}
	\label{fig_smooth}
\end{figure}

Similarly, \figref{fig_noisy} shows the same plots but for noisy problems (multiplicative Gaussian, additive Gaussian and additive $\chi^2$ respectively).
Here, DFO-GN suffers a small performance penalty (of approximately 5-10\%) compared to DFBOLS, particularly when using $2n+1$ and $(n+1)(n+2)/2$ interpolation points, suggesting that the extra curvature and evaluation information in DFBOLS has some benefit for noisy problems.
Also, the performance penalty is larger in the case of additive noise than multiplicative (approximately 10\% vs.~5\%).
Note that additive noise makes all our test problems nonzero residual (i.e.~$f(\bx^*)>0$ for the true minimum $\bx^*$).
Thus the worse performance of DFO-GN compared to DFBOLS for additive noise is similar to the derivative-based case, where for nonzero residual problems the Gauss-Newton method has a lower asymptotic convergence rate than Newton's method \cite{Nocedal2006}.

Also of note is that although BOBYQA suffers a substantial performance penalty when moving from smooth to noisy problems, this penalty (compared to DFO-GN and DFBOLS) is much less for additive $\chi^2$ noise.
This is likely because this noise model makes each residual function nonsmooth by taking square roots, but the change to the full objective is relatively benign --- simply adding $\chi^2$ random variables.

\begin{figure}
	\centering
	\begin{subfigure}[b]{0.48\textwidth}
		\includegraphics[width=\textwidth]{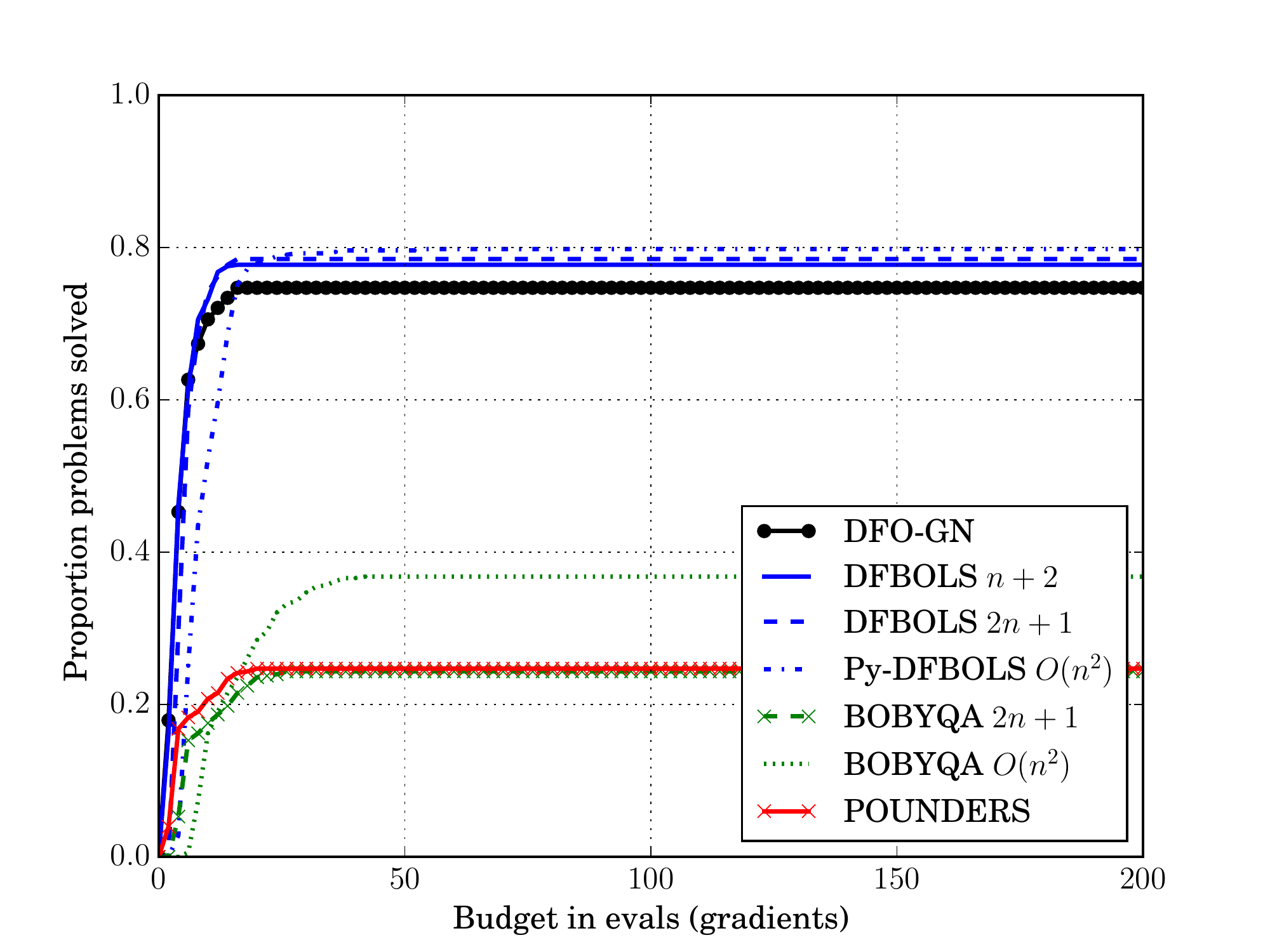}
		\caption{Mult.~Gaussian, data profile}
		\label{fig_ubgsn_data}
	\end{subfigure}
	~
	\begin{subfigure}[b]{0.48\textwidth}
		\includegraphics[width=\textwidth]{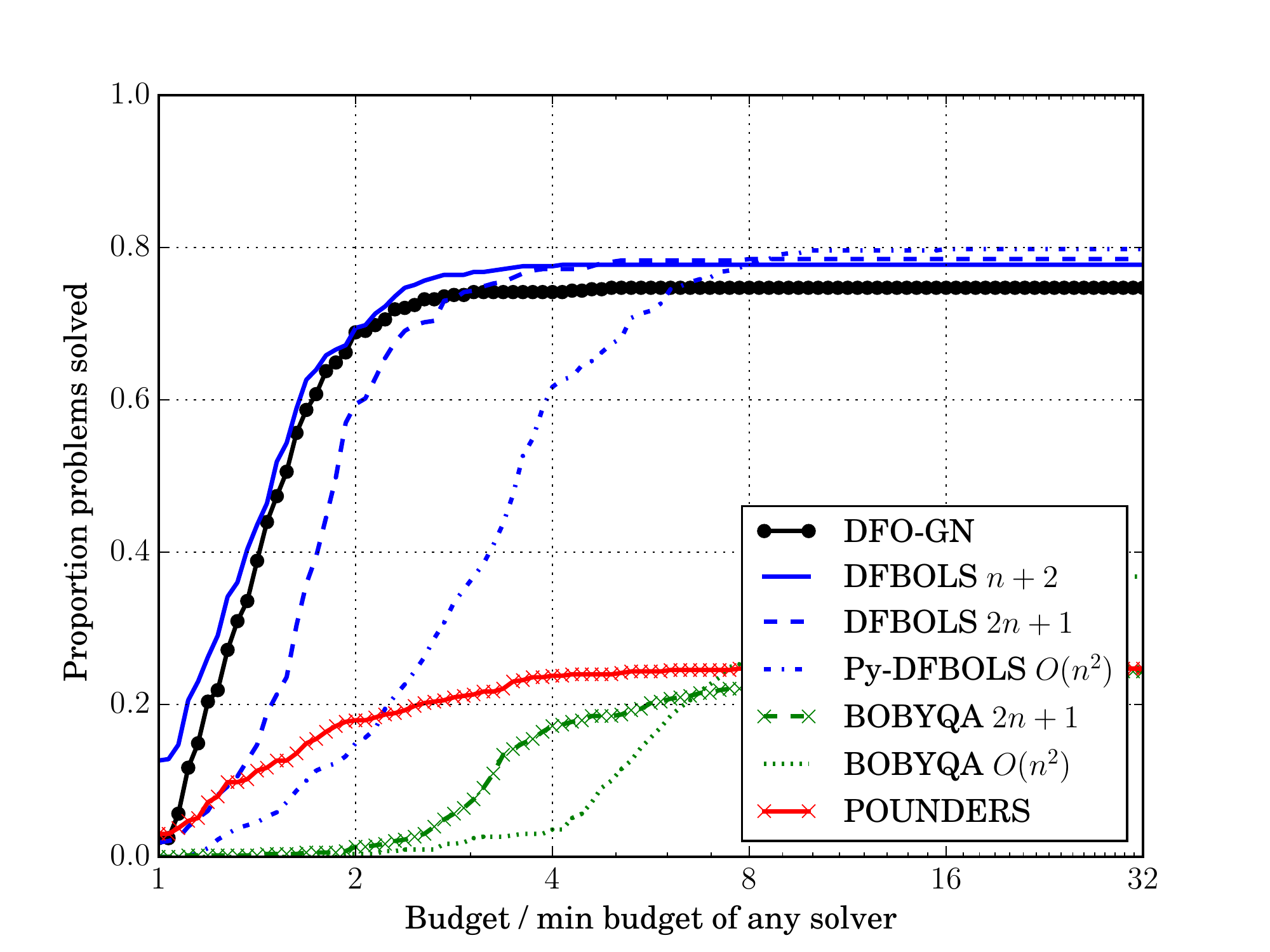}
		\caption{Mult.~Gaussian, performance profile}
		\label{fig_ubgsn_perf}
	\end{subfigure}
	\\
	\begin{subfigure}[b]{0.48\textwidth}
		\includegraphics[width=\textwidth]{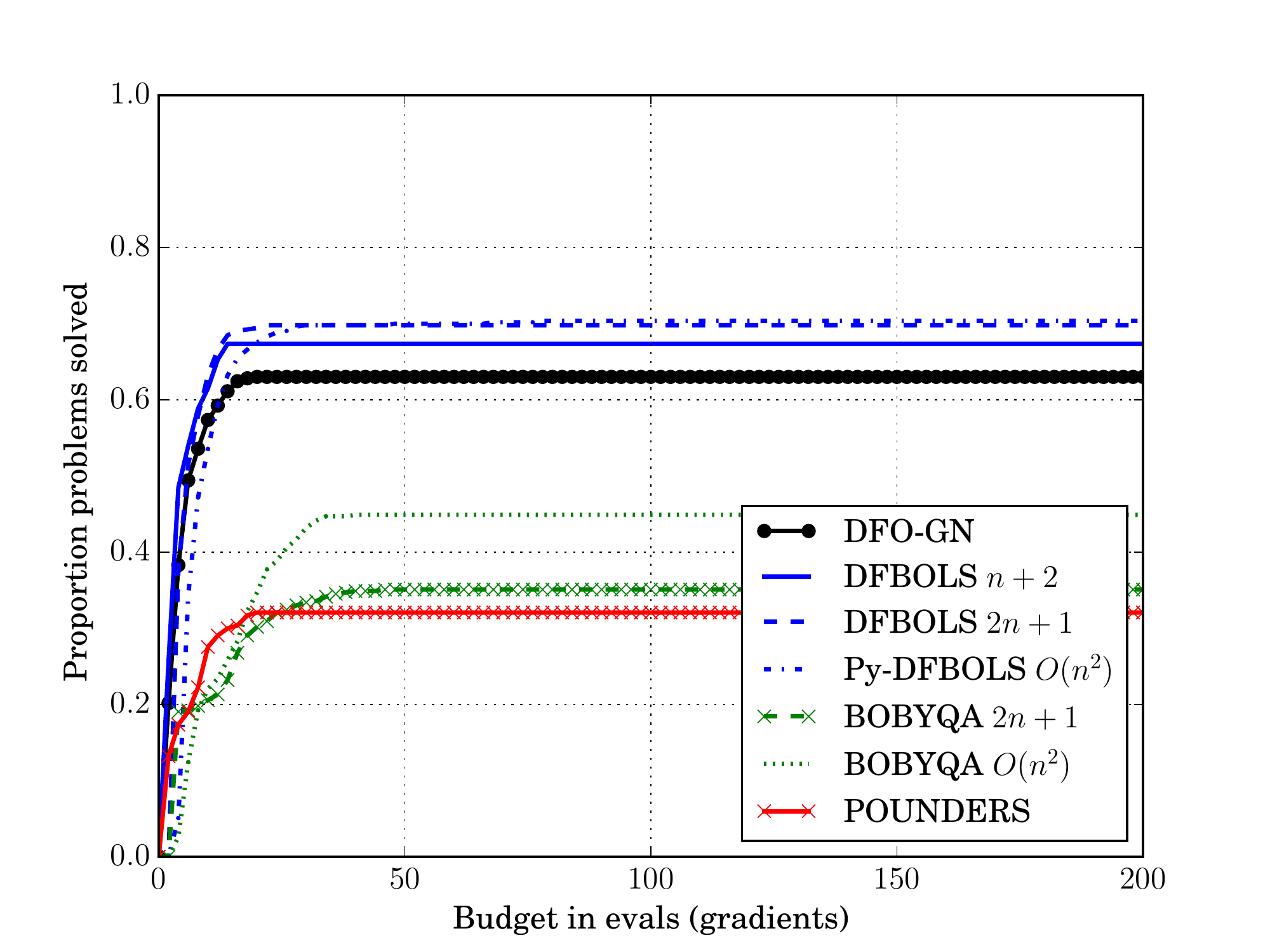}
		\caption{Add.~Gaussian, data profile}
		\label{fig_addgsn_data}
	\end{subfigure}
	~
	\begin{subfigure}[b]{0.48\textwidth}
		\includegraphics[width=\textwidth]{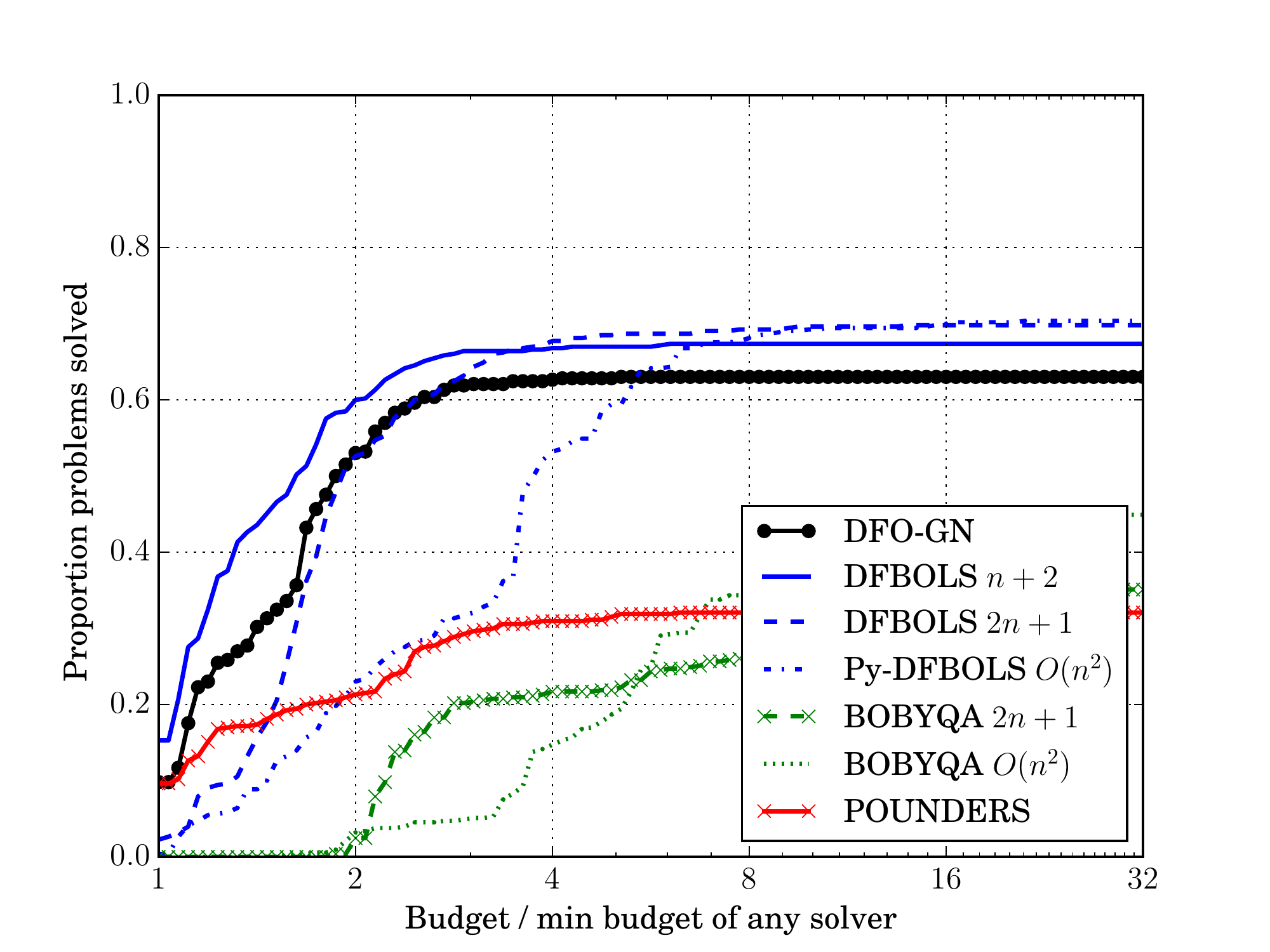}
		\caption{Add.~Gaussian, performance profile}
		\label{fig_addgsn_perf}
	\end{subfigure}
	\\
	\begin{subfigure}[b]{0.48\textwidth}
		\includegraphics[width=\textwidth]{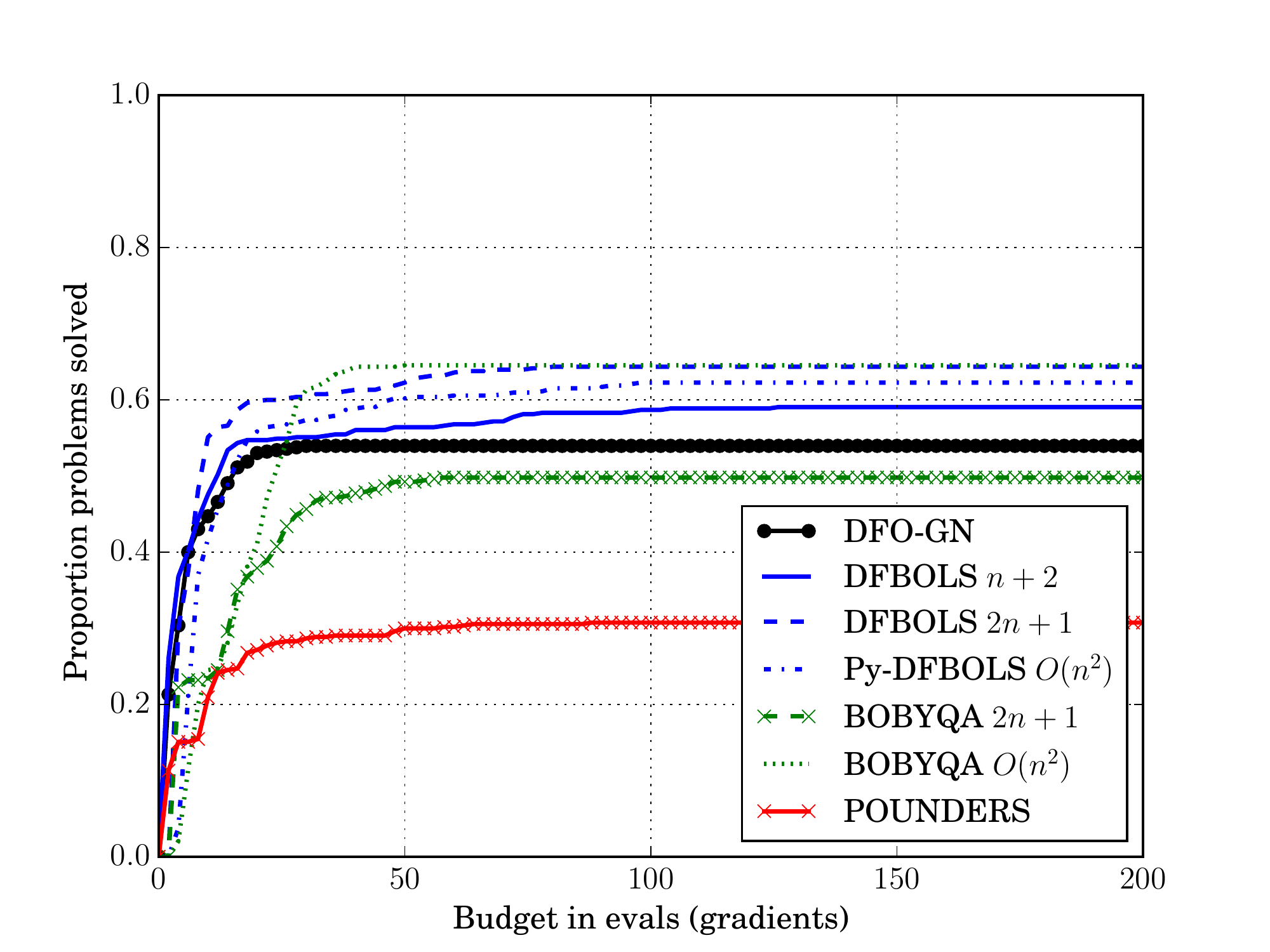}
		\caption{Add.~$\chi^2$, data profile}
		\label{fig_addchisq_data}
	\end{subfigure}
	~
	\begin{subfigure}[b]{0.48\textwidth}
		\includegraphics[width=\textwidth]{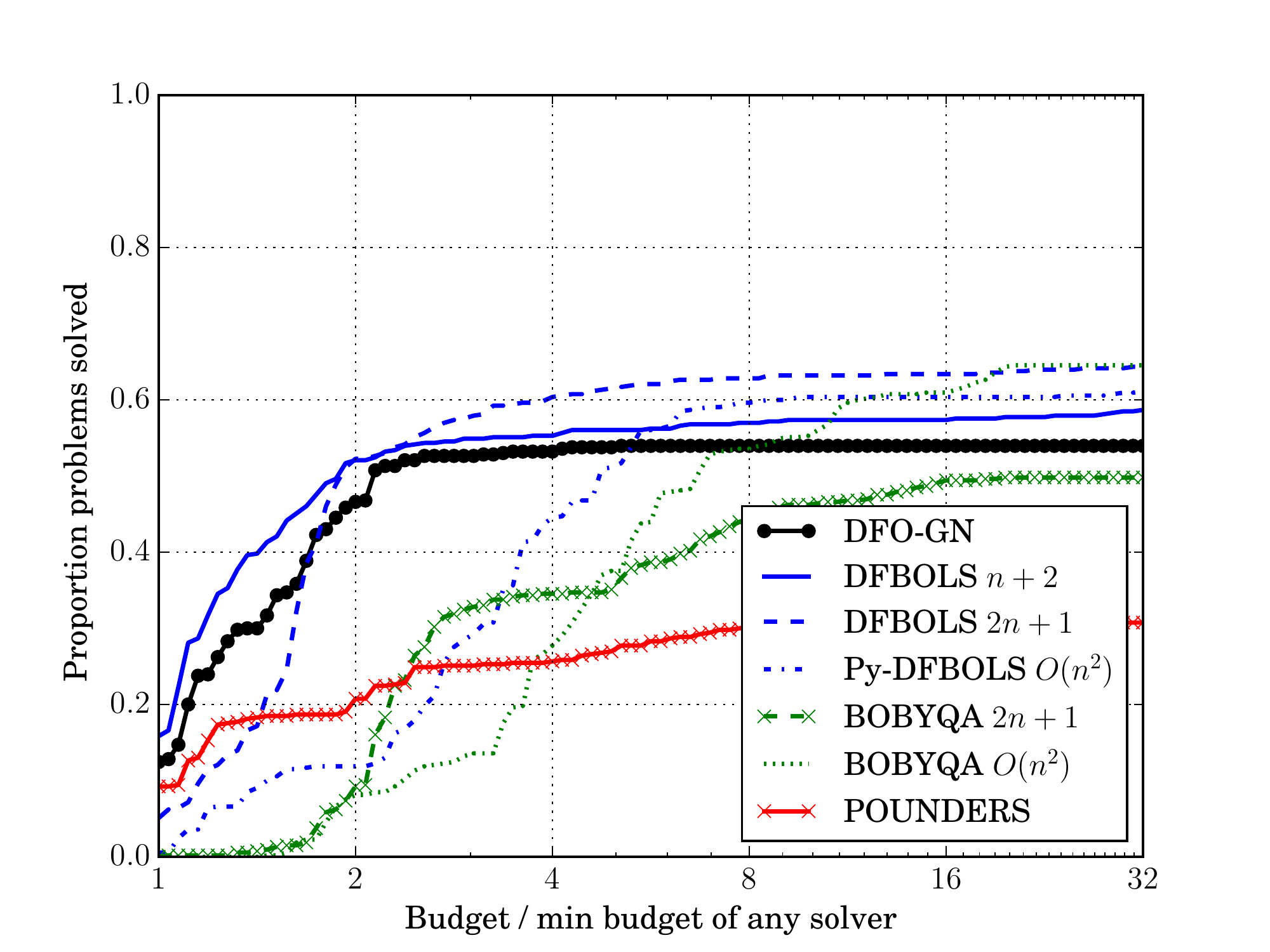}
		\caption{Add.~$\chi^2$, performance profile}
		\label{fig_addchisq_perf}
	\end{subfigure}
	\caption{Comparison of DFO-GN with BOBYQA, DFBOLS and POUNDERS for objectives with multiplicative Gaussian, additive Gaussian and additive $\chi^2$ noise with $\sigma=10^{-2}$, to accuracy $\tau=10^{-5}$ (average of 10 runs for each solver). For the BOBYQA and DFBOLS runs, $n+2$, $2n+1$ and $\bigO(n^2)=(n+1)(n+2)/2$ are the number of interpolation points.}
	\label{fig_noisy}
\end{figure}

\paragraph{Nonzero Residual Problems}
We saw above that DFO-GN suffered a higher --- but still small ---  loss of performance, compared to DFBOLS, for problems with additive noise, where all problems become nonzero residual.
We investigate this further by extracting the performance of the nonzero residual problems only from the test set results we already presented; 
 \figref{fig_nonzero} shows the resulting performance profiles for accuracy $\tau=10^{-5}$, for smooth objectives and multiplicative Gaussian noise ($\sigma=10^{-2}$).
We notice that in the smooth case, DFBOLS with $2n+1$ points is now the fastest solver on 30\% of problems versus 20\% for DFO-GN, compared to looking at all problems (seen in \subfigref{fig_smooth_perf});
but  DFO-GN is comparably robust and we do not see the same loss of performance as in the additive noise case.
In the case of multiplicative Gaussian noise, DFO-GN, and in fact (Py-)DFBOLS, see a loss of performance compared to looking at all problems; also, the difference between DFO-GN and DFBOLS with increasing numbers of interpolation points is slightly more clear for nonzero residual problems (8\% vs.~5\% for all problems). 

\paragraph{Conclusions to evaluation comparisons}  The numerical results in this section show that DFO-GN performs comparably to DFBOLS in terms of evaluation counts, 
and outperforms BOBYQA and POUNDERS, in both smooth and noisy settings, and for low and high accuracy. 
DFO-GN exhibits a slight performance loss compared to DFBOLS for additive noisy problems and for noisy non-zero residual problems. We note that we also tested other noise models --- such as multiplicative uniform noise and also biased variants of the Gaussian noise;
all these performed either better (such as in the case of uniform noise) or essentially indistinguishable to the results already presented above. We also tried other noise variance levels, smaller than $\sigma=10^{-2}$, for which the performance of both DFO-GN and DFBOLS solvers vary similarly/comparably.
We may explain the similar performance of DFO-GN
to DFBOLS, despite the higher order models used by the latter, as being due to the general effectiveness of Gauss-Newton-like frameworks for nonlinear least-squares, especially
for zero-residual problems; and furthermore, by the usual
remit of DFO algorithms in which the asymptotic regimes are not or cannot really be observed or targeted
by the accuracy at which the problems are solved.

Though an accuracy level $\tau=10^{-5}$ is common and considered reasonably high in DFO, 
to ensure that our results are robust, we also performed the same tests for higher accuracy levels $\tau\in\{10^{-7},10^{-9},10^{-11}\}$. The resulting profiles are given in
\appref{sec_numerics_high_accuracy}.
% Appendix \ref{higher_and_higher_acc}.
%The resulting profiles are given in Appendix F of the extended technical report \cite{DFO-GN_tech_report} of this paper. 
For smooth problems, DFO-GN is still able to solve essentially the same proportion of problems as (Py-)DFBOLS.
For noisy problems, the results are more mixed: on average, DFO-GN does slightly worse for Gaussian noise, but slightly better for $\chi^2$ noise.
These results are the same when looking at all problems, or just nonzero residual problems.
This reinforces our previous conclusions, and gives us confidence that a Gauss-Newton framework for DFO is a suitable choice, and is robust to the level of accuracy required for a given problem.

% To ensure that our results are robust, we also performed the same tests for higher accuracy values of $\tau=10^{-7}$, .... We found....
% This reinforces our conclusions, that a GN-like framework is quite suitable...
%Again, we believe the comparable performance is due to the fact that even at these accuracies we do not see the

%%However, for the relatively low accuracy levels at which DFO is used, there is no observed loss of performance for nonzero residual problems in the smooth case.

\begin{figure}
	\centering
	\begin{subfigure}[b]{0.48\textwidth}
		\includegraphics[width=\textwidth]{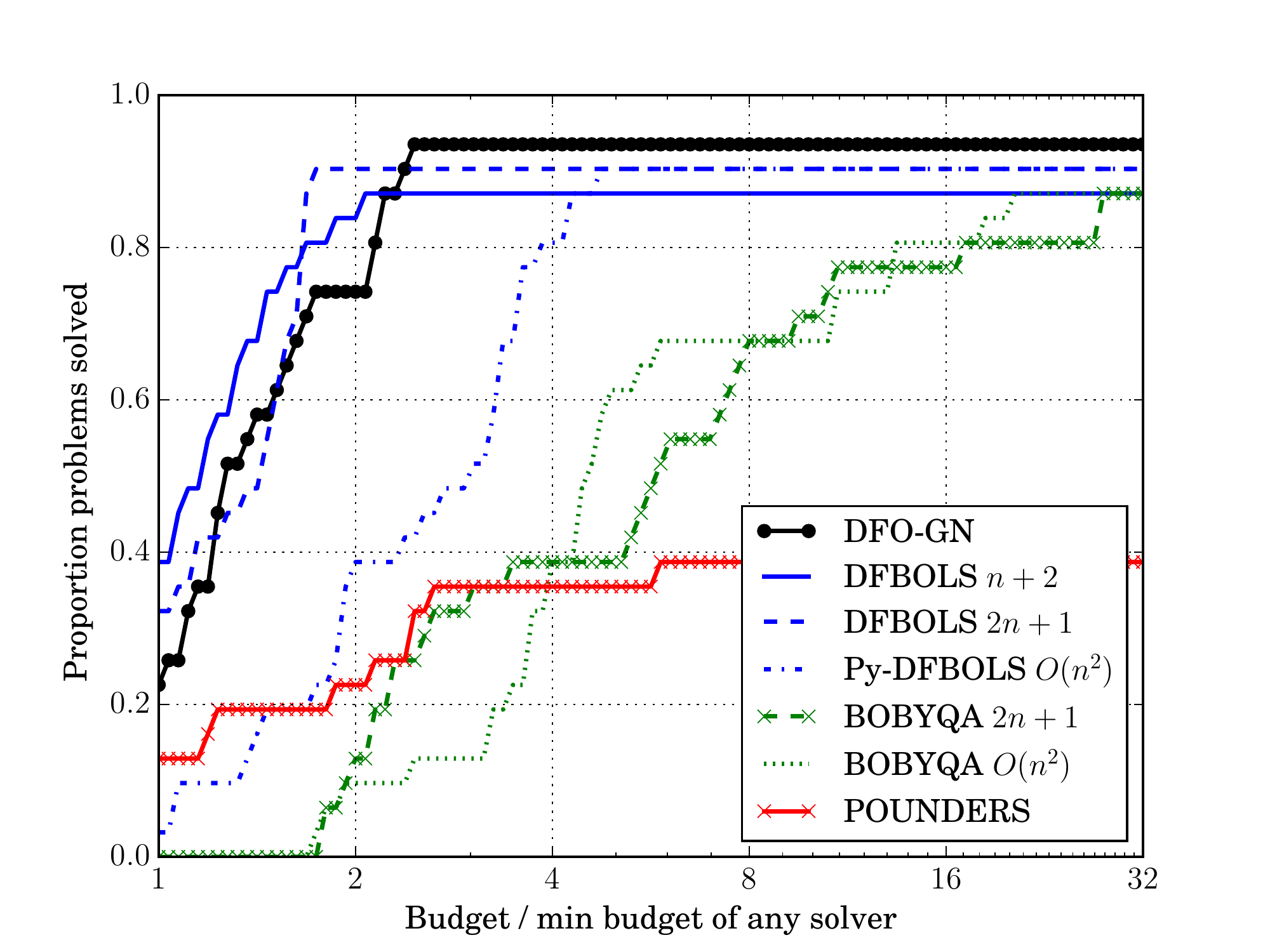}
		\caption{Smooth objective}
		\label{fig_nonzero_smooth}
	\end{subfigure}
	~
	\begin{subfigure}[b]{0.48\textwidth}
		\includegraphics[width=\textwidth]{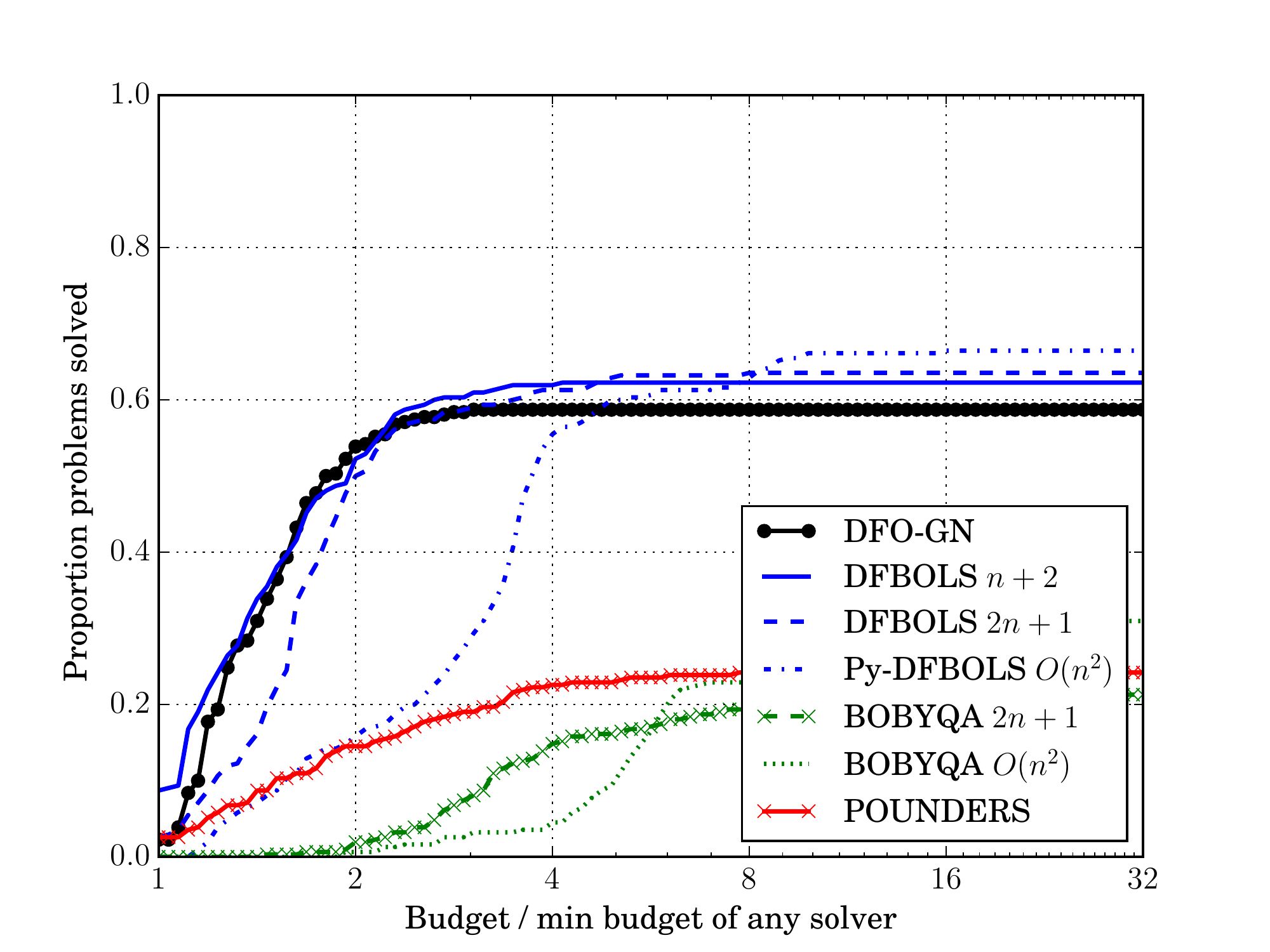}
		\caption{Mult.~Gaussian noise $\sigma=10^{-2}$}
		\label{fig_nonzero_ubgsn}
	\end{subfigure}
	\caption{Performance profile comparison of DFO-GN with BOBYQA, DFBOLS and POUNDERS for nonzero residual problems only, to accuracy $\tau=10^{-5}$. For the BOBYQA and DFBOLS runs, $n+2$, $2n+1$ and $\bigO(n^2)=(n+1)(n+2)/2$ are the number of interpolation points. For \subfigref{fig_nonzero_ubgsn}, results shown are an average of 10 runs for each solver.}
	\label{fig_nonzero}
\end{figure}

\subsection{Runtime Comparison} \label{sec_timing}
The use of linear models in DFO-GN also leads to reduced linear algebra cost. 
The interpolation system for DFO-GN \eqref{eq_linear_interp_system_Jonly} is of size $n$. 
By comparison, for underdetermined quadratic interpolation, such as in (Py-)DFBOLS, the interpolation system has size $p+n+1$ when we have $p\geq n+2$ interpolation points.
Hence the DFBOLS system is at least twice the size of the DFO-GN system, so we would expect the computational cost of solving this system to be at least 8 times larger than for DFO-GN.
To verify this, in this section we compare the runtime of DFO-GN with Py-DFBOLS.
The other solvers are implemented in lower-level languages (Fortran and C), and so a runtime comparison against DFO-GN does not provide a fair comparison.

\begin{table}
	\centering
	\small{
	\begin{tabular}{lllll}  %=========== Wall times ===========%
% 		\toprule
		\hline\noalign{\smallskip}
		Solver & Smooth & Mult.~Gaussian & Add.~Gaussian & Add.~$\chi^2$ \\ \noalign{\smallskip}\hline\noalign{\smallskip} % \midrule
		DFO-GN & 51s [1x] & 8s [1x] & 8s [1x] & 11s [1x] \\
		Py-DFBOLS $n+2$ & 426s [8.4x] & 67s [8.5x] & 59s [7.6x] & 94s [8.8x] \\
		Py-DFBOLS $2n+1$ & 600s [11.8x] & 192s [24.3x] & 174s [22.4x] & 219s [20.4x] \\
		Py-DFBOLS $\bigO(n^2)$ & 2157s [42.3x] & 3052s [386.5x] & 2930s [377.4x] & 3563s [331.1x] \\
% 		\bottomrule
		\noalign{\smallskip}\hline
	\end{tabular}
	}
	\caption{Runtimes for DFO-GN and Py-DFBOLS. Noisy results are an average of 10 runs with $\sigma=10^{-2}$, and all runs used a budget of $200(n+1)$ objective evaluations. Values are raw (in seconds) and ratio compared to the DFO-GN runtime. For Py-DFBOLS, $\bigO(n^2)=(n+1)(n+2)/2$ interpolation points.}
	\label{tab_runtime}
\end{table}

The wall time required by each solver to run the above testing (with a budget of $200(n+1)$ objective evaluations) on a Lenovo ThinkCentre M900 (with one 64-bit Intel i5 processor, 8GB of RAM), is shown in \tabref{tab_runtime} for both smooth and noisy evaluations.
We find that DFO-GN is 7-9x faster than Py-DFBOLS with $n+2$ points, 11-25x faster than Py-DFBOLS with $2n+1$ points, and 42-387x faster than Py-DFBOLS with $(n+1)(n+2)/2$ points.
In all cases, this is a substantial improvement, particularly given the small difference in performance (measured in function evaluations) between DFO-GN and (Py-)DFBOLS described in \secref{sec_results}.

We note that these runtimes include objective evaluations.
However, all solvers used the same Python implementation of the objective functions and the same total budget, so the inclusion of objective evaluations in the runtime will not materially affect the results.

\subsection{Scalability Features}
We saw in \secref{sec_timing} that DFO-GN runs faster due to the lower cost of solving the interpolation linear system.
Another important benefit is that storing the interpolation models for each residual requires only $\bigO(mn)$ memory, rather than $\bigO(mn^2)$ for quadratic models.
These two observations together suggest that DFO-GN should scale to large problems better than DFBOLS --- in this section we demonstrate this.
We consider Problem 29 from Mor\'e, Garbow \& Hillstrom \cite{More1981} (which is Problem 33 ({\sc integreq}) in the CUTEst set in \tabref{tab_cutest_problems}).
 %, originally from \cite{More1979}.
This is a zero-residual least-squares problem with $m=n$ variable for solving a one-dimensional integral equation, using an $n$-point discretization of $(0,1)$.
Both DFO-GN and DFBOLS\footnote{We used $|Y_k|=n+2$ for DFBOLS, as it has the smallest memory usage and runtime of all possible values.} solve this problem --- terminating on small objective value $f(\bx)\leq10^{-12}$ --- quickly, in no more than 20 iterations, regardless of $n$.

\begin{figure}
	\centering
	\begin{subfigure}[b]{0.48\textwidth}
		\includegraphics[width=\textwidth]{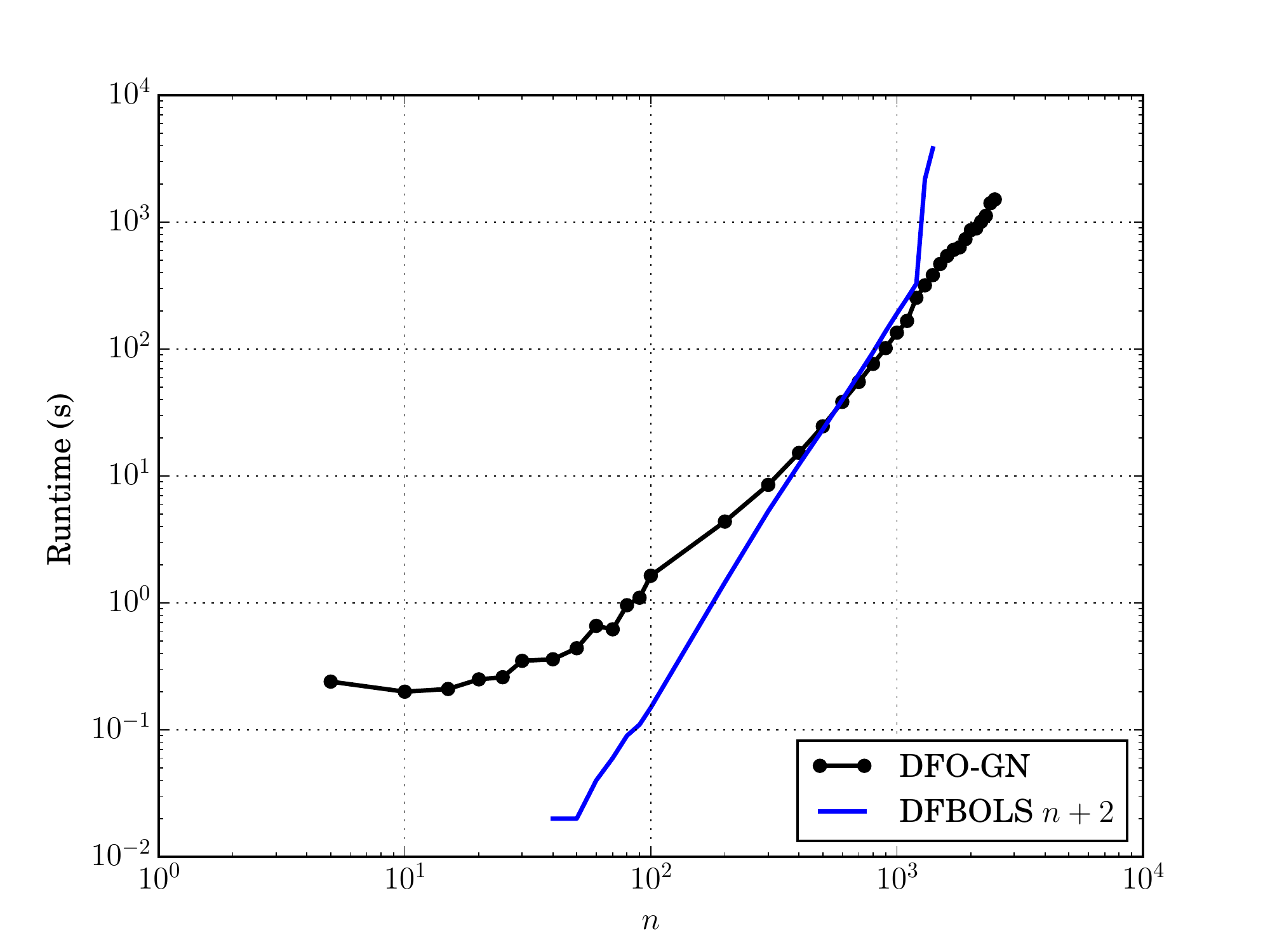}
		\caption{Runtime}
		\label{fig_scaling_runtime}
	\end{subfigure}
	~
	\begin{subfigure}[b]{0.48\textwidth}
		\includegraphics[width=\textwidth]{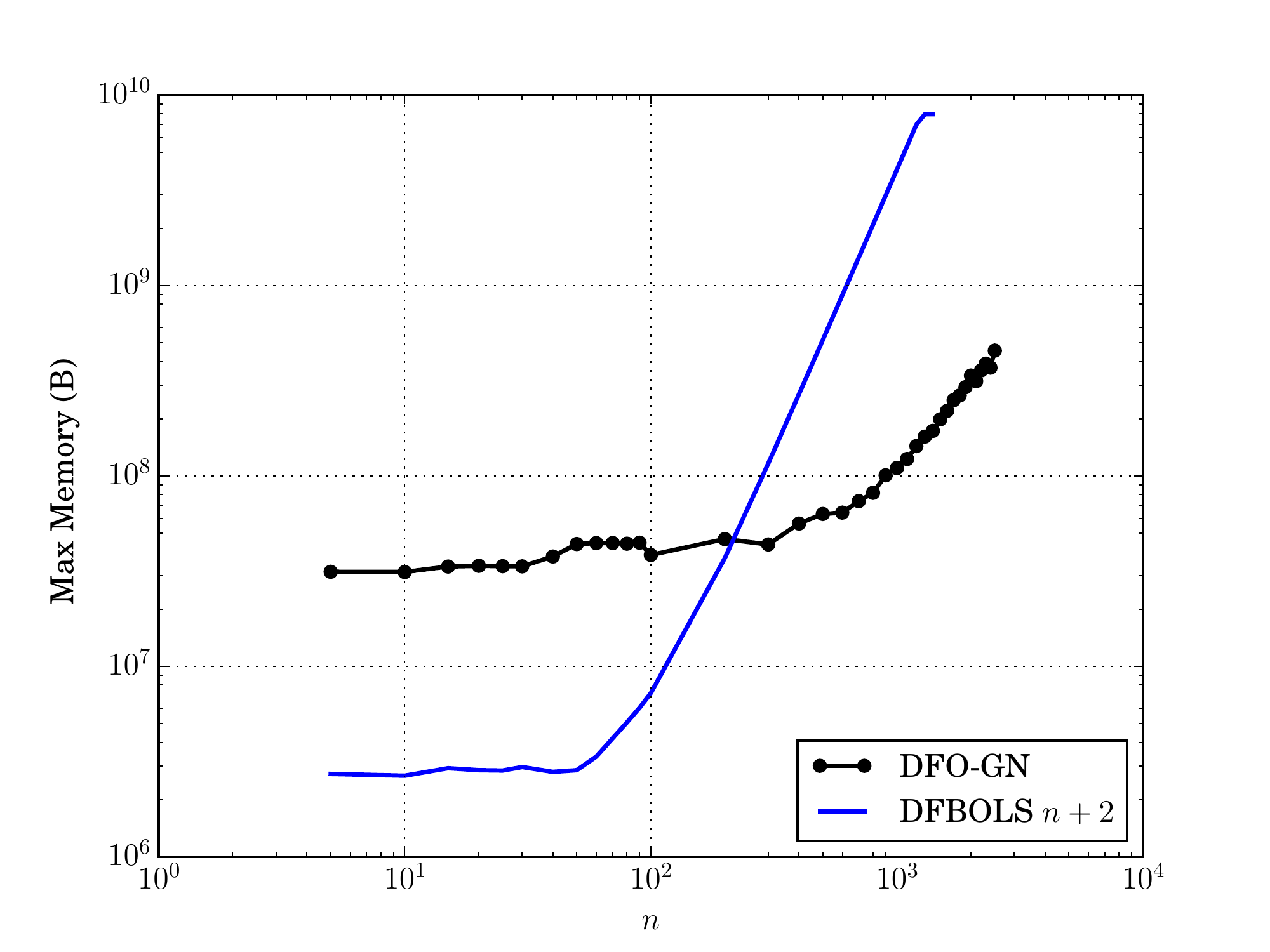}
		\caption{Peak memory usage}
		\label{fig_scaling_memory}
	\end{subfigure}
	\caption{Comparison of runtime and peak memory usage of DFBOLS (original Fortran implementation with $n+2$ interpolation points) and DFO-GN for solving the discretized integral equation, as problem dimension $n$ increases. The largest values tested were $n=1400$ for DFBOLS and $n=2500$ for DFO-GN.}
	\label{fig_scaling}
\end{figure}

In \figref{fig_scaling} we compare the runtime and peak memory usage of DFBOLS and DFO-GN as $n$ increases.
Note that we are comparing DFO-GN (implemented in Python) against DFBOLS (implemented in Fortran) rather than Py-DFBOLS (as used in \secref{sec_timing}), to put ourselves at a substantial disadvantage.
We see that for small $n$, DFBOLS has significantly lower runtime and memory requirements than DFO-GN (which is unsurprising, since it is implemented in Fortran rather than Python).
However, as expected, both the runtime and memory usage increases much faster for DFBOLS than for DFO-GN as $n$ is increased.
In fact, for $n>500$, DFO-GN actually runs faster than DFBOLS.
For $n>1200$, DFBOLS exceeds the memory capacity of the system.
At this point, it has to store data on disk, and as a result the runtime increases very quickly.
DFO-GN does not suffer from this issue, and can continue solving problems quickly for substantially larger $n$.
For instance, DFO-GN solves the $n=2500$ problem over 2.5 times faster than DFBOLS solves the much smaller $n=1400$ problem.

Similarly to before, it is important to gain an understanding of whether this improved scalability comes at the cost of performance.
To assess this, we consider a set of 60 medium-sized problems ($25\leq n\leq 120$ and $25\leq m \leq 400$, with $n\approx 100$ for most problems) from the CUTEst test set \cite{Gould2015}.
The full list of problems is given in \appref{sec_cutest_problems}.
For these problems, we compare DFO-GN with DFBOLS using a smaller budget of $50(n+1)$ evaluations, commensurate with the greater cost of objective evaluation.
Given this small budget, we only test DFBOLS with $n+2$ and $2n+1$ interpolation points; using $(n+1)(n+2)/2$ interpolation points would mean in most cases the full budget is entirely used building the initial sample set.

In \figref{fig_smooth_cutest}, we show data and performance profiles for accuracy $\tau=10^{-5}$.
As before, we see that DFO-GN has very similar performance to DFBOLS, and although we have gained improved scalability, we have not lost in terms of performance on medium-sized test problems.

\begin{figure}
	\centering
	\begin{subfigure}[b]{0.48\textwidth}
		\includegraphics[width=\textwidth]{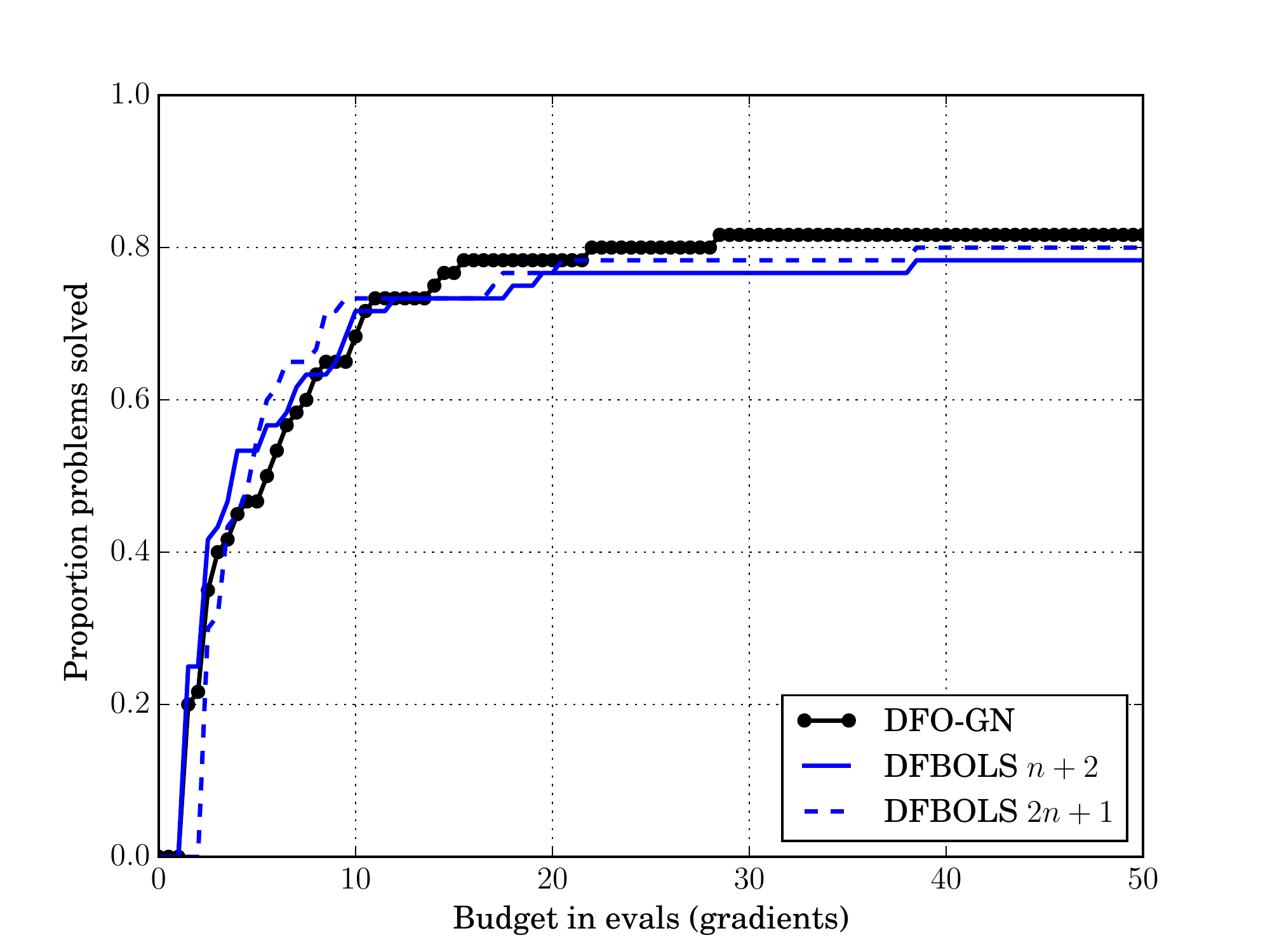}
		\caption{Data Profile}
		\label{fig_smooth_data_cutest}
	\end{subfigure}
	~
	\begin{subfigure}[b]{0.48\textwidth}
		\includegraphics[width=\textwidth]{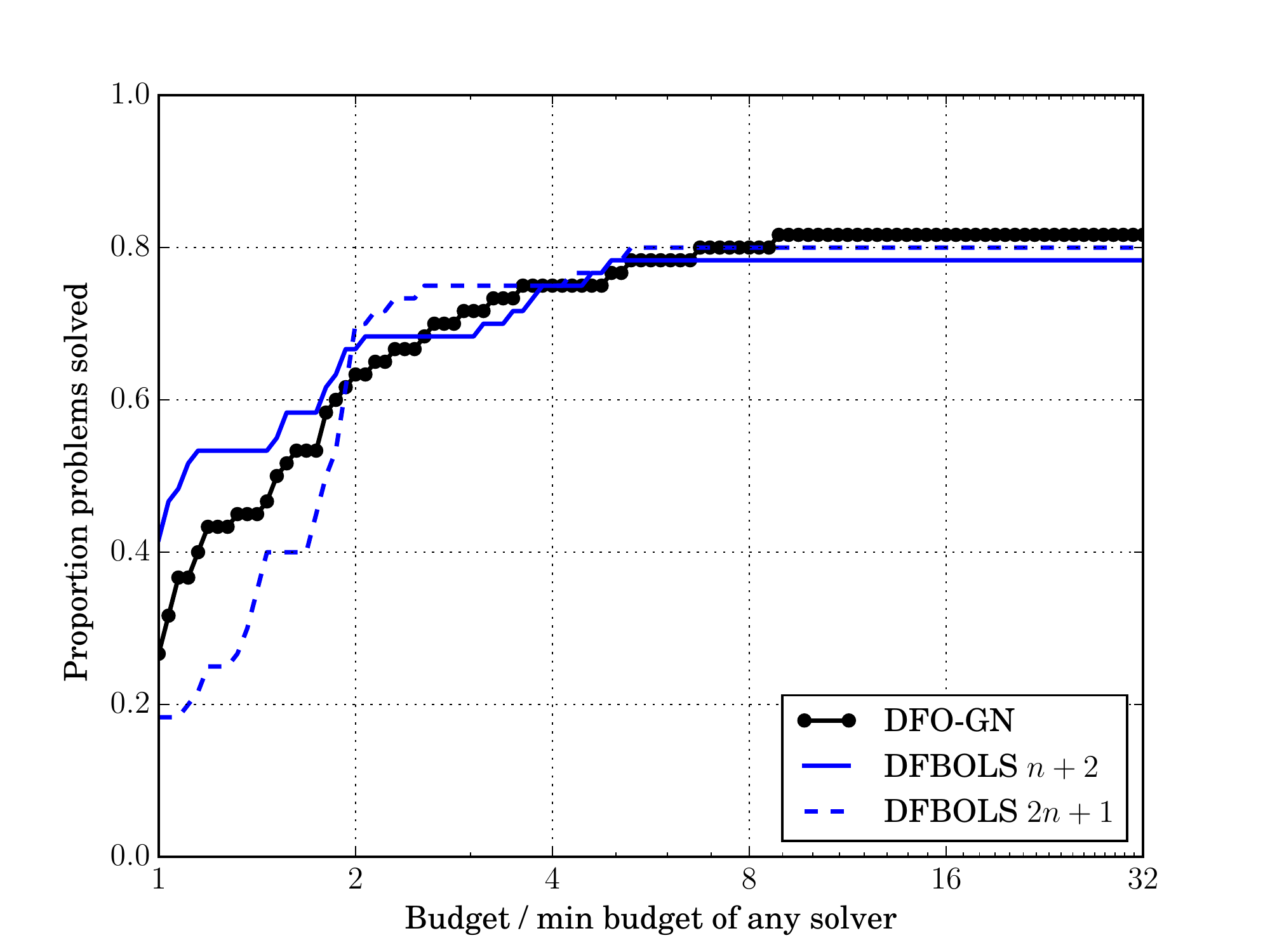}
		\caption{Performance Profile}
		\label{fig_smooth_perf_cutest}
	\end{subfigure}
	\caption{Comparison of DFO-GN with DFBOLS for smooth objectives from the set of medium-sized CUTEst problems, to accuracy $\tau=10^{-5}$. For the DFBOLS runs, $n+2$, $2n+1$ are the number of interpolation points.}
	\label{fig_smooth_cutest}
\end{figure}

\section{Concluding Remarks} \label{sec_conclusion}
%We have presented a derivative-free variant of the Gauss-Newton method for nonlinear least-squares, called DFO-GN.
It is well-known that, for nonlinear least-squares problems, using only linear models for each residual is sufficient to approximate the objective well, especially for zero-residual problems. This forms the basis of the derivative-based Gauss-Newton and Levenberg-Marquardt methods, and has motivated our derivative-free, model-based trust-region variant  here, called DFO-GN.

In \cite{Zhang2010,Wild2017}, quadratic local models are constructed by interpolation for each residual, and these are aggregated to produce a quadratic model for the objective.
By contrast, in DFO-GN we build linear models for each residual, which retains first-order convergence and worst-case complexity, and 
reduces both the computational cost of solving the interpolation problem (leading to a runtime reduction of at least a factor of $7$) and the memory cost of storing the models (from $\bigO(mn^2)$ to $\bigO(mn)$).
These savings result in a substantially faster runtime and improved scalability of DFO-GN compared to DFBOLS, the implementation from \cite{Zhang2010}.
Furthermore, the simpler local models do not adversely affect the algorithm's performance numerically, in terms of evaluation counts:
DFO-GN performs as well as DFBOLS and better than POUNDERS, the implementation from \cite{Wild2017}, on smooth test problems from the Mor\'e \& Wild and CUTEst collections.
When the objective has noise, DFO-GN suffers a small performance penalty compared to DFBOLS (but not to POUNDERS), which is larger for additive  than multiplicative noise 
as all problems become nonzero residual. Nonetheless, this, together with 
%where linear models are known to have slower asymptotic convergence than quadratic models in the derivative-based case.
 the substantial improvements in runtime and scalability, make DFO-GN  an appealing choice for both zero and nonzero residuals, and in the presence of noise.
We delegate to future work showing local quadratic rate of convergence for DFO-GN when applied to nondegenerate zero-residual problems, and generally
improving the performance of DFO methods in the presence of noise. 

%\alert{Comment about complexity?} quadratic local convergence

%===========================================================
%===========================================================
% list BibTeX (.bib) files and choose bibliography style
%\clearpage % make sure all figures/tables have been printed before bibliography pages, start new page
\addcontentsline{toc}{section}{References} % add contents line for bibliography
\bibliographystyle{siam}
% \bibliography{../../library,../../refs_urls} % bib file
\bibliography{dfogn_refs} % bib file

\begin{thebibliography}{10}

\bibitem{AokiEtAl2014}
{\sc Y.~Aoki, B.~Hayami, H.~De~Sterck, and A.~Konagaya}, {\em Cluster {N}ewton
  method for sampling multiple solutions of underdetermined inverse problems:
  application to a parameter identification problem in pharmacokinetics}, SIAM
  J. Sci. Comput., 36 (2014), pp.~B14--B44.

\bibitem{Bergou2016}
{\sc E.~Bergou, S.~Gratton, and L.~N. Vicente}, {\em {Levenberg–Marquardt
  Methods Based on Probabilistic Gradient Models and Inexact Subproblem
  Solution, with Application to Data Assimilation}}, SIAM/ASA J. Uncertain.
  Quantif., 4 (2016), pp.~924--951.

\bibitem{Conn2000}
{\sc A.~R. Conn, N.~I.~M. Gould, and P.~L. Toint}, {\em {Trust-Region
  Methods}}, MPS-SIAM Series on Optimization, MPS/SIAM, Philadelphia, 2000.

\bibitem{Conn1997}
{\sc A.~R. Conn, K.~Scheinberg, and P.~L. Toint}, {\em {Recent progress in
  unconstrained nonlinear optimization without derivatives}}, Math. Program.,
  79 (1997), pp.~397--414.

\bibitem{Conn2007}
{\sc A.~R. Conn, K.~Scheinberg, and L.~N. Vicente}, {\em {Geometry of
  interpolation sets in derivative free optimization}}, Math. Program., 111
  (2007), pp.~141--172.

\bibitem{Conn2009a}
\leavevmode\vrule height 2pt depth -1.6pt width 23pt, {\em {Global Convergence
  of General Derivative-Free Trust-Region Algorithms to First- and Second-Order
  Critical Points}}, SIAM J. Optim., 20 (2009), pp.~387--415.

\bibitem{Conn2009}
\leavevmode\vrule height 2pt depth -1.6pt width 23pt, {\em {Introduction to
  Derivative-Free Optimization}}, vol.~8 of MPS-SIAM Series on Optimization,
  MPS/SIAM, Philadelphia, 2009.

\bibitem{Conn1996}
{\sc A.~R. Conn and P.~L. Toint}, {\em {An Algorithm using Quadratic
  Interpolation for Unconstrained Derivative Free Optimization}}, in Nonlinear
  Optim. Appl., G.~{Di Pillo} and F.~Gianessi, eds., Plenum Publishing, New
  York, 1996, pp.~27--47.

\bibitem{Custodio2017}
{\sc A.~L. Cust{\'{o}}dio, K.~Scheinberg, and L.~N. Vicente}, {\em
  {Methodologies and Software for Derivative-free Optimization}}, in Adv.
  Trends Optim. with Eng. Appl., T.~Terlaky, M.~F. Anjos, and S.~Ahmed, eds.,
  MOS-SIAM Book Series on Optimization, SIAM, Philadelphia, 2017.

\bibitem{Garmanjani2016}
{\sc R.~Garmanjani, D.~J{\'{u}}dice, and L.~N. Vicente}, {\em {Trust-Region
  Methods Without Using Derivatives: Worst Case Complexity and the Nonsmooth
  Case}}, SIAM J. Optim., 26 (2016), pp.~1987--2011.

\bibitem{Gould2015}
{\sc N.~I.~M. Gould, D.~Orban, and P.~L. Toint}, {\em {CUTEst: a Constrained
  and Unconstrained Testing Environment with safe threads for mathematical
  optimization}}, Comput. Optim. Appl., 60 (2015), pp.~545--557.

\bibitem{Gould2012}
{\sc N.~I.~M. Gould, M.~Porcelli, and P.~L. Toint}, {\em {Updating the
  regularization parameter in the adaptive cubic regularization algorithm}},
  Comput. Optim. Appl., 53 (2012), pp.~1--22.

\bibitem{Grapiglia2016}
{\sc G.~N. Grapiglia, J.~Yuan, and Y.-x. Yuan}, {\em {A derivative-free
  trust-region algorithm for composite nonsmooth optimization}}, Comput. Appl.
  Math., 35 (2016), pp.~475--499.

\bibitem{Judice2015}
{\sc D.~J{\'{u}}dice}, {\em {Trust-Region Methods without using Derivatives :
  Worst-Case Complexity and the Non-Smooth Case}}, PhD thesis, University of
  Coimbra, 2015.

\bibitem{Kolda2003}
{\sc T.~G. Kolda, R.~M. Lewis, and V.~Torczon}, {\em {Optimization by Direct
  Search : New Perspectives on Some Classical and Modern Methods}}, SIAM Rev.,
  45 (2003), pp.~385--482.

\bibitem{Luksan1996}
{\sc L.~Luk{\v{s}}an}, {\em {Hybrid Methods for Large Sparse Nonlinear Least
  Squares}}, J. Optim. Theory Appl., 89 (1996), pp.~575--595.

\bibitem{More1981}
{\sc J.~J. Mor{\'{e}}, B.~S. Garbow, and K.~E. Hillstrom}, {\em {Testing
  Unconstrained Optimization Software}}, ACM Trans. Math. Softw., 7 (1981),
  pp.~17--41.

\bibitem{More2009}
{\sc J.~J. Mor{\'{e}} and S.~M. Wild}, {\em {Benchmarking Derivative-Free
  Optimization Algorithms}}, SIAM J. Optim., 20 (2009), pp.~172--191.

\bibitem{Nocedal2006}
{\sc J.~Nocedal and S.~J. Wright}, {\em {Numerical Optimization}}, Springer
  Series in Operations Research and Financial Engineering, Springer, New York,
  2nd~ed., 2006.

\bibitem{Oeuvray2009}
{\sc R.~Oeuvray and M.~Bierlaire}, {\em {BOOSTERS: A Derivative-Free Algorithm
  Based on Radial Basis Functions}}, Int. J. Model. Simul., 29 (2009),
  pp.~26--36.

\bibitem{Powell1994}
{\sc M.~J.~D. Powell}, {\em {A direct search optimization method that models
  the objective and constraint functions by linear interpolation}}, in Adv.
  Optim. Numer. Anal., S.~Gomez and J.-P. Hennart, eds., Kluwer Academic
  Publishers, Dordrecht, 1994, pp.~51--67.

\bibitem{Powell1998}
\leavevmode\vrule height 2pt depth -1.6pt width 23pt, {\em Direct search
  algorithms for optimization calculations}, Acta Numerica, 7 (1998),
  pp.~287--336.

\bibitem{Powell2002}
\leavevmode\vrule height 2pt depth -1.6pt width 23pt, {\em {UOBYQA:
  Unconstrained optimization by quadratic approximation}}, Math. Program., 92
  (2002), pp.~555--582.

\bibitem{Powell2003}
\leavevmode\vrule height 2pt depth -1.6pt width 23pt, {\em {On trust region
  methods for unconstrained minimization without derivatives}}, Math. Program.,
  97 (2003), pp.~605--623.

\bibitem{Powell2004a}
\leavevmode\vrule height 2pt depth -1.6pt width 23pt, {\em {Least Frobenius
  norm updating of quadratic models that satisfy interpolation conditions}},
  Math. Program., 100 (2004), pp.~183--215.

\bibitem{Powell2007a}
\leavevmode\vrule height 2pt depth -1.6pt width 23pt, {\em {A view of
  algorithms for optimization without derivatives}}, Tech. Rep. DAMTP
  2007/NA03, University of Cambridge, 2007.

\bibitem{Powell2009}
\leavevmode\vrule height 2pt depth -1.6pt width 23pt, {\em {The BOBYQA
  algorithm for bound constrained optimization without derivatives}}, Tech.
  Rep. DAMTP 2009/NA06, University of Cambridge, 2009.

\bibitem{Scheinberg2010}
{\sc K.~Scheinberg and P.~L. Toint}, {\em {Self-Correcting Geometry in
  Model-Based Algorithms for Derivative-Free Unconstrained Optimization}}, SIAM
  J. Optim., 20 (2010), pp.~3512--3532.

\bibitem{Tange2011}
{\sc O.~Tange}, {\em {GNU Parallel: the command-line power tool}}, ;login
  USENIX Mag., 36 (2011), pp.~42--47.

\bibitem{Wild2017}
{\sc S.~M. Wild}, {\em {POUNDERS in TAO: Solving Derivative-Free Nonlinear
  Least-Squares Problems with POUNDERS}}, in Adv. Trends Optim. with Eng.
  Appl., SIAM, Philadelphia, PA, 2017, ch.~40, pp.~529--539.

\bibitem{Wild2008}
{\sc S.~M. Wild, R.~G. Regis, and C.~A. Shoemaker}, {\em {ORBIT: Optimization
  by Radial Basis Function Interpolation in Trust-Regions}}, SIAM J. Sci.
  Comput., 30 (2008), pp.~3197--3219.

\bibitem{Wild2013}
{\sc S.~M. Wild and C.~A. Shoemaker}, {\em {Global Convergence of Radial Basis
  Function Trust-Region Algorithms for Derivative-Free Optimization}}, SIAM
  Rev., 55 (2013), pp.~349--371.

\bibitem{Winfield1973}
{\sc D.~Winfield}, {\em {Function minimization by interpolation in a data
  table}}, IMA J. Appl. Math., 12 (1973), pp.~339--347.

\bibitem{Zhang2010}
{\sc H.~Zhang, A.~R. Conn, and K.~Scheinberg}, {\em {A Derivative-Free
  Algorithm for Least-Squares Minimization}}, SIAM J. Optim., 20 (2010),
  pp.~3555--3576.

\bibitem{Zhang_URL}
{\sc Z.~Zhang}, {\em {Software by Professor M.~J.~D.~Powell}}.
\newblock \url{http://mat.uc.pt/~zhang/software.html}, 2017.

\end{thebibliography}
% \clearpage
\appendix

\section{Proof of \lemref{lem_fully_linear}} \label{sec_fully_linear_pf}
This proof is similar to \cite[Lemma 4.3]{Zhang2010} and \cite[Theorems 2.11 and 2.12]{Conn2009}.
Define $B\defeq B(\bx_k, \Delta_k)$ for convenience.
We recall the standard bound \cite[Appendix A]{Nocedal2006} 
\be \|\br(\by) - \br(\bx_k) - J(\bx_k)(\by-\bx_k)\| \leq \frac{1}{2}L_J\|\by-\bx_k\|^2. \label{eq_lipschitz_bd} \ee
From the interpolation conditions \eqref{eq_interp_conditions}, we have
\be J_k(\by_t - \bx_k) = \br(\by_t) - \br(\bx_k), \qquad \text{for $t=1,\ldots,n$.} \ee
Using \eqref{eq_lipschitz_bd}, we compute for any $t=1,\ldots,n$,
\be \|[J_k-J(\bx_k)](\by_t-\bx_k)/\Delta_k\| = \Delta_k^{-1}\|\br(\by_t) - \br(\bx_k) - J(\bx_k)(\by_t-\bx_k)\| \leq \frac{1}{2}L_J \Delta_k. \ee
Let $\hat{W}_k$ be the interpolation matrix of the system
\eqref{eq_linear_interp_system_Jonly} scaled by $\Delta_k^{-1}$. 
Considering the matrix $[J_k-J(\bx_k)]\hat{W}^{\top}_k$, with columns $[J_k-J(\bx_k)](\by_t-\bx_k)/\Delta_k$, we have
\be \|[J_k-J(\bx_k)]\hat{W}^{\top}_k\|^2 \leq \|[J_k-J(\bx_k)]\hat{W}^{\top}_k\|_F^2 = \sum_{t=1}^{n}\|[J_k-J(\bx_k)](\by_t-\bx_k)/\Delta_k\|^2, \ee
and so using the identity $\|\hat{W}^{-1}_k\|=\|\hat{W}^{-\top}_k\|$, we get
\be \|J_k-J(\bx_k)\| \leq \|[J_k-J(\bx_k)]\hat{W}^{\top}_k\|\cdot\|\hat{W}^{-\top}_k\| \leq \frac{1}{2}L_J \sqrt{n} \|\hat{W}^{-1}_k\| \Delta_k. \ee
Thus we conclude that for any $\by\in B$
\be \|J_k-J(\by)\| \leq \|J_k-J(\bx_k)\| + \|J(\by)-J(\bx_k)\| \leq
L_J \left(1 +
  \frac{1}{2}\sqrt{n}\|\hat{W}^{-1}_k\|\right)\Delta_k. \ee
Since $Y_k$
is $\Lambda$-poised in $B$, we have $\|\hat{W}^{-1}_k\|=\bigO(\Lambda)$
from \cite[Theorem 3.14]{Conn2009}. Thus \eqref{eq_fully_linear_vector_g} holds with $\kappa_{eg}^r \defeq
L_J \left(1 + \frac{1}{2}\sqrt{n}C\right)$, where
$C=\bigO(\Lambda)$.
Next, we prove \eqref{eq_fully_linear_vector_f} by computing
\begin{align}
	\|\bem_k(\by-\bx_k) - \br(\by)\| &= \|\br(\by) - \br(\bx_k) - J_k(\by-\bx_k)\|, \\
	&\leq \|\br(\by) - \br(\bx_k) - J(\bx_k)(\by-\bx_k)\| + \|J(\bx_k)-J_k\|\cdot\|\by-\bx_k\|, \\
	&\leq \left(\frac{L_J}{2} + \kappa_{eg}^r\right)\Delta_k^2,
\end{align}
where we use \eqref{eq_fully_linear_vector_g} and \eqref{eq_lipschitz_bd}.
Hence we have \eqref{eq_fully_linear_vector_f} with $\kappa_{ef}^r = \kappa_{eg}^r + L_J/2$, as required.
%Thus we have shown that $\bem_k$ is a fully linear model for $\br$, with constants in \defref{def_fully_linear_vector} given by \eqref{eq_fully_linear_vector_consts}.

Since $\bem_k$ is fully linear, we also get from \eqref{eq_fully_linear_vector_g} the bound
\be \|J_k\| \leq \|J(\bx_k) - J_k\| + \|J(\bx_k)\| \leq \kappa_{eg}^r\Delta_{max} + J_{max}, \ee
so $\|J_k\|$ is uniformly bounded for all $k$.
Since $H_k = J_k^{\top}J_k$, this means that $\|H_k\| = \|J_k\|^2$ is uniformly bounded for all $k$.

To prove full linearity of $m_k$, we first compute
\begin{align}
	\|\grad m_k(\by-\bx_k) - \grad f(\by)\| &= \|\grad f(\by) - J_k^{\top}\br(\bx_k) - J_k^{\top}J_k(\by-\bx_k)\|, \\
	&\leq \|\grad f(\by) - \grad f(\bx_k)\| + \|(J(\bx_k)-J_k)^{\top}\br(\bx_k)\| \nonumber \\
	&\qquad\qquad + \|J_k^{\top}J_k\|\cdot\|\by-\bx_k\|, \\
	&\leq L_{\grad f} \Delta_k + \kappa_{eg}^r r_{max}\Delta_k + \left(\kappa_{eg}^r\Delta_{max}+J_{max}\right)^2\Delta_k,
\end{align}
recovering \eqref{eq_fully_linear_scalar_g} with $\kappa_{eg}=L_{\grad f} + \kappa_{eg}^r r_{max} + (\kappa_{eg}^r\Delta_{max}+J_{max})^2$, as required.

Lastly, to show \eqref{eq_fully_linear_scalar_f}, we recall the scalar version of \eqref{eq_lipschitz_bd}
\be |f(\by)-f(\bx_k)-\grad f(\bx_k)^{\top}(\by-\bx_k)| \leq \frac{1}{2}L_{\grad f}\|\by-\bx_k\|^2. \ee
We use this and \eqref{eq_fully_linear_scalar_g} to compute
\begin{align}
	|m_k(\by-\bx_k) - f(\by)| &= \left|f(\by) - f(\bx_k) - \bg_k^{\top}(\by-\bx_k) - \frac{1}{2}(\by-\bx_k)^{\top}H_k(\by-\bx_k)\right|, \\
	&\leq \left|f(\by) - f(\bx_k) - \grad f(\bx_k)^{\top}(\by-\bx_k)\right| \nonumber \\ 
	&\qquad\qquad + \left\|\grad f(\bx_k) - \bg_k - \frac{1}{2}H_k(\by-\bx_k)\right\|\cdot\|\by-\bx_k\|, \\
	&\leq \frac{1}{2}L_{\grad f} \Delta_k^2 + \left[\|\grad f(\bx_k) - \grad m_k(\by-\bx_k)\| + \frac{1}{2}\|H_k\|\cdot\|\by-\bx_k\|\right]\cdot \Delta_k, \\
	&\leq \frac{1}{2}L_{\grad f} \Delta_k^2 + \left[\kappa_{eg}\Delta_k + \frac{1}{2}(\kappa_{eg}^r\Delta_{max}+J_{max})^2\Delta_k\right]\Delta_k,
\end{align}
and so we have \eqref{eq_fully_linear_scalar_f} with $\kappa_{ef} = \kappa_{eg} + L_{\grad f}/2 + (\kappa_{eg}^r\Delta_{max} + J_{max})^2/2$.\smartqed

% \clearpage

\section{Geometry Improvement in Criticality Phase} \label{sec_criticality_geom}
Here, we describe the geometry-improvement step performed in the criticality phase of \algref{alg_dfogn}, and prove its convergence.
The proof of \lemref{lem_criticality_complexity} can be derived from the proof of \cite[Lemma 10.5]{Conn2009}.

\begin{algorithm}[H]\small{
	\begin{algorithmic}[1]
		\Require Iterate $\bx_k$, initial set $Y_k$ and trust region radius $\Delta_k^{init}$. 
		\Statex Parameters are $\mu>0$, $\omega_C\in(0,1)$ and poisedness constant $\Lambda>0$.
		\State Set $Y_k^{(0)}=Y_k$.
		\For{$i=1,2,\ldots$}
			\State Form $Y_k^{(i)}$ by modifying $Y_k^{(i-1)}$ until it is $\Lambda$-poised in $B(\bx_k, \omega_C^{i-1}\Delta_k^{init})$. 
			\State Solve the interpolation system for $Y_k^{(i)}$ and form $m_k^{(i)}$ \eqref{eq_gn_full_model_dfo}. 
			\If{$\omega_C^{i-1}\Delta_k^{init} \leq \mu\|\bg_k^{(i)}\|$}
				\State \textbf{return} $Y_k^{(i)}$, $m_k^{(i)}$, $\Delta_k\gets\omega_C^{i-1}\Delta_k^{init}$.
			\EndIf
		\EndFor
	\end{algorithmic}}
	\caption{Geometry-Improvement for Criticality Phase}
	\label{alg_criticality}
\end{algorithm}

We recall from \lemref{lem_fully_linear} that if $Y_k^{(i)}$ is $\Lambda$-poised in $B(\bx_k,\omega_C^{i-1}\Delta_k^{init})$, then $m_k^{(i)}$ is fully linear in the sense of \defref{def_fully_linear_scalar}, with associated constants $\kappa_{ef}$ and $\kappa_{eg}$ in \eqref{eq_fully_linear_scalar_f} and \eqref{eq_fully_linear_scalar_g} respectively given by \eqref{eq_fully_linear_scalar_consts}.

\begin{lemma} \label{lem_criticality_complexity}
	Suppose $\|\grad f(\bx_k)\|\geq\epsilon>0$. 
	Then for any $\mu>0$ and $\omega_C\in(0,1)$, \algref{alg_criticality} terminates in finite time with $Y_k$ $\Lambda$-poised in $B(\bx_k,\Delta_k)$ and $\Delta_k\leq\mu\|\bg_k\|$ for any $\mu>0$ and $\omega_C\in(0,1)$.
	We also have the bound
	\be \min\left(\Delta_k^{init}, \frac{\omega_C\epsilon}{\kappa_{eg}+1/\mu}\right) \leq \Delta_k \leq \Delta_k^{init}. \label{eq_criticality_delta_bound} \ee
\end{lemma}
\begin{proof}
	First, suppose \algref{alg_criticality} terminates on the first iteration.
	Then $\Delta_k=\Delta_k^{init}$, and the result holds.
	
	Otherwise, consider some iteration $i$ where \algref{alg_criticality} does not terminate; that is, where $\omega_C^{i-1}\Delta_k^{init} > \mu\|\bg_k^{(i)}\|$.
	Then since $m_k^{(i)}$ is fully linear in $B(\bx_k, \omega_C^{i-1}\Delta_k^{init})$, we have
	\be \epsilon \leq \|\grad f(\bx_k)\| \leq \|\grad f(\bx_k) - \bg_k^{(i)}\| + \|\bg_k^{(i)}\| \leq \left(\kappa_{eg} + \frac{1}{\mu}\right)\omega_C^{i-1}\Delta_k^{init}, \ee
	or equivalently
	\be \omega_C^{i-1} \geq \frac{\epsilon}{(\kappa_{eg}+1/\mu)\Delta_k^{init}}. \label{eq_criticality_tmp1} \ee
	That is, if termination does not occur on iteration $i$, we must have
	\be i \leq 1 + \frac{1}{|\log\omega_C|}\log\left(\frac{(\kappa_{eg}+1/\mu)\Delta_k^{init}}{\epsilon}\right), \ee
	so \algref{alg_criticality} terminates in finite time.
	We also have
	\be \omega_C^{i-1}\Delta_k^{init} \geq \frac{\epsilon}{\kappa_{eg}+1/\mu}, \ee
	from which \eqref{eq_criticality_delta_bound} follows.
\end{proof}

% \clearpage

\section{Calculating Geometry-Improving Steps} \label{sec_geom_step_algo}
Here, we provide \algref{alg_geom_step}, a method for solving one subproblem for calculating a geometry-improving step \eqref{eq_geom_improvement} subject to bound constraints.
That is, we solve the problem
\be \by^+ \defeq \argmax_{\by\in B(\bx_k,\Delta_k)} \bg^{\top}\by \quad \text{subject to $\b{a} \leq \by \leq \b{b}$.} \label{eq_geom_constrained} \ee

\begin{algorithm}\small{
	\begin{algorithmic}[1]
		\State Initialize $\by_0\defeq \bx_k$, step direction $\bs_0=-\bg$ and inactive set for bounds $\mathcal{I}=\{1,\ldots,n\}$.
		\For{$j=0,1,2,\ldots,n-1$}
			\If{$\|\bs_j\|=0$}
				\State \textbf{return} $\by_j$. \label{ln_zero_quit}
			\EndIf
			\State Find tentative step length $\alpha_j\geq 0$ by finding the largest solution to $\|\by_j+\alpha_j\bs_j\|^2=\Delta_k^2$.
			\If{$\b{a} \leq \by_j + \alpha_j\bs_j \leq \b{b}$}
				\State \textbf{return} $\by_{j+1} = \by_j + \alpha_j\bs_j$. \label{ln_uncon_quit}
			\Else
				\State Let $i\in\mathcal{I}$ be any index such that $y_{j,i}+\alpha_j s_{j,i} \notin[a_i, b_i]$.
				\State Let $\beta_j \leq \alpha_j$ be the largest value such that $y_{j,i}+\beta_j s_{j,i} \in[a_i, b_i]$
				\State Set $\by_{j+1} = \by_j + \beta_j\bs_j$ and remove $i$ from $\mathcal{I}$.
				\State Set $\bs_{j+1}=\bs_j$, except for setting $s_{j+1,i}=0$.
			\EndIf
		\EndFor
		\State \textbf{return} $\by_n$. \label{ln_box_quit}
	\end{algorithmic}}
	\caption{Solve geometry-improving subproblem \eqref{eq_geom_constrained}.}
	\label{alg_geom_step}
\end{algorithm}

\begin{lemma}
	Suppose $\b{a}\leq \bx_k \leq \b{b}$ and $\b{a}<\b{b}$. Then \algref{alg_geom_step} returns a minimizer of \eqref{eq_geom_constrained}.
\end{lemma}
\begin{proof}
	Without loss of generality, we may assume that $\bx_k=\b{0}$, as otherwise we can simply shift $\b{a}\to\b{a}-\bx_k$, and similarly for $\b{b}$.	
	Since \eqref{eq_geom_constrained} is convex, so it suffices to show convergence to a KKT point.
	
	Suppose $\by$ is a KKT point where $y_i\in\{a_i, b_i\}$ for some $i\in\{1,\ldots,n\}$. 
	If $y_i=a_i$, then if we took $\t{\by}=\by$ except for $\t{y}_i=b_i$, then 
	\be \bg^{\top}\t{\by} = \bg^{\top}\by + g_i(b_i-a_i) \geq \bg^{\top}\by, \ee
	since $\by$ is a global minimizer.
	Since $b_i-a_i>0$ by assumption, we must have $g_i\geq0$.
	Similarly, if $y_i=b_i$ is to hold at a KKT point, we require $g_i\leq0$.
	
	If $g_i=0$, any value $y_i\in[a_i,b_i]$ gives the same objective value, so we can take $y_i=0$.
	Given this, which \algref{alg_geom_step} provides, the KKT conditions reduce to: there exist $\mu$ and $\lambda_i$ for $i=1,\ldots,n$ such that (for all $i=1,\ldots,n$, where relevant)
	\begin{subequations}
	\begin{align}
		\mu &\geq 0, \qquad \|\by\|^2 \leq \Delta^2, \\
		\lambda_i &\geq 0, \qquad a_i \leq y_i \leq b_i, \\
		0 &= \mu(\Delta^2-\|\by\|^2), \label{eq_geom_tr_complementarity}\\
		0 &= \begin{cases} \lambda_i(b_i-x_i), & g_i > 0, \\ \lambda_i(x_i-a_i), & g_i < 0,\end{cases} \\
		\bg &= -\mu\by + \sum_{i=1}^{n}\lambda_i \begin{cases}\bee_i, & g_i > 0, \\ -\bee_i, & g_i < 0.\end{cases} \label{eq_geom_kkt_grad}
	\end{align}
	\end{subequations}
	Since $y_i=0$ whenever $g_i=0$, and we always write $\bg$ as a sum of $\pm\bee_i$, we only need to check $\mu\geq0$ and $\lambda_i\geq 0$. 
	For convenience, let $\gamma_j\defeq\sum_{l=0}^{j}\beta_l$.
	
	Consider the start of some iteration $j$, where currently $y_{j,i} = -\gamma_{j-1} g_i$ for all $i\in\mathcal{I}$.
	The equation for $\alpha_j$ is then
	\be \sum_{i\notin\mathcal{I}}y_{j,i}^2 + \sum_{i\in\mathcal{I}}[(-\gamma_{j-1}-\alpha_j)g_i]^2 = \Delta^2,  \ee
	for which the largest solution is
	\be \alpha_j = -\gamma_{j-1} + \sqrt{\frac{\Delta^2 - \sum_{i\notin\mathcal{I}}y_{j,i}^2}{\sum_{i\in\mathcal{I}}g_i^2}}. \label{eq_talpha_relation} \ee
	Hence $y_{j,i} = -\t{\alpha}_j g_i$ for all $i\in\mathcal{I}$, where $\t{\alpha}_j\defeq\alpha_j+\gamma_{j-1}\geq 0$.
	
	If \algref{alg_geom_step} terminates at line \ref{ln_uncon_quit} for this iteration $j$, then we have $\mu=1/\t{\alpha}_j>0$, and $y_i\in\{a_i, b_i\}$ for all $i\notin\mathcal{I}$.
	Thus the $i$-th component of \eqref{eq_geom_kkt_grad} is satisfied for all $i\in\mathcal{I}$ with $\lambda_i=0$.
	
	Now consider $i\notin\mathcal{I}$, and suppose $g_i>0$.
	Then at some iteration $l<j$, we had $-\t{\alpha}_l g_i < a_i \leq 0$, and set $y_{l+1,i}=a_i$, by choosing $\beta_l$ so that $-\gamma_l g_i=a_i$. 
	Thus, we have $\gamma_l\geq 0$.
	Similarly, if $g_i<0$, we had $-\t{\alpha}_l g_i > b_i \geq 0$, and needed $\beta_l$ so that $-\gamma_l g_i = b_i \geq 0$, so we also need $\gamma_l\geq 0$ for this $l$.
	Thus $0 \leq \gamma_j\leq \t{\alpha}_j$ for all $j$.
	
	For termination at line \ref{ln_uncon_quit} of \algref{alg_geom_step} and for \eqref{eq_geom_kkt_grad} to hold with $\lambda_i\geq 0$, we need $-\t{\alpha}_j g_i \leq a_i$ or $-\t{\alpha}_j g_i \geq b_i$, depending on the sign of $g_i$.
	Equivalently, we need $\t{\alpha}_j\geq \gamma_l$ for all $l<j$.
	Since $\gamma_l \leq \t{\alpha}_l$, it suffices to show that $\t{\alpha}_{j-1}\leq \t{\alpha}_j$ for all $j$.
	
	To show this, let $\t{\by}_{j-2}$ be the vector with the fixed components of $\by_{j-2}$ (i.e.~those $i\notin\mathcal{I}$ at iteration $j-2$), and zeros otherwise. 
	Then if $i$ is the index removed from $\mathcal{I}$ at the end of iteration $j-1$, the equation for $\t{\alpha}_{j-1}$ is
	\begin{align}
		\|\by_{j-1}+\alpha_{j-1}\bs_{j-1}\|^2 &= \Delta^2, \\
		\|\t{\by}_{j-2}\|^2 + \t{\alpha}_{j-1}^2 g_i^2 + \t{\alpha}_{j-1}^2\|\bs_j\|^2 &= \Delta^2. \label{eq_talpha_jm1}
	\end{align}
	Similarly, the equation for $\t{\alpha}_j$ is
	\be \|\t{\by}_{j-2}\|^2 + c_i^2 + \t{\alpha}_{j-1}^2\|\bs_j\|^2 = \Delta^2, \label{eq_talpha_j} \ee
	where $c_i\in\{a_i, b_i\}$ is whichever value we fixed $y_{j-1,i}$ to be.
	Equating \eqref{eq_talpha_jm1} and \eqref{eq_talpha_j} and using $|\t{\alpha}_{j-1}g_i| \geq |c_i|$ (from above), we get
	\be \t{\alpha}_{j-1}^2 g_i^2 + \t{\alpha}_{j-1}^2\|\bs_j\|^2 = c_i^2 + \t{\alpha}_j^2\|\bs_j\|^2 \leq \t{\alpha}_{j-1}^2 g_i^2 + \t{\alpha}_j^2\|\bs_j\|^2, \ee
	and so using $\|\bs_j\|>0$ (from line \ref{ln_zero_quit} of \algref{alg_geom_step}) and $\t{\alpha}_j\geq0$ for all $j$ (shown above), we get $\t{\alpha}_{j-1}\leq\t{\alpha}_j$.
	This finishes the proof for when we terminate at line \ref{ln_uncon_quit} of \algref{alg_geom_step}.
	
	Now suppose we terminate at line \ref{ln_box_quit} of \algref{alg_geom_step}.
	Then we have fixed $y_i\in\{a_i,b_i\}$ for all $i$ (depending on the sign of $g_i$), with $\|\by\|\leq\Delta$ still valid.
	Hence we can have $\mu=0$ (satisfying \eqref{eq_geom_tr_complementarity}) and $\lambda_i=|g_i|$, and we have a KKT point.
\end{proof}

% \clearpage

\section{Test Problems} \label{sec_mw_problems}
\begin{table}[H]
	\centering
% 	\small{
	\footnotesize{
	\begin{tabular}{llcccc}
% 		\toprule
		\hline\noalign{\smallskip}
		\# & Objective Function & $n$ & $m$ & $2f(\b{x}_0)$ & $2f^*$ \\ \noalign{\smallskip}\hline\noalign{\smallskip} %\midrule
		1 & Linear (full rank) & 9 & 45 & 72 & 36 \\
		2 & Linear (full rank) & 9 & 45 & 1125 & 36 \\
		3 & Linear (rank 1) & 7 & 35 & $1.165420\times 10^7$ & 8.380282 \\
		4 & Linear (rank 1) & 7 & 35 & $1.168591\times 10^9$ & 8.380282 \\
		5 & Linear (rank 1 with zero row \& column) & 7 & 35 & $4.989195\times 10^6$ & 9.880597 \\
		6 & Linear (rank 1 with zero row \& column) & 7 & 35 & $5.009356\times 10^8$ & 9.880597 \\
		7 & Rosenbrock & 2 & 2 & 24.2 & 0 \\
		8 & Rosenbrock & 2 & 2 & $1.795769\times 10^6$ & 0 \\
		9 & Helical Valley & 3 & 3 & 2500 & 0 \\
		10 & Helical Valley & 3 & 3 & 10600 & 0 \\
		11 & Powell Singular & 4 & 4 & 215 & 0 \\
		12 & Powell Singular & 4 & 4 & $1.615400\times 10^6$ & 0 \\
		13 & Freudenstein \& Roth & 2 & 2 & 400.5 & 48.98425 \\
		14 & Freudenstein \& Roth & 2 & 2 & $1.545754\times 10^8$ & 48.98425 \\
		15 & Bard & 3 & 15 & 41.68170 & $8.214877\times 10^{-3}$ \\
		16 & Bard & 3 & 15 & 1306.234 & $8.214877\times 10^{-3}$ \\
		17 & Kowalik \& Osborne & 4 & 11 & $5.313172\times 10^{-3}$ & $3.075056\times 10^{-4}$ \\
		18 & Meyer & 3 & 16 & $1.693608\times 10^9$ & 87.94586 \\
		19 & Watson & 6 & 31 & 16.43083 & $2.287670\times 10^{-3}$ \\
		20 & Watson & 6 & 31 & $2.323367\times 10^6$ & $2.287670\times 10^{-3}$ \\
		21 & Watson & 9 & 31 & 26.90417 & $1.399760\times 10^{-6}$ \\
		22 & Watson & 9 & 31 & $8.158877\times 10^6$ & $1.399760\times 10^{-6}$ \\
		23 & Watson & 12 & 31 & 73.67821 & $4.722381\times 10^{-10}$ \\
		24 & Watson & 12 & 31 & $2.059384\times 10^7$ & $4.722381\times 10^{-10}$ \\
		25 & Box 3d & 3 & 10 & 1031.154 & 0 \\
		26 & Jennrich \& Sampson & 2 & 10 & 4171.306 & 124.3622 \\
		27 & Brown \& Dennis & 4 & 20 & $7.926693\times 10^6$ & $8.582220\times 10^4$ \\
		28 & Brown \& Dennis & 4 & 20 & $3.081064\times 10^{11}$ & $8.582220\times 10^4$ \\
		29 & Chebyquad & 6 & 6 & $4.642817\times 10^{-2}$ & 0 \\
		30 & Chebyquad & 7 & 7 & $3.377064\times 10^{-2}$ & 0 \\
		31 & Chebyquad & 8 & 8 & $3.861770\times 10^{-2}$ & $3.516874\times 10^{-3}$ \\
		32 & Chebyquad & 9 & 9 & $2.888298\times 10^{-2}$ & 0 \\
		33 & Chebyquad & 10 & 10 & $3.376327\times 10^{-2}$ & $4.772714\times 10^{-3}$ \\
		34 & Chebyquad & 11 & 11 & $2.674060\times 10^{-2}$ & $2.799762\times 10^{-3}$ \\
		35 & Brown almost-linear & 10 & 10 & 273.2480 & 0 \\
		36 & Osborne 1 & 5 & 33 & 16.17411 & $5.464895\times 10^{-5}$ \\
		37 & Osborne 2 & 11 & 65 & 2.093420 & $4.013774\times 10^{-2}$ \\
		38 & Osborne 2 & 11 & 65 & 199.6847 & $4.013774\times 10^{-2}$ \\
		39 & bdqrtic & 8 & 8 & 904 & 10.23897 \\
		40 & bdqrtic & 10 & 12 & 1356 & 18.28116 \\
		41 & bdqrtic & 11 & 14 & 1582 & 22.26059 \\
		42 & bdqrtic & 12 & 16 & 1808 & 26.27277 \\
		43 & Cube & 5 & 5 & 56.5 & 0 \\
		44 & Cube & 6 & 6 & 70.5625 & 0 \\
		45 & Cube & 8 & 8 & 98.6875 & 0 \\
		46 & Mancino & 5 & 5 & $2.539084\times 10^9$ & 0 \\
		47 & Mancino & 5 & 5 & $6.873795\times 10^{12}$ & 0 \\
		48 & Mancino & 8 & 8 & $3.367961\times 10^9$ & 0 \\
		49 & Mancino & 10 & 10 & $3.735127\times 10^9$ & 0 \\
		50 & Mancino & 12 & 12 & $3.991072\times 10^9$ & 0 \\
		51 & Mancino & 12 & 12 & $1.130015\times 10^{13}$ & 0 \\
		52 & Heart8ls & 8 & 8 & 9.385672 & 0 \\
		53 & Heart8ls & 8 & 8 & $3.365815\times 10^{10}$ & 0 \\
% 		\bottomrule
		\noalign{\smallskip}\hline
	\end{tabular}}
	\caption{Details of test problems from \cite{More2009}, including the value of $f^*$ used in \eqref{eq_solved_threshold} for each problem. Note that, in line with the implementation of DFO-GN, we show $2f(\b{x}_0)$ and $2f^*$; i.e.~excluding the $1/2$ factor in \eqref{eq_ls_definition}.}
	\label{tab_more_wild_values}
\end{table}

\section{Medium-Scale Test Problems} \label{sec_cutest_problems}
\begin{table}[H]
	\centering
	\scriptsize{
% 	\footnotesize{
	%\small{
	\begin{tabular}{rlccccc} % convention: show f0 and f* to 7 sig figs
% 		\toprule
		\hline\noalign{\smallskip}
		\# & Problem & $n$ & $m$ & $2f(\b{x}_0)$ & $2f^*$ & Parameters \\ \noalign{\smallskip}\hline\noalign{\smallskip} %\midrule
		1 & ARGLALE & 100 & 400 & 700 & 300 & $N=100$ \\
		2 & ARGLBLE & 100 & 400 & $5.460944\times 10^{14}$ & 99.62547 & $N=100$ \\
		3 & ARGTRIG & 100 & 100 & 32.99641 & 0 & $N=100$ \\
		4 & ARTIF & 100 & 100 & 36.59115 & 0 & $N=100$ \\
		5 & ARWHDNE & 100 & 198 & 495 & 27.66203 & $N=100$ \\
		6 & BDVALUES & 100 & 100 & $1.943417\times 10^{7}$ & 0 & $NDP=102$ \\
		7 & BRATU2D & 64 & 64 & 0.1560738 & 0 & $P=10$ \\
		8 & BRATU2DT & 64 & 64 & 0.4521311 & $1.853474\times 10^{-5}$ & $P=10$ \\
		9 & BRATU3D & 27 & 27 & 4.888529 & 0 & $P=5$ \\
		10 & BROWNALE & 100 & 100 & $2.524757\times 10^{5}$ & 0 & $N=100$ \\
		11 & BROYDN3D & 100 & 100 & 111 & 0 & $N=100$ \\
		12 & BROYDNBD & 100 & 100 & 2404 & 0 & $N=100$ \\
		13 & CBRATU2D & 50 & 50 & 0.4822531 & 0 & $P=7$ \\
		14 & CHANDHEQ* & 100 & 100 & 6.923365 & 0 & $N=100$ \\
		15 & CHEMRCTA* & 100 & 100 & 3.0935 & 0 & $N=50$ \\
		16 & CHEMRCTB* & 100 & 100 & 1.446513 & $1.404424\times 10^{-3}$ & $N=100$ \\
		17 & CHNRSBNE & 50 & 98 & 7635.84 & 0 & $N=50$ \\
		18 & DRCAVTY1 & 100 & 100 & 0.4513889 & 0 & $M=10$ \\
		19 & DRCAVTY2 & 100 & 100 & 0.4513889 & $5.449602\times 10^{-3}$ & $M=10$ \\
		20 & DRCAVTY3 & 100 & 100 & 0.4513889 & 0 & $M=10$ \\
		21 & EIGENA* & 110 & 110 & 285 & 0 & $N=10$ \\
		22 & EIGENB & 110 & 110 & 19 & 0 & $N=10$ \\
		23 & FLOSP2HH & 59 & 59 & 519 & 0.3333333 & $M=2$ \\
		24 & FLOSP2HL & 59 & 59 & 519 & 0.3333333 & $M=2$ \\
		25 & FLOSP2HM & 59 & 59 & 519 & 0.3333333 & $M=2$ \\
		26 & FLOSP2TH & 59 & 59 & 516 & 0 & $M=2$ \\
		27 & FLOSP2TL & 59 & 59 & 516 & 0 & $M=2$ \\
		28 & FLOSP2TM & 59 & 59 & 516 & 0 & $M=2$ \\
		29 & FREURONE & 100 & 198 & $9.95565\times 10^{4}$ & $1.196458\times 10^{4}$ & $N=100$ \\
		30 & HATFLDG & 25 & 25 & 27 & 0 & --- \\
		31 & HYDCAR20 & 99 & 99 & 1341.663 & 0 & --- \\
		32 & HYDCAR6 & 29 & 29 & 704.1073 & 0 & --- \\
		33 & INTEGREQ & 100 & 100 & 0.5730503 & 0 & $N=100$ \\
		34 & METHANB8 & 31 & 31 & 1.043105 & 0 & --- \\
		35 & METHANL8 & 31 & 31 & 4345.100 & 0 & --- \\
		36 & MOREBVNE & 100 & 100 & $3.633100\times 10^{-4}$ & 0 & $N=100$ \\
		37 & MSQRTA & 100 & 100 & 212.7162 & 0 & $P=10$ \\
		38 & MSQRTB & 100 & 100 & 205.0846 & 0 & $P=10$ \\
		39 & OSCIGRNE & 100 & 100 & $6.120720\times 10^{8}$ & 0 & $N=100$ \\
		40 & PENLT1NE & 100 & 101 & $1.144806\times 10^{11}$ & $9.025000\times 10^{-9}$ & $N=100$ \\
		41 & PENLT2NE & 100 & 200 & $1.591383\times 10^{6}$ & 0.9809377 & $N=100$ \\
		42 & POWELLSE & 100 & 100 & 41875 & 0 & $N=100$ \\
		43 & QR3D* & 40 & 40 & 1.2 & 0 & $M=5$ \\
		44 & QR3DBD* & 37 & 40 & 1.2 & 0 & $M=5$ \\
		45 & SEMICN2U & 100 & 100 & $2.025037\times 10^{4}$ & 0 & $(N,LN) = (100, 90)$ \\
		46 & SEMICON2* & 100 & 100 & $2.025037\times 10^{4}$ & 0 & $(N,LN) = (100, 90)$ \\
		47 & SPMSQRT & 100 & 164 & 74.33542 & 0 & $M=34$ \\
		48 & VARDIMNE & 100 & 102 & $1.310584\times 10^{14}$ & 0 & $N=100$ \\
		49 & WATSONNE & 31 & 31 & 30 & 0 & $N=31$ \\
		50 & YATP1SQ & 120 & 120 & $2.073643\times 10^{6}$ & 0 & $N=10$ \\
		51 & YATP2SQ & 120 & 120 & $1.831687\times 10^{5}$ & 0 & $N=10$ \\
		52 & LUKSAN11 & 100 & 198 & 626.0640 & 0 & --- \\
		53 & LUKSAN12 & 98 & 192 & $3.2160\times 10^{4}$ & 4292.197 & --- \\
		54 & LUKSAN13 & 98 & 224 & $6.4352\times 10^{4}$ & $2.518886\times 10^{4}$ & --- \\
		55 & LUKSAN14 & 98 & 224 & $2.6880\times 10^{4}$ & 123.9235 & --- \\
		56 & LUKSAN15 & 100 & 196 & $2.701585\times 10^{4}$ & 3.569697 & --- \\
		57 & LUKSAN16 & 100 & 196 & $1.306848\times 10^{4}$ & 3.569697 & --- \\
		58 & LUKSAN17 & 100 & 196 & $1.687370\times 10^{6}$ & 0.4931613 & --- \\
		59 & LUKSAN21 & 100 & 100 & 99.98751 & 0 & --- \\
		60 & LUKSAN22 & 100 & 198 & $2.487686\times 10^{4}$ & 872.9230 & --- \\
% 		\bottomrule
		\noalign{\smallskip}\hline
	\end{tabular}
	} % font size
	\caption{Details of medium-scale test problems from the CUTEst test set (showing $2f(\b{x}_0)$ and $2f^*$, as per \tabref{tab_more_wild_values}). The set of problems are taken primarily from \cite{Gould2012,More1981,Luksan1996}. Some problems are variable-dimensional; the relevant parameters yielding the given $(n,m)$ are provided. Problems marked * have box constraints. The value of $n$ shown excludes fixed variables.}
	\label{tab_cutest_problems}
\end{table}

\clearpage

\section{Numerical Results for Higher Accuracy Levels} \label{sec_numerics_high_accuracy}
To supplement the results provided in \secref{sec_numerics}, we provide here a number of the same comparisons of solvers, but for higher accuracy levels $\tau\in\{10^{-7},10^{-9},10^{-11}\}$, compared to $\tau=10^{-5}$ in the main text.
For each set of plots below, the title indicates the content of the plots and the figure in the main text with the corresponding results for $\tau=10^{-5}$.

\subsection{Mor\'e \& Wild test set --- smooth objectives (compare to \figref{fig_smooth})}
\begin{figure}[H]
	\centering
	\begin{subfigure}[b]{0.48\textwidth}
		\includegraphics[width=\textwidth]{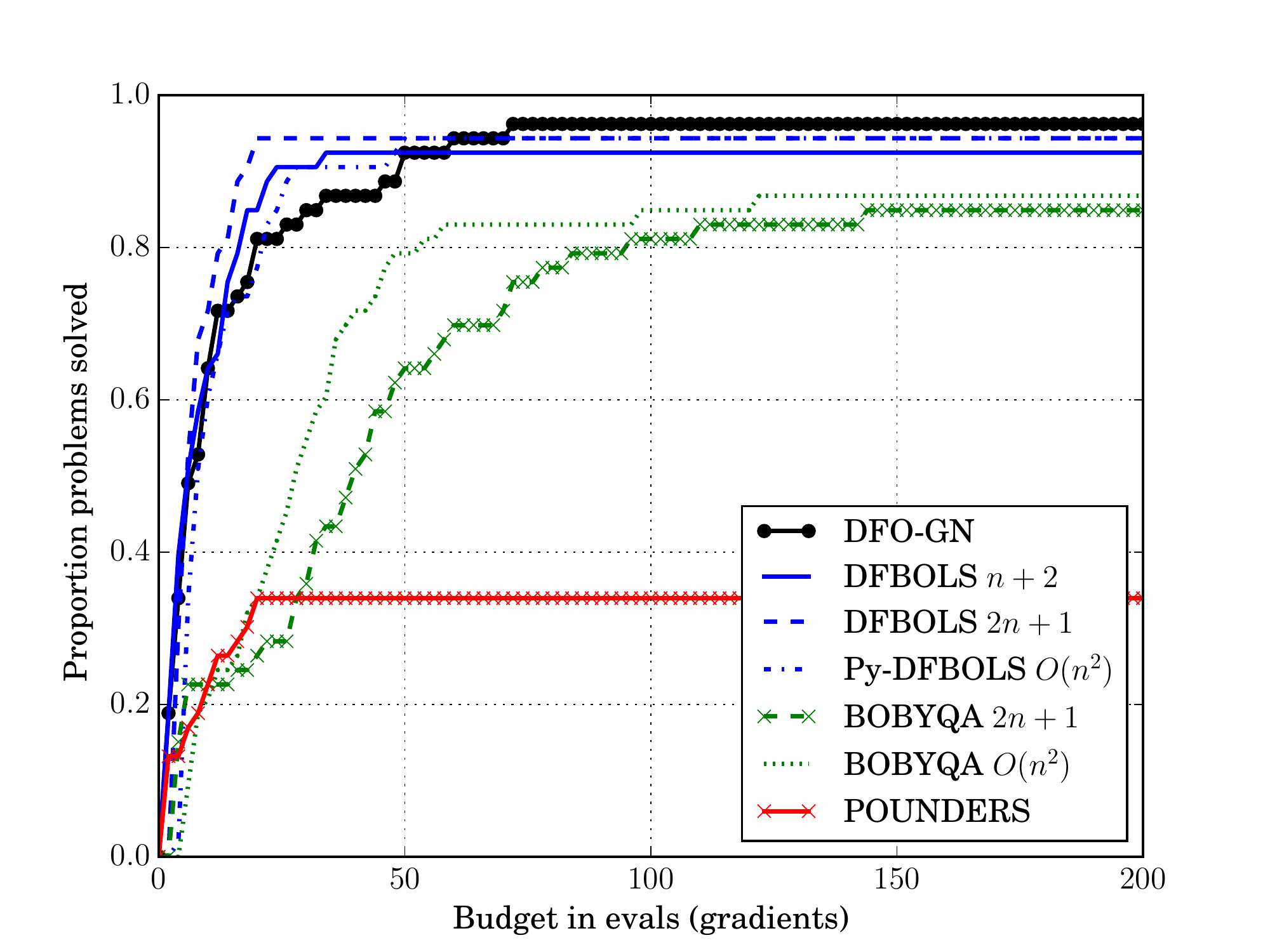}
		\caption{Data profile, $\tau=10^{-7}$}
	\end{subfigure}
	~
	\begin{subfigure}[b]{0.48\textwidth}
		\includegraphics[width=\textwidth]{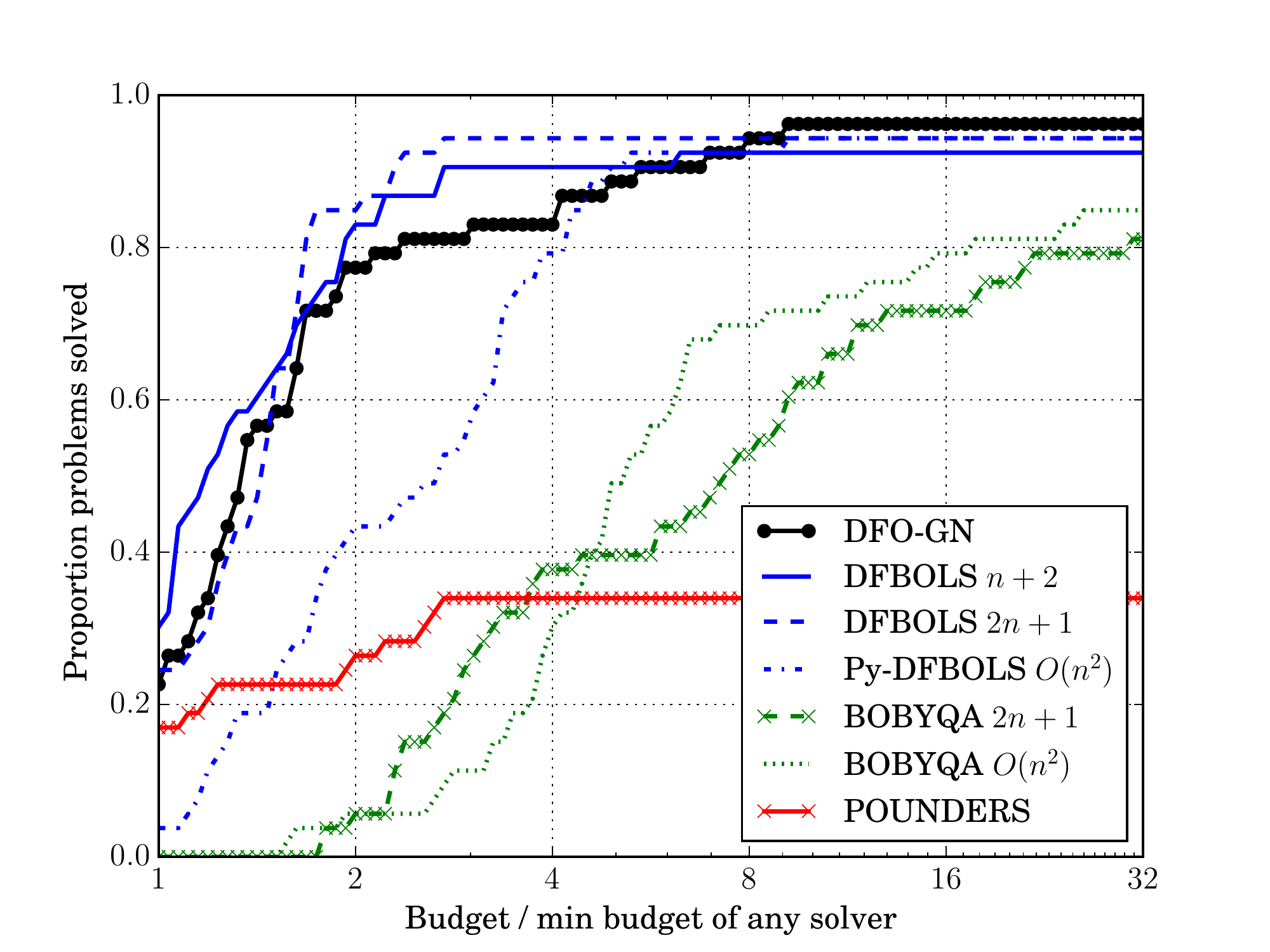}
		\caption{Perf profile, $\tau=10^{-7}$}
	\end{subfigure}
	\\
	\begin{subfigure}[b]{0.48\textwidth}
		\includegraphics[width=\textwidth]{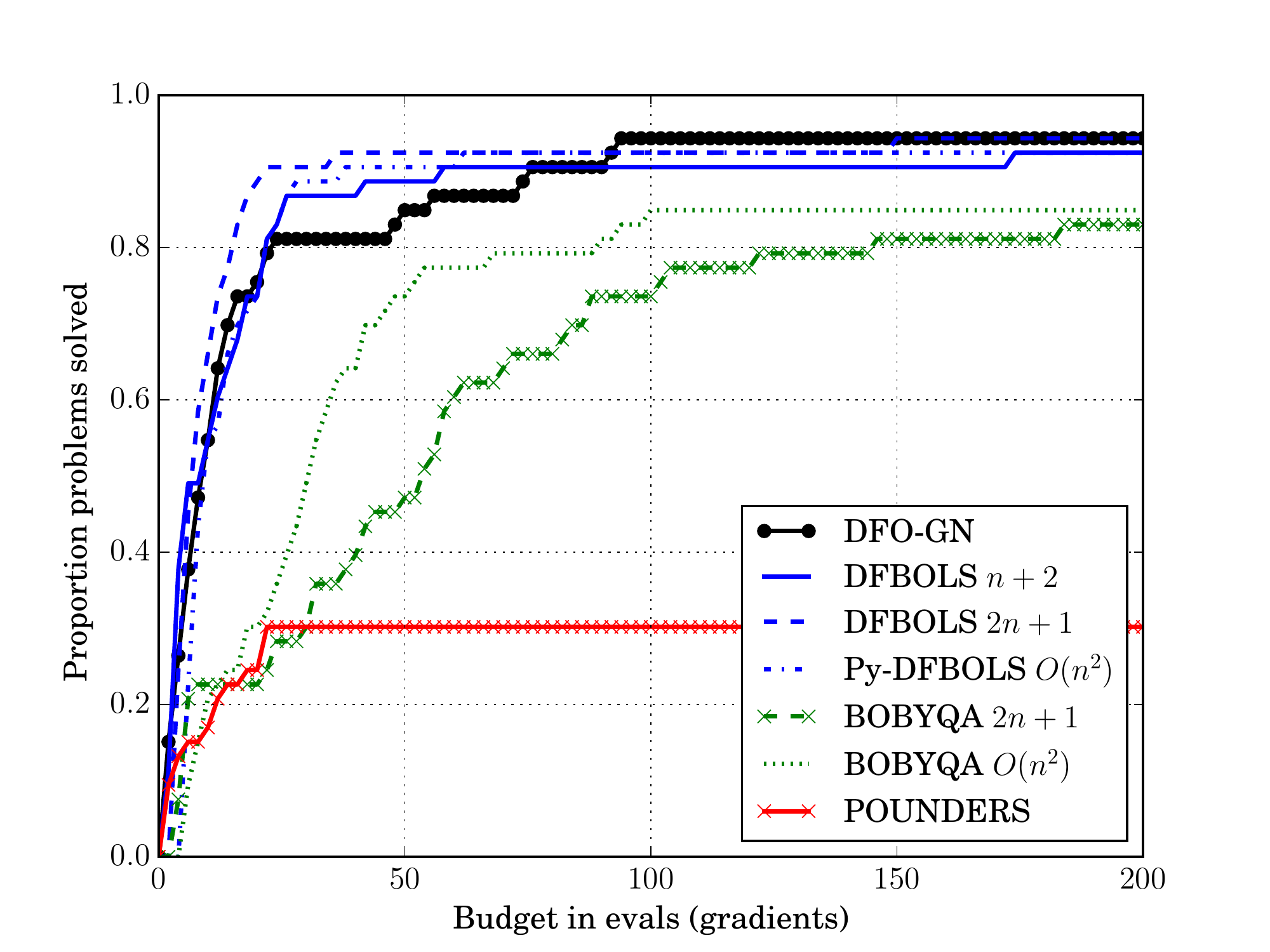}
		\caption{Data profile, $\tau=10^{-9}$}
	\end{subfigure}
	~
	\begin{subfigure}[b]{0.48\textwidth}
		\includegraphics[width=\textwidth]{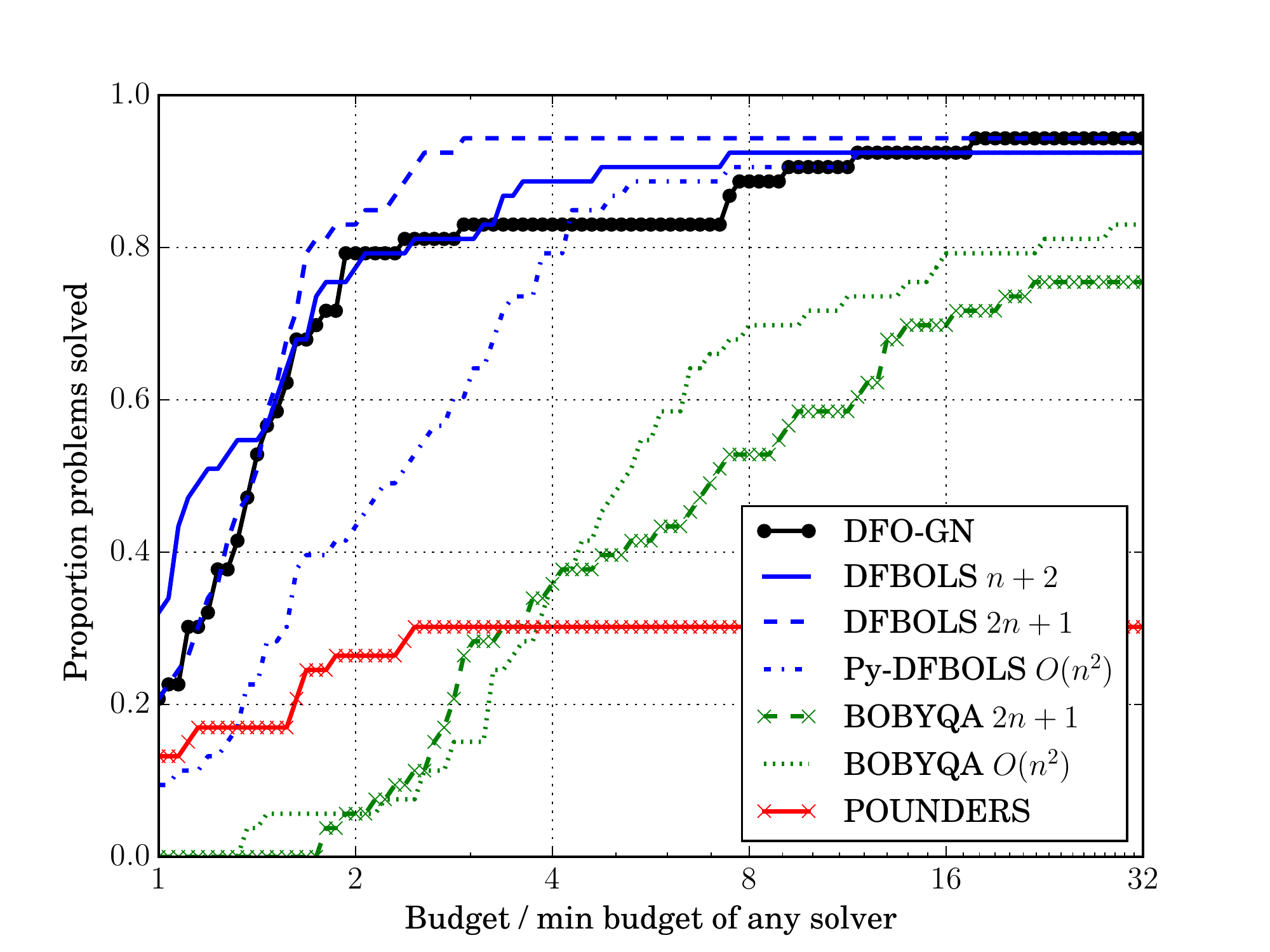}
		\caption{Perf profile, $\tau=10^{-9}$}
	\end{subfigure}
	\\
	\begin{subfigure}[b]{0.48\textwidth}
		\includegraphics[width=\textwidth]{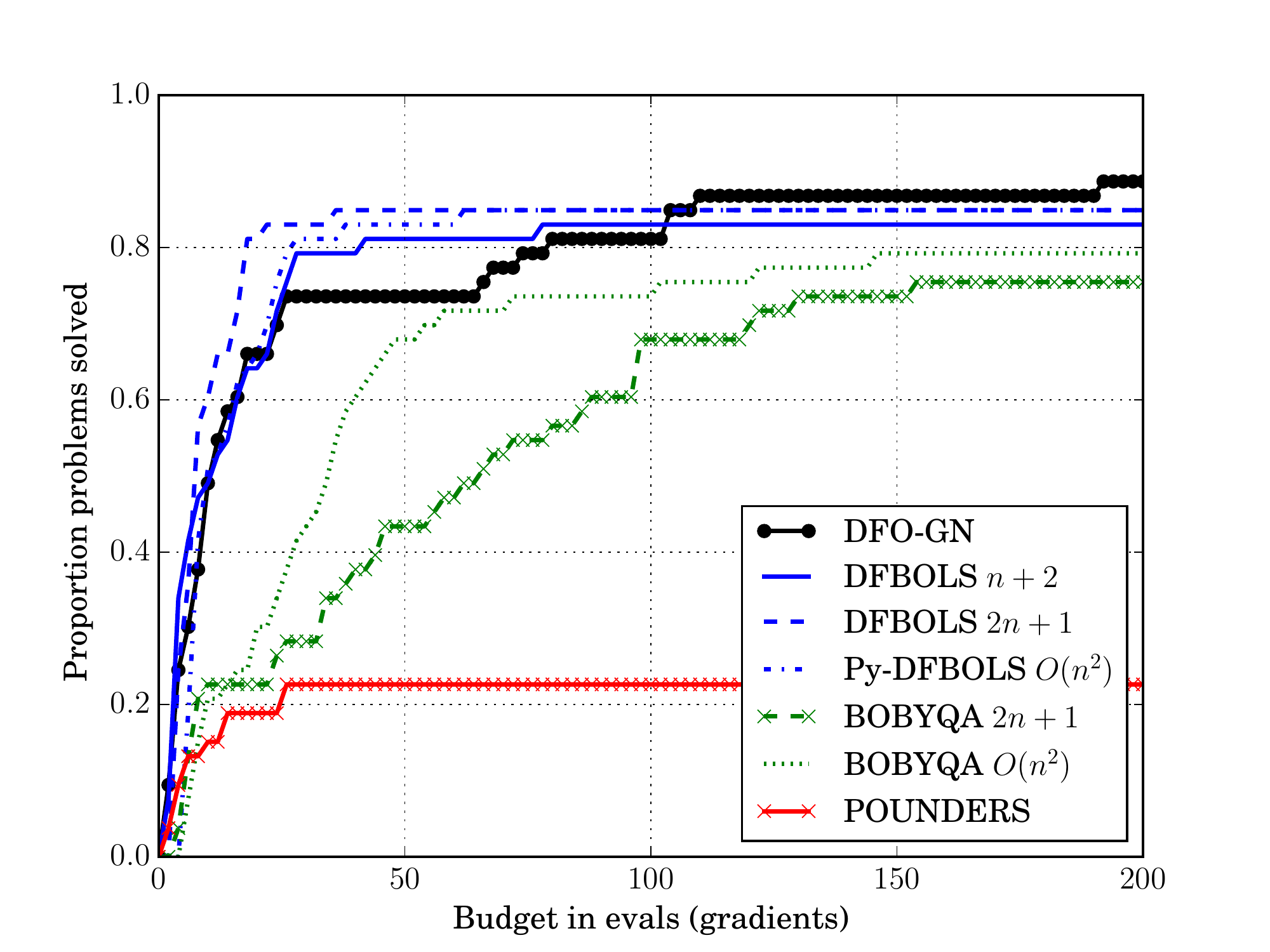}
		\caption{Data profile, $\tau=10^{-11}$}
	\end{subfigure}
	~
	\begin{subfigure}[b]{0.48\textwidth}
		\includegraphics[width=\textwidth]{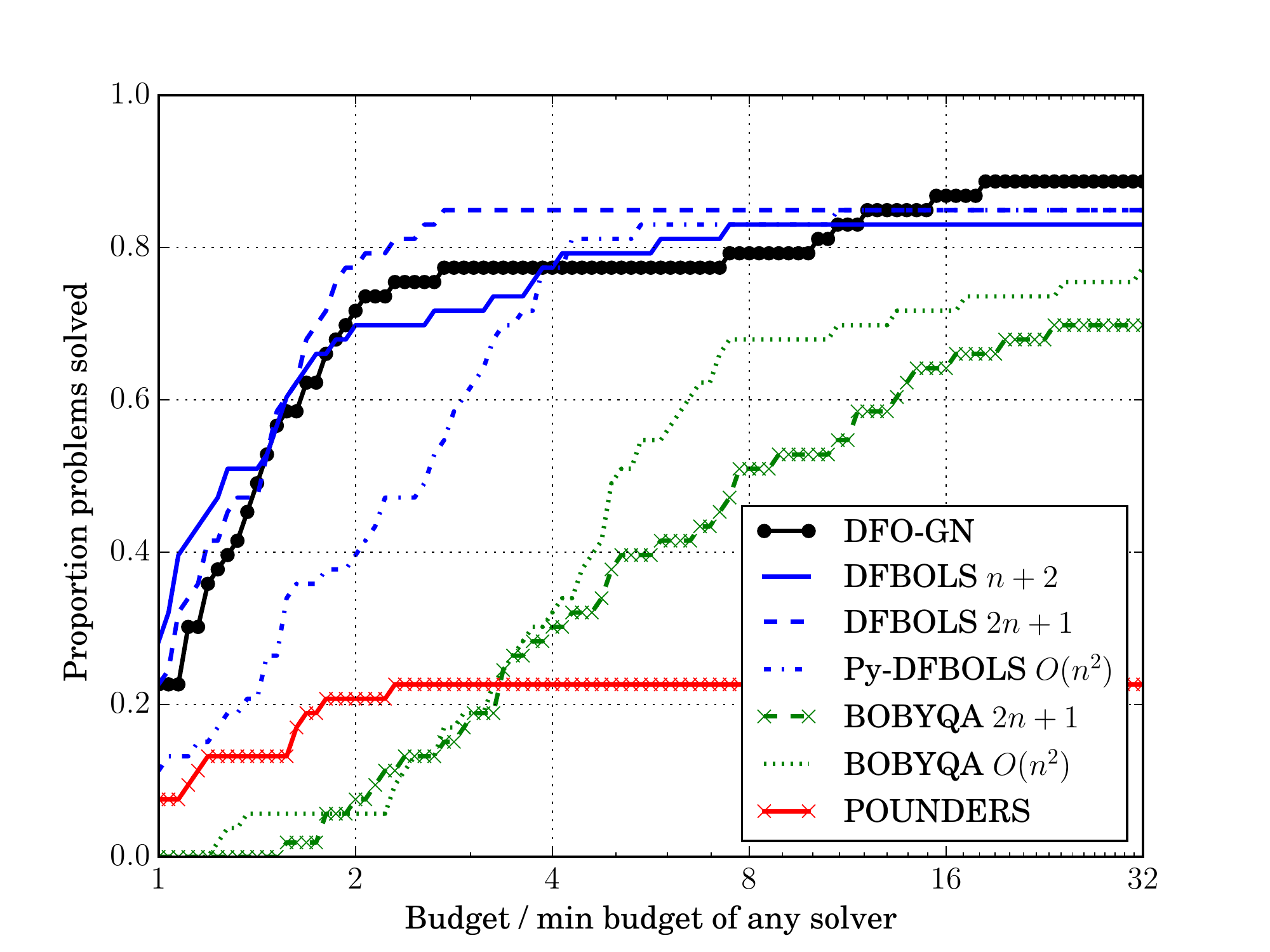}
		\caption{Perf profile, $\tau=10^{-11}$}
	\end{subfigure}
	\caption{Comparison of DFO-GN with BOBYQA, DFBOLS and POUNDERS for smooth objectives, to accuracy $\tau\in\{10^{-7},10^{-9},10^{-11}\}$. For the BOBYQA and DFBOLS runs, $n+2$, $2n+1$ and $\bigO(n^2)=(n+1)(n+2)/2$ are the number of interpolation points.}
	\label{fig_smooth_high_accuracy}
\end{figure}

\subsection{Mor\'e \& Wild test set --- noisy objectives for $\tau=10^{-7}$ (compare to \figref{fig_noisy})}
\begin{figure}[H]
	\centering
	\begin{subfigure}[b]{0.48\textwidth}
		\includegraphics[width=\textwidth]{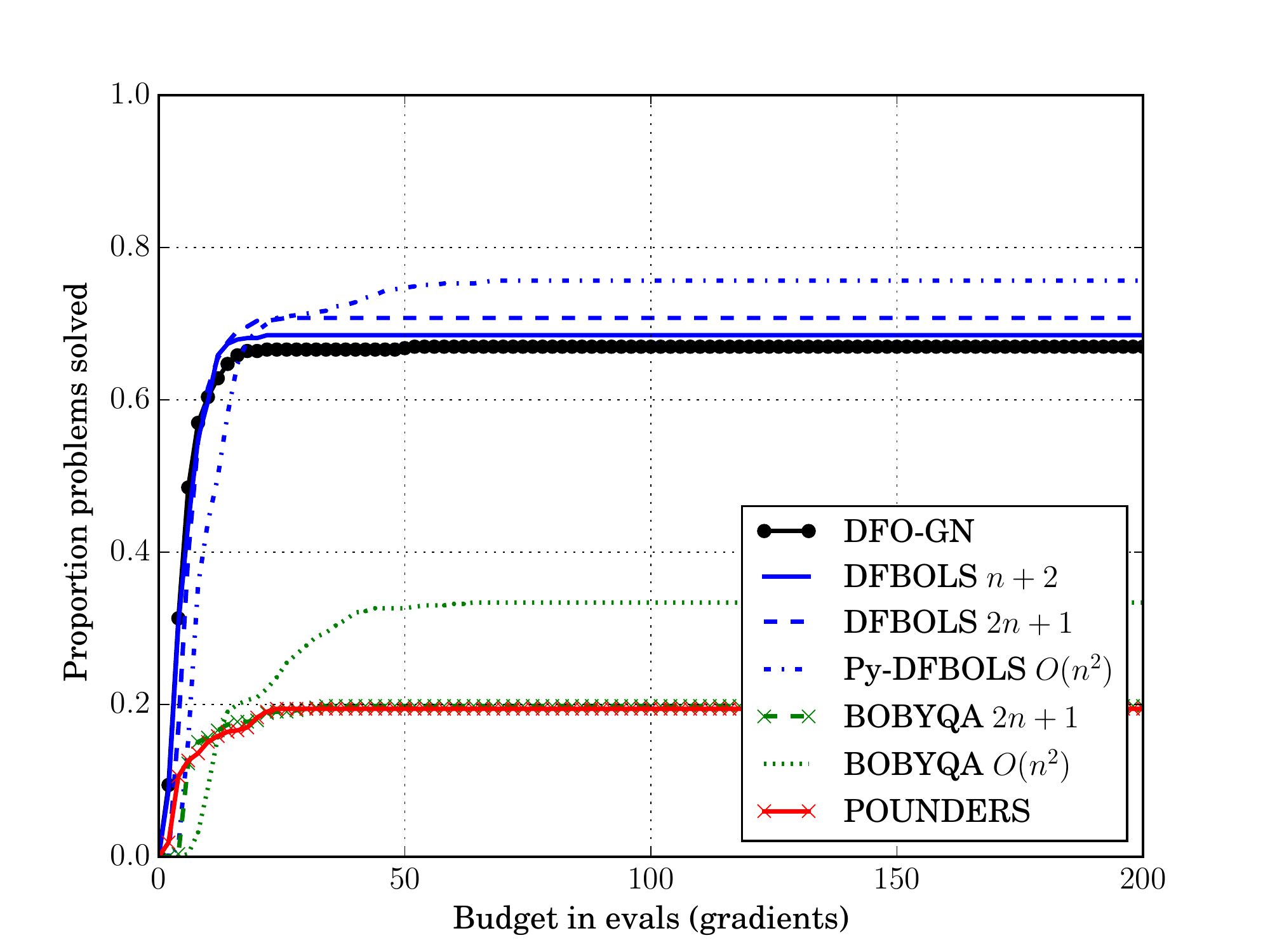}
		\caption{Mult.~Gaussian, data profile}
	\end{subfigure}
	~
	\begin{subfigure}[b]{0.48\textwidth}
		\includegraphics[width=\textwidth]{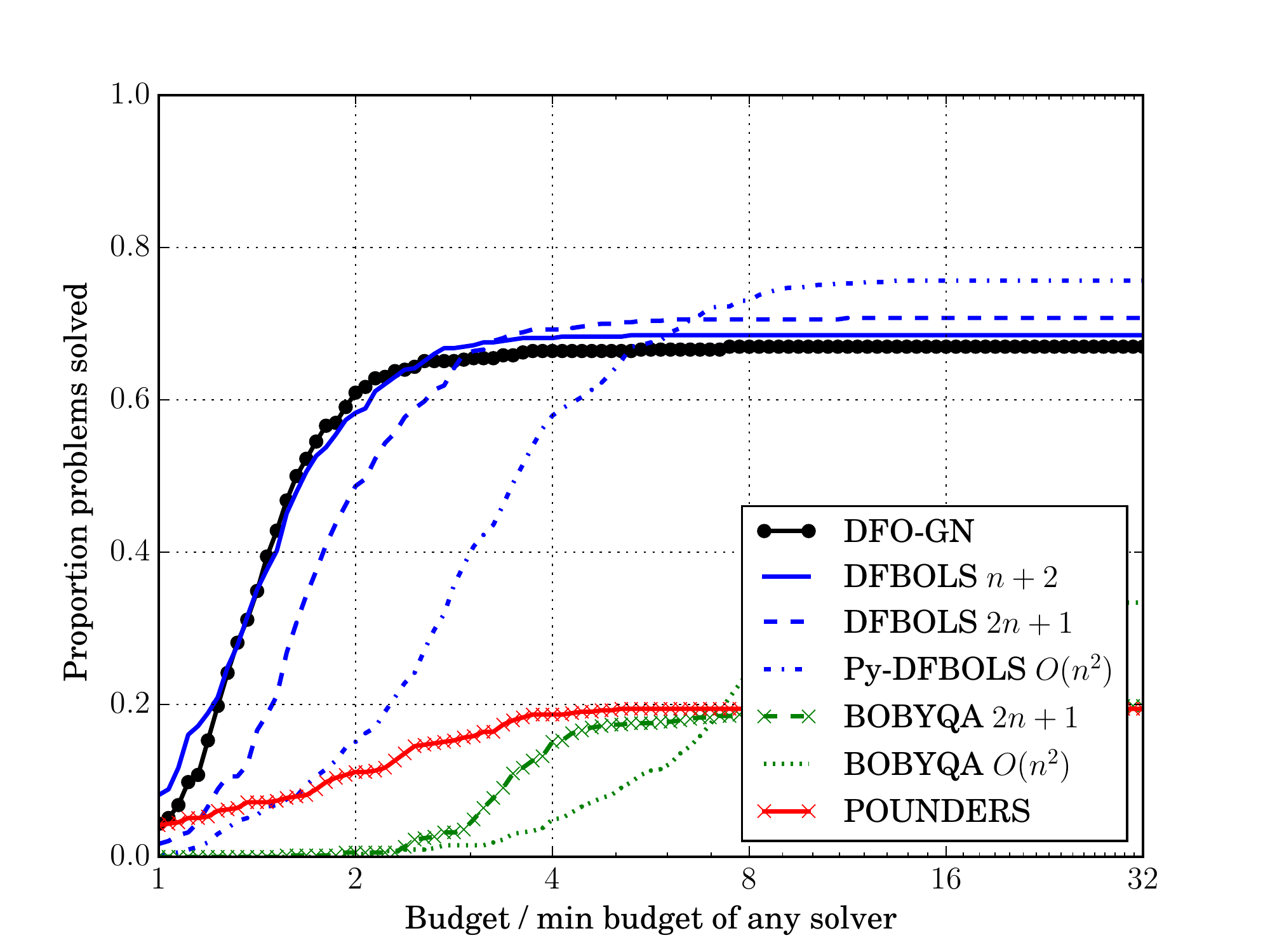}
		\caption{Mult.~Gaussian, performance profile}
	\end{subfigure}
	\\
	\begin{subfigure}[b]{0.48\textwidth}
		\includegraphics[width=\textwidth]{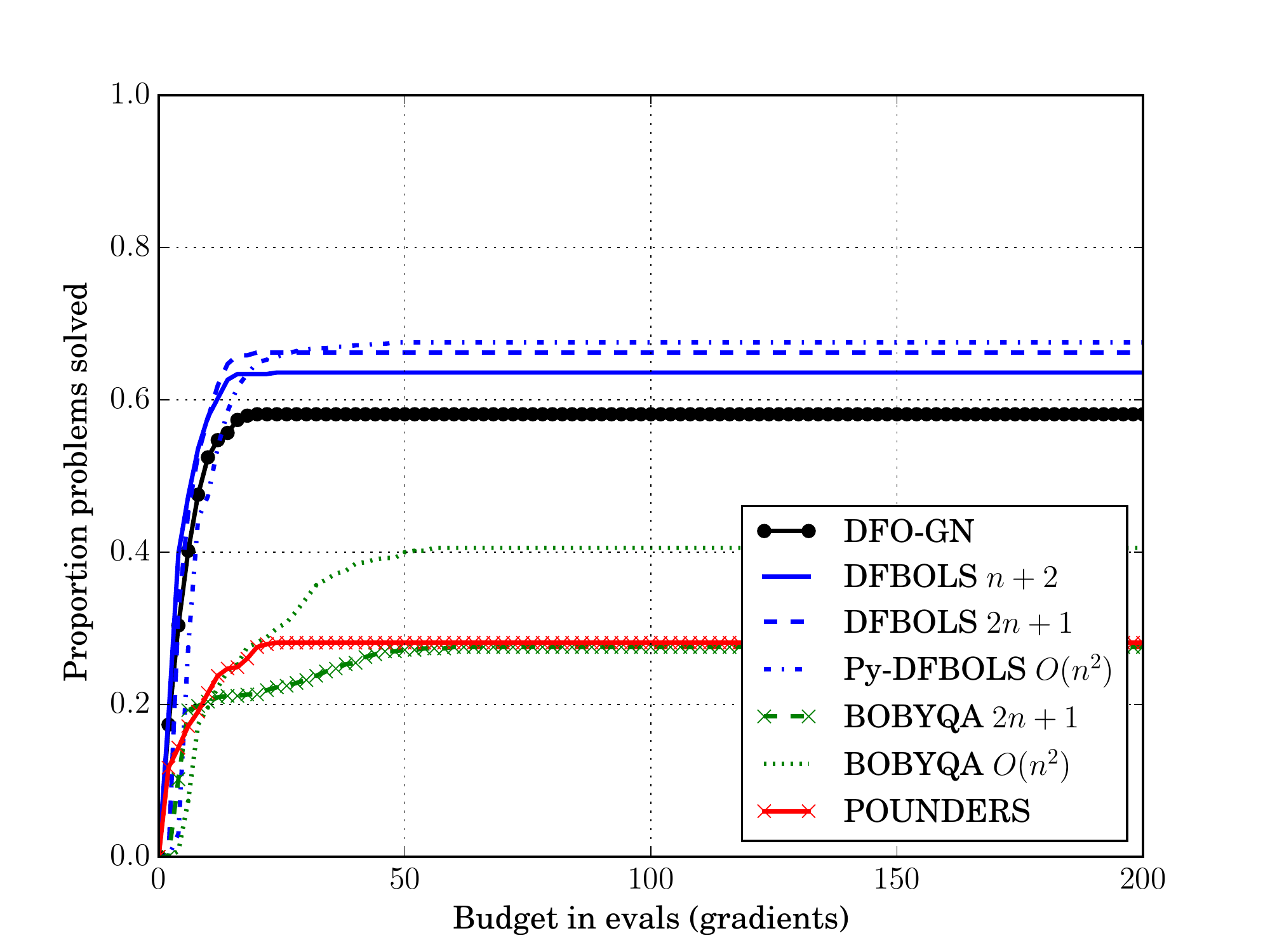}
		\caption{Add.~Gaussian, data profile}
	\end{subfigure}
	~
	\begin{subfigure}[b]{0.48\textwidth}
		\includegraphics[width=\textwidth]{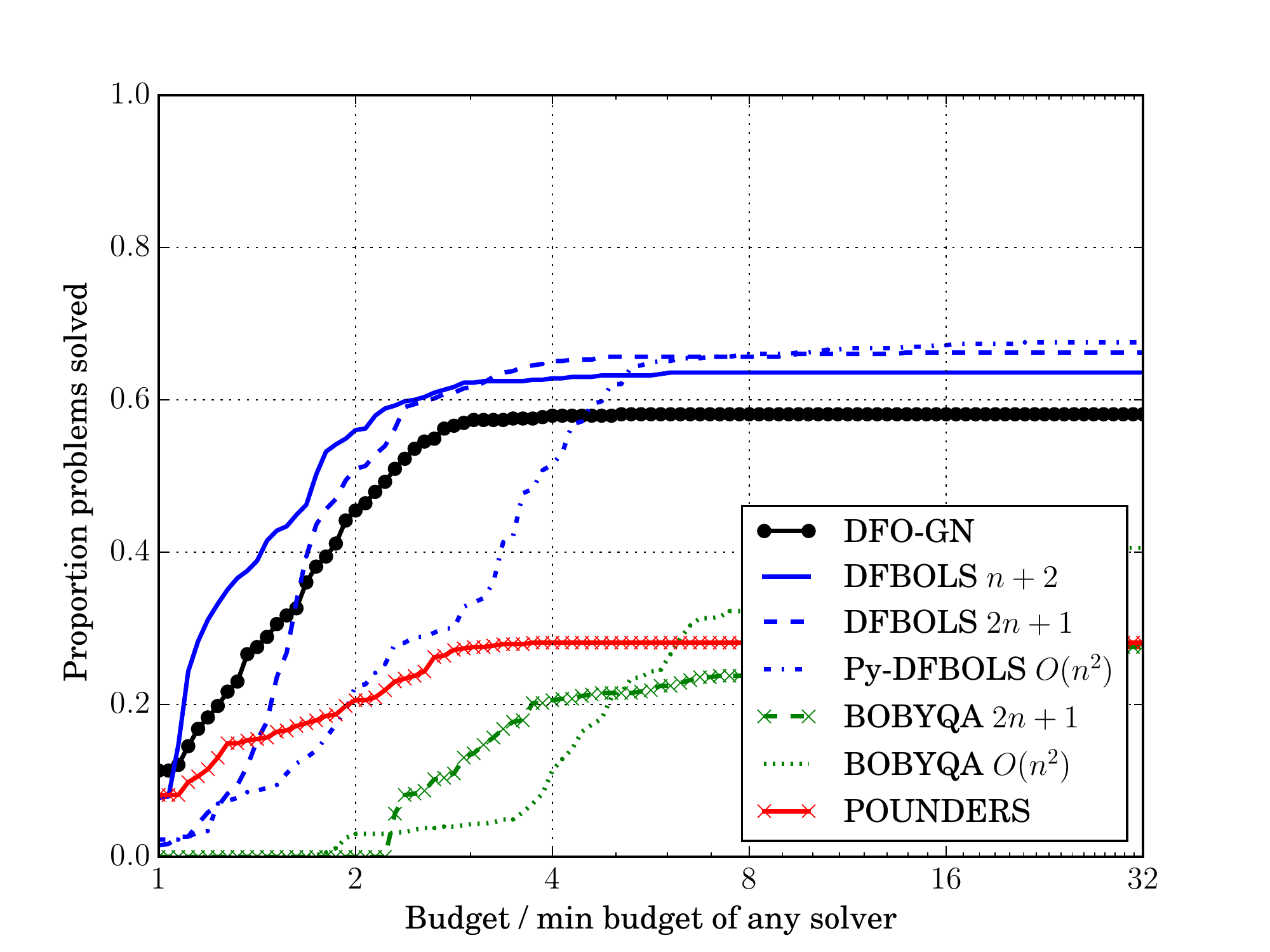}
		\caption{Add.~Gaussian, performance profile}
	\end{subfigure}
	\\
	\begin{subfigure}[b]{0.48\textwidth}
		\includegraphics[width=\textwidth]{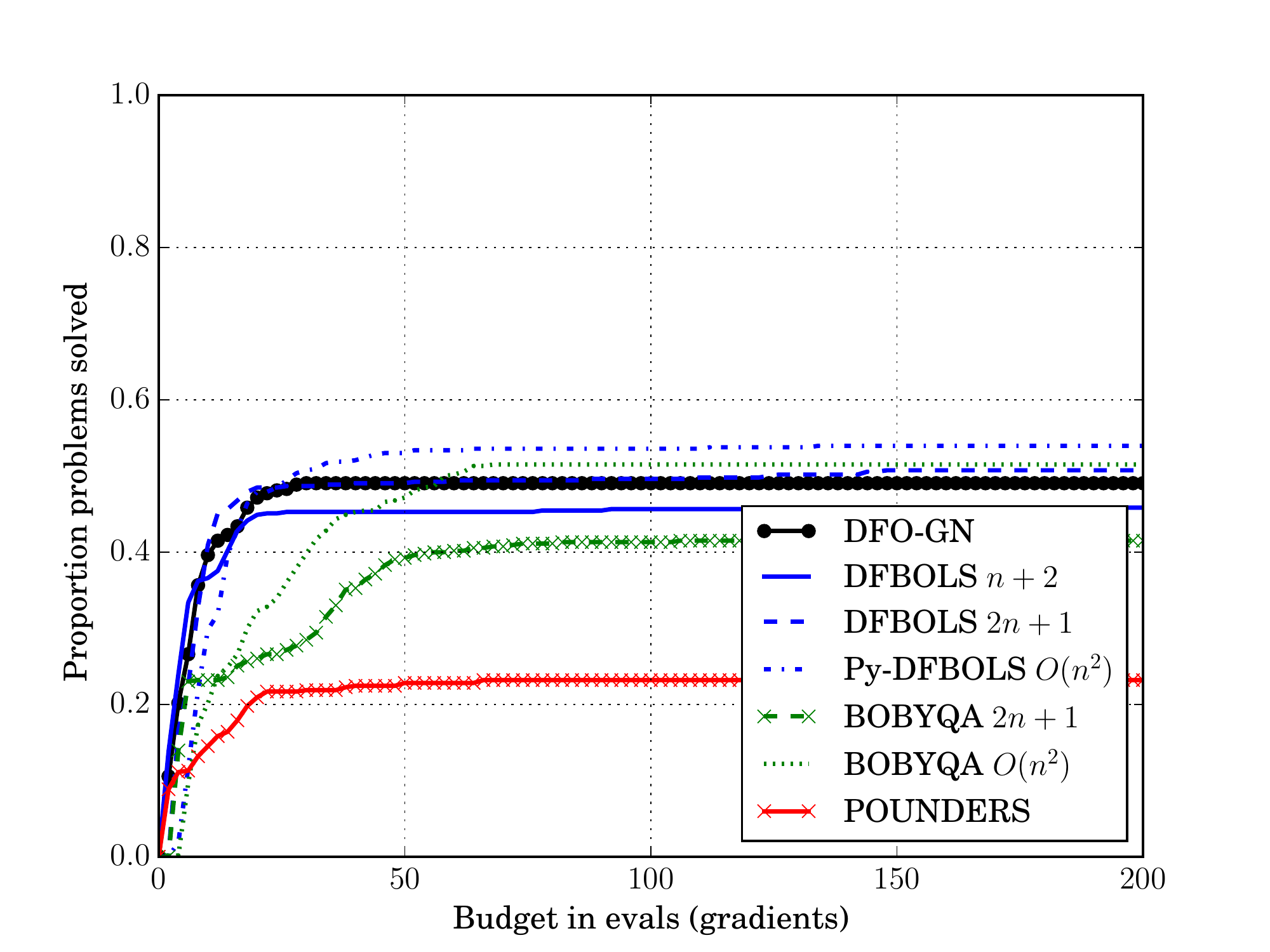}
		\caption{Add.~$\chi^2$, data profile}
	\end{subfigure}
	~
	\begin{subfigure}[b]{0.48\textwidth}
		\includegraphics[width=\textwidth]{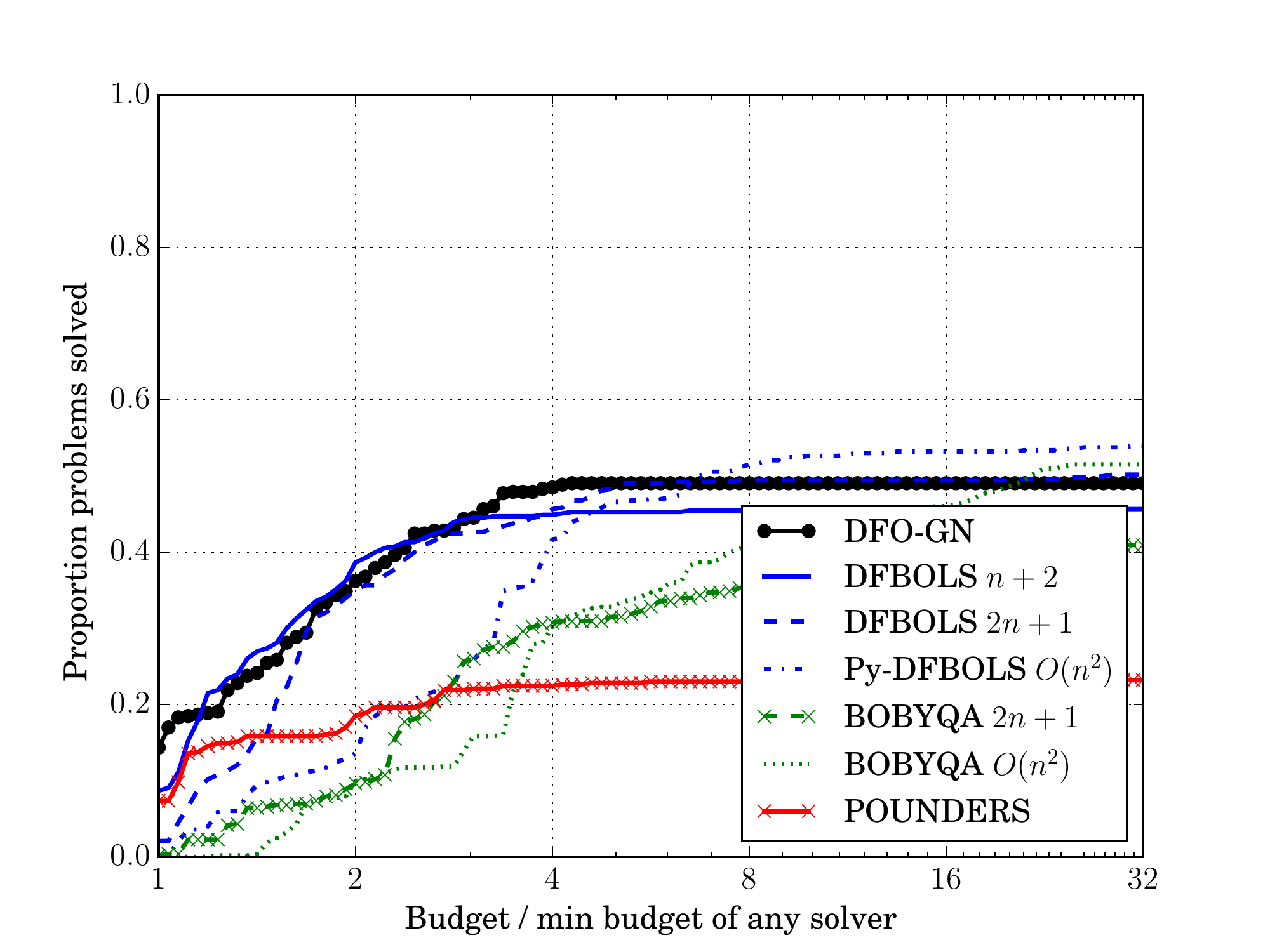}
		\caption{Add.~$\chi^2$, performance profile}
	\end{subfigure}
	\caption{Comparison of DFO-GN with BOBYQA, DFBOLS and POUNDERS for objectives with multiplicative Gaussian, additive Gaussian and additive $\chi^2$ noise with $\sigma=10^{-2}$, to accuracy $\tau=10^{-7}$ (average of 10 runs for each solver). For the BOBYQA and DFBOLS runs, $n+2$, $2n+1$ and $\bigO(n^2)=(n+1)(n+2)/2$ are the number of interpolation points.}
	\label{fig_noisy_high_accuracy_tau7}
\end{figure}

\subsection{Mor\'e \& Wild test set --- noisy objectives for $\tau=10^{-9}$ (compare to \figref{fig_noisy})}
\begin{figure}[H]
	\centering
	\begin{subfigure}[b]{0.48\textwidth}
		\includegraphics[width=\textwidth]{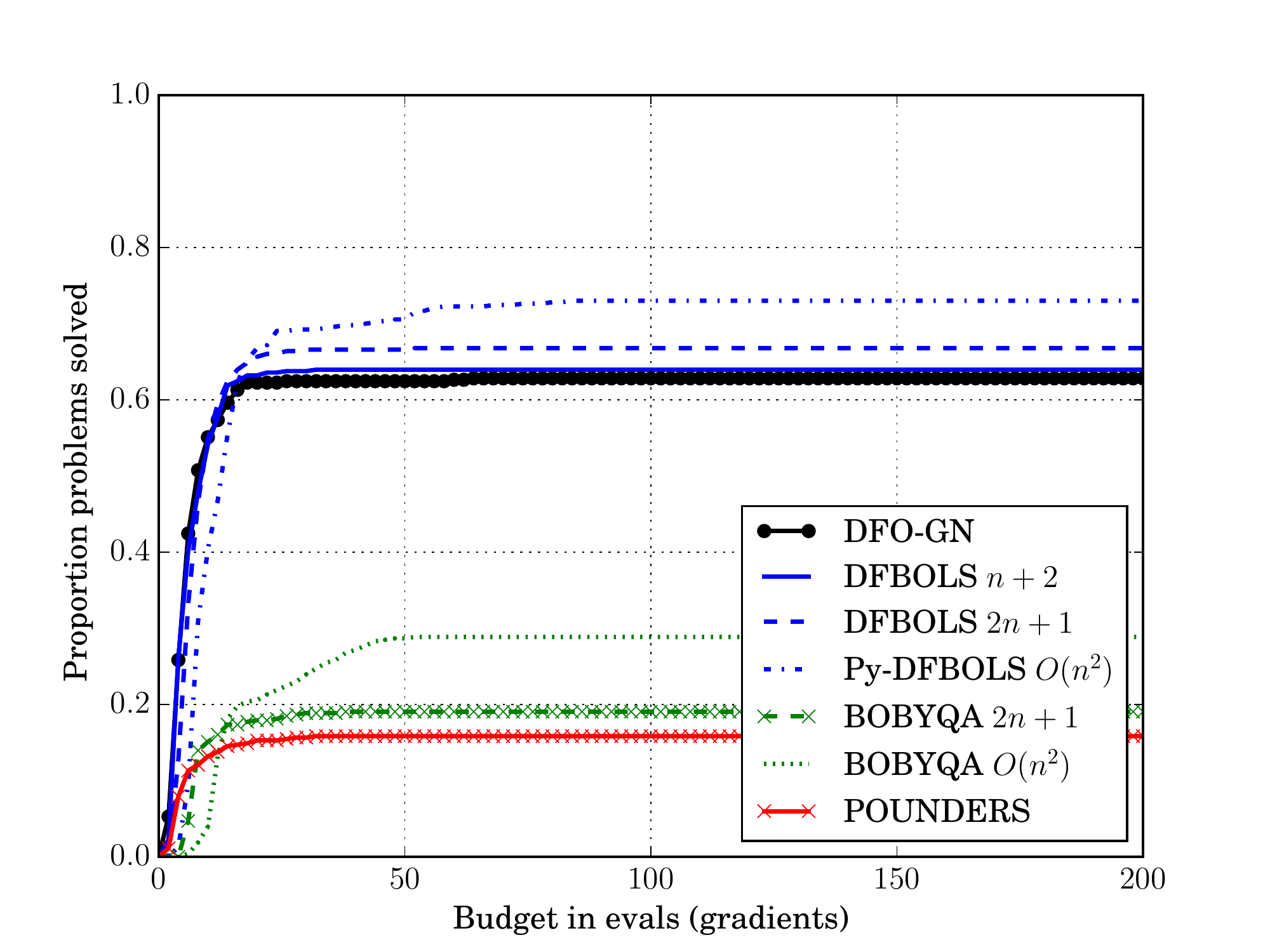}
		\caption{Mult.~Gaussian, data profile}
	\end{subfigure}
	~
	\begin{subfigure}[b]{0.48\textwidth}
		\includegraphics[width=\textwidth]{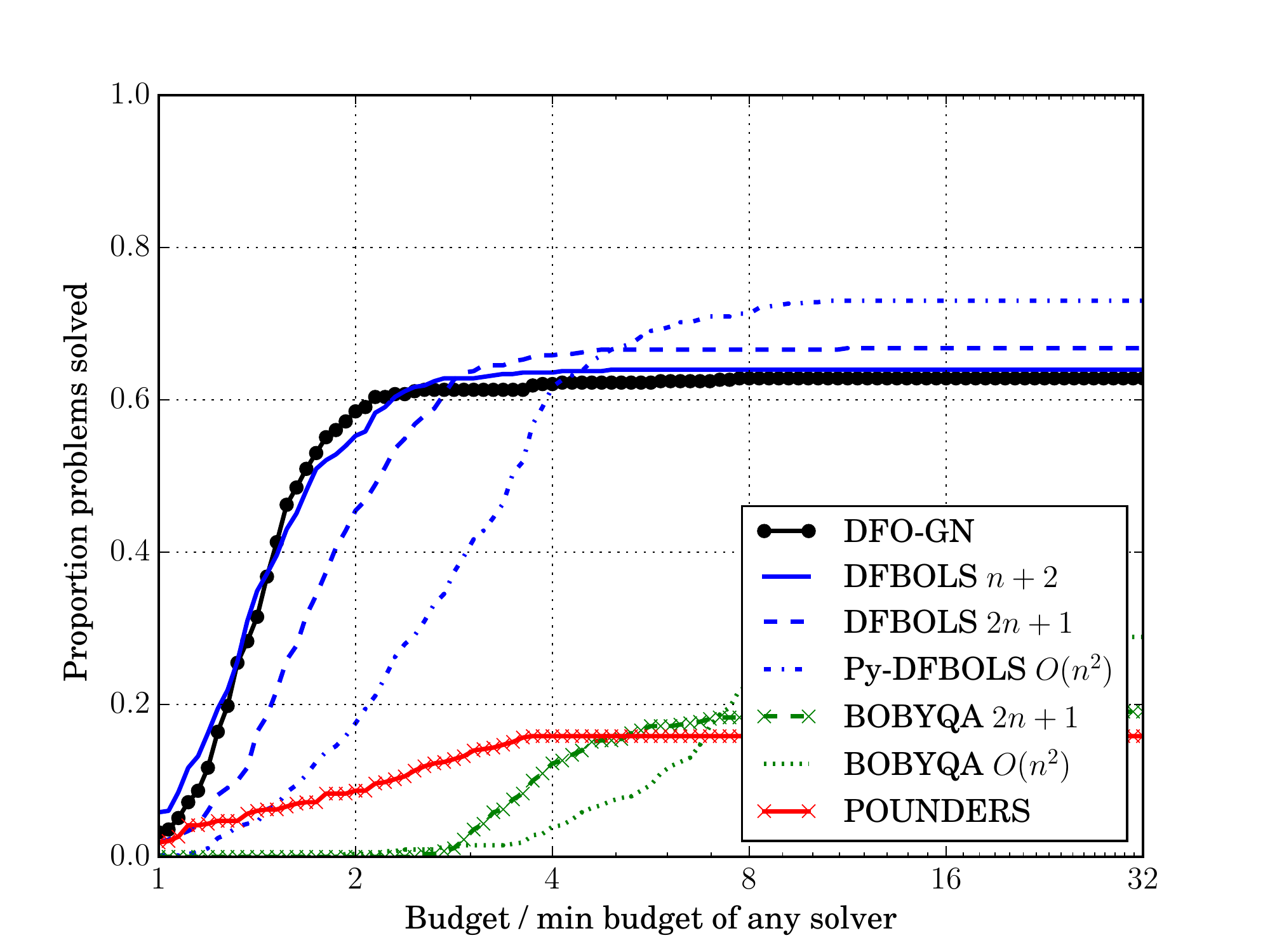}
		\caption{Mult.~Gaussian, performance profile}
	\end{subfigure}
	\\
	\begin{subfigure}[b]{0.48\textwidth}
		\includegraphics[width=\textwidth]{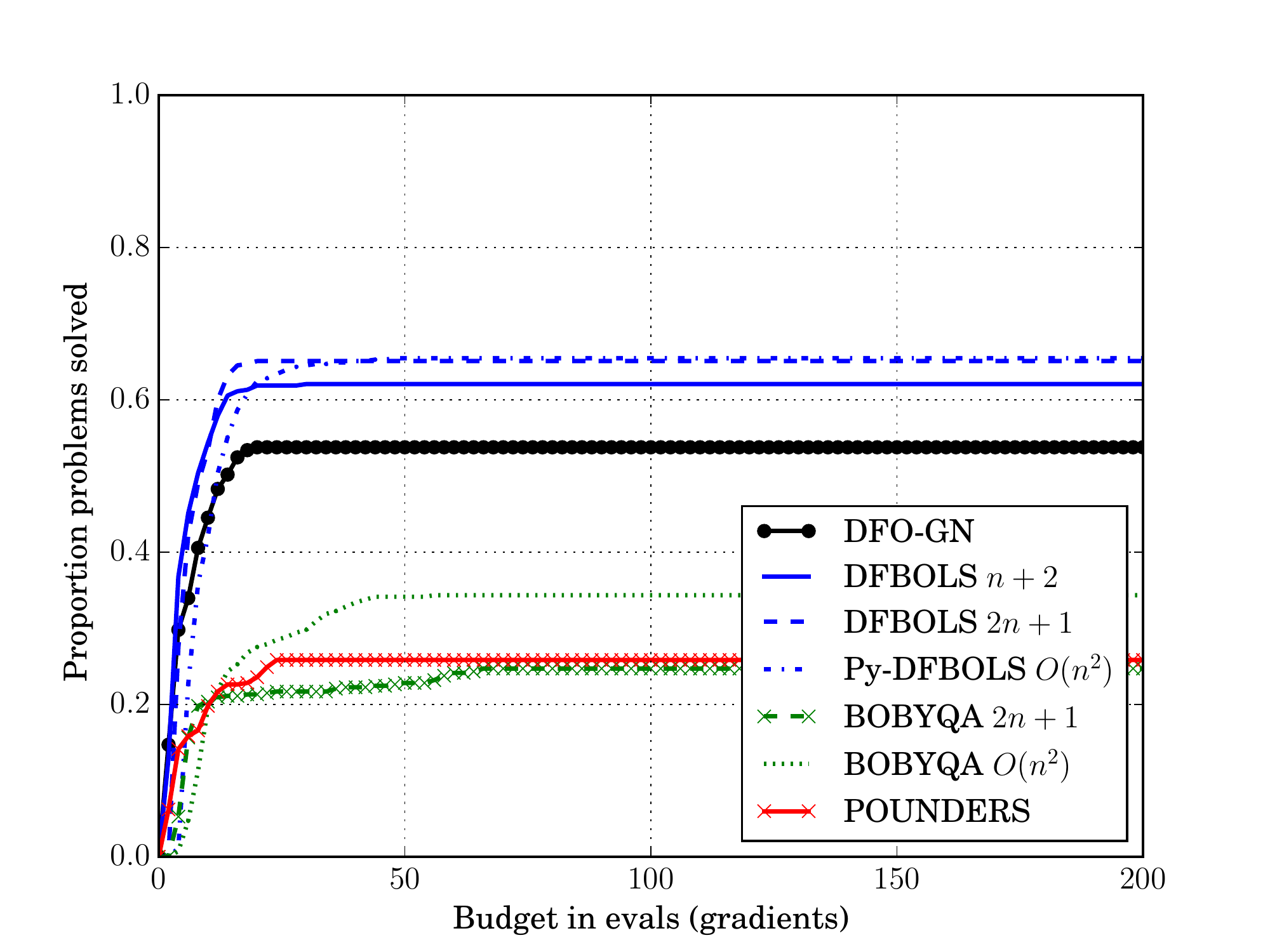}
		\caption{Add.~Gaussian, data profile}
	\end{subfigure}
	~
	\begin{subfigure}[b]{0.48\textwidth}
		\includegraphics[width=\textwidth]{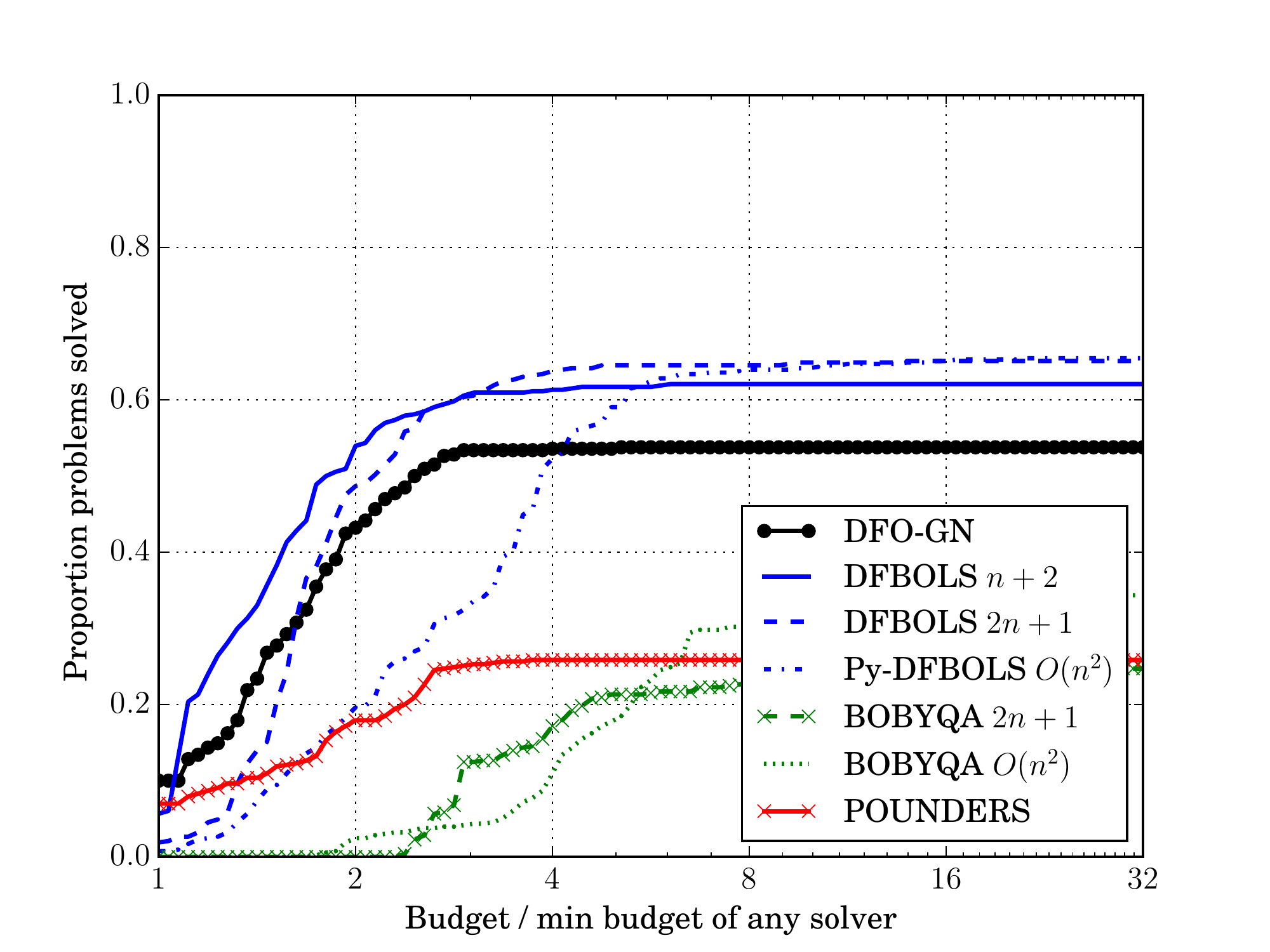}
		\caption{Add.~Gaussian, performance profile}
	\end{subfigure}
	\\
	\begin{subfigure}[b]{0.48\textwidth}
		\includegraphics[width=\textwidth]{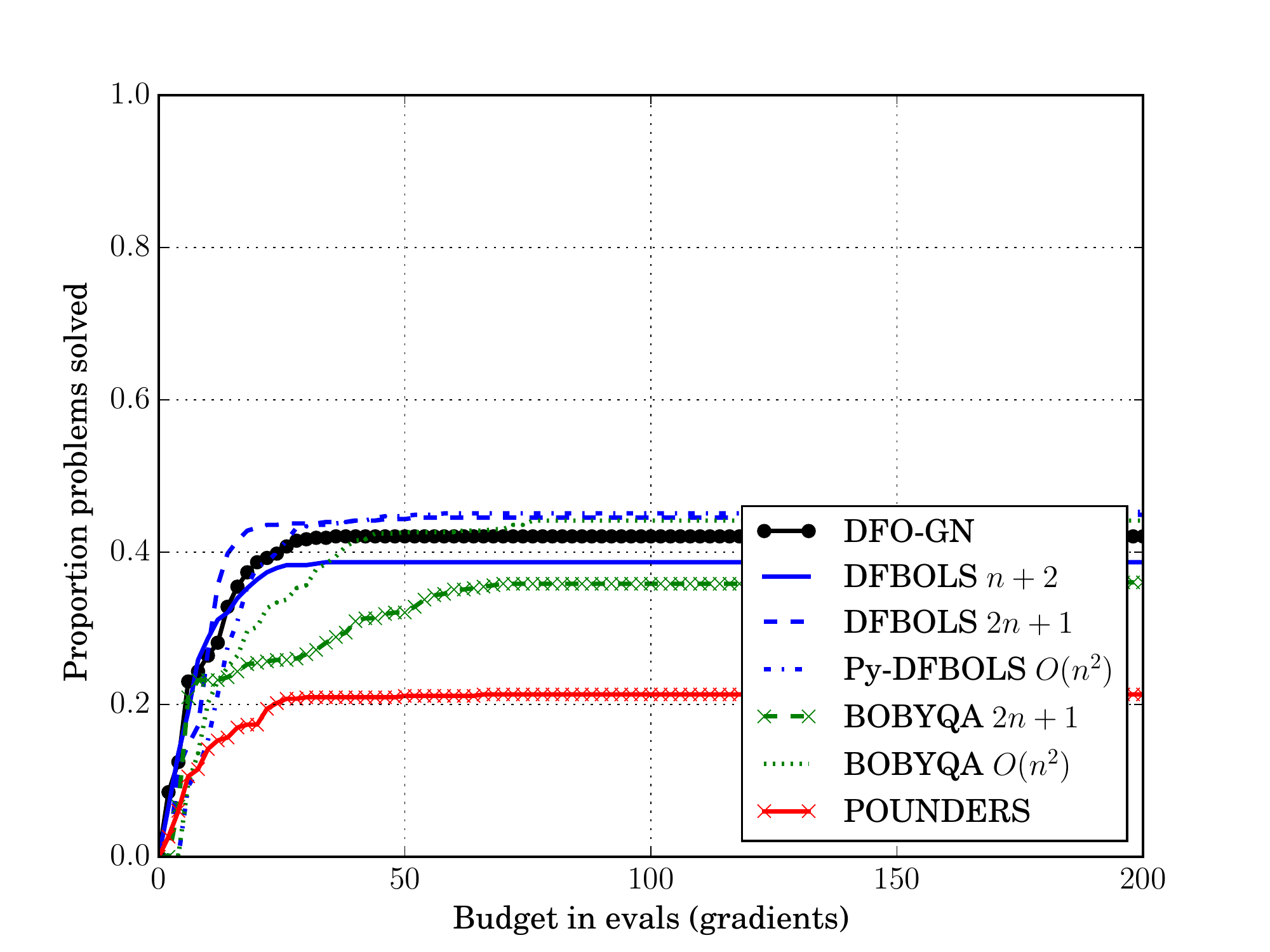}
		\caption{Add.~$\chi^2$, data profile}
	\end{subfigure}
	~
	\begin{subfigure}[b]{0.48\textwidth}
		\includegraphics[width=\textwidth]{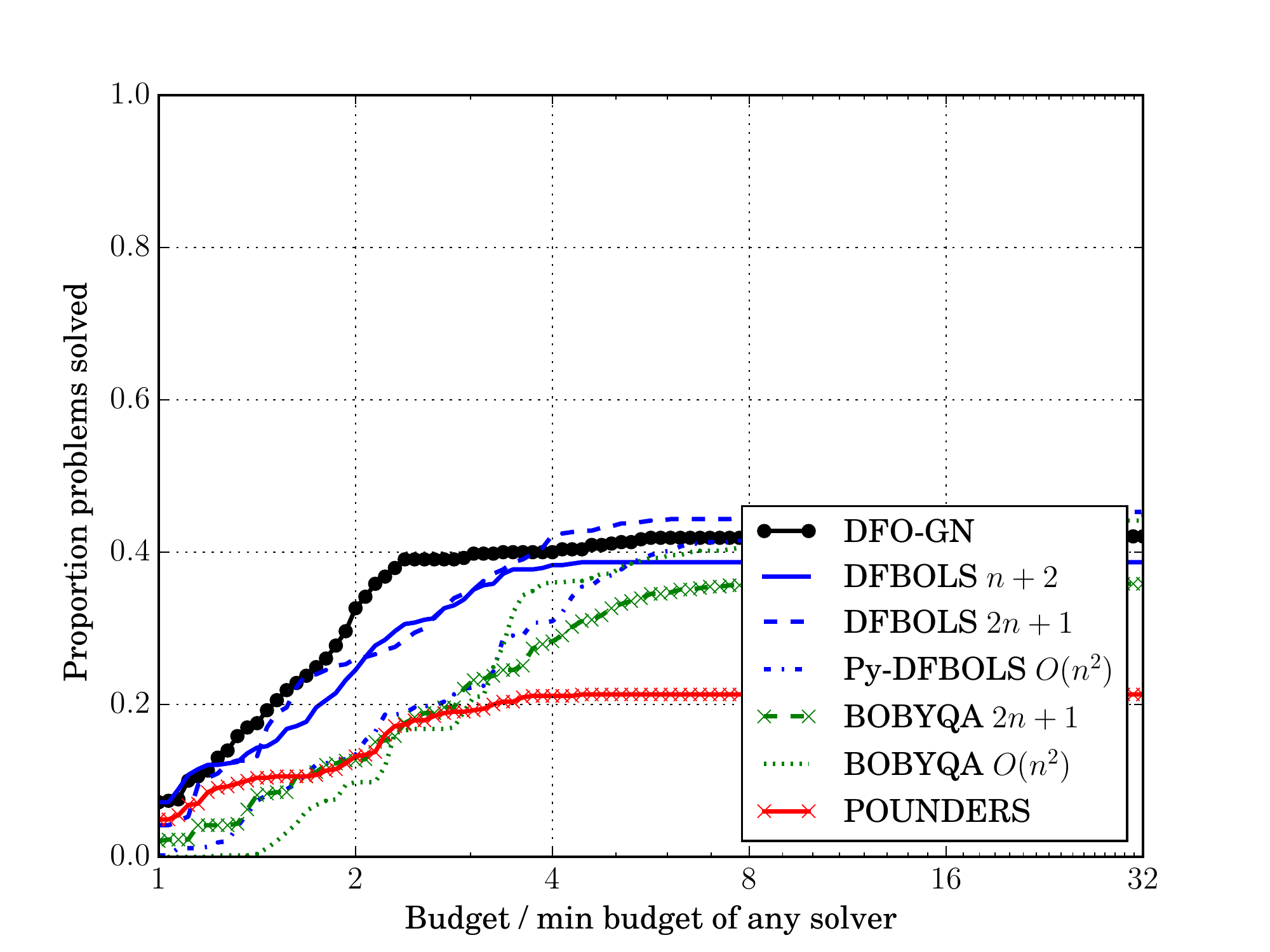}
		\caption{Add.~$\chi^2$, performance profile}
	\end{subfigure}
	\caption{Comparison of DFO-GN with BOBYQA, DFBOLS and POUNDERS for objectives with multiplicative Gaussian, additive Gaussian and additive $\chi^2$ noise with $\sigma=10^{-2}$, to accuracy $\tau=10^{-9}$ (average of 10 runs for each solver). For the BOBYQA and DFBOLS runs, $n+2$, $2n+1$ and $\bigO(n^2)=(n+1)(n+2)/2$ are the number of interpolation points.}
	\label{fig_noisy_high_accuracy_tau9}
\end{figure}

\subsection{Mor\'e \& Wild test set --- noisy objectives for $\tau=10^{-11}$ (compare to \figref{fig_noisy})}
\begin{figure}[H]
	\centering
	\begin{subfigure}[b]{0.48\textwidth}
		\includegraphics[width=\textwidth]{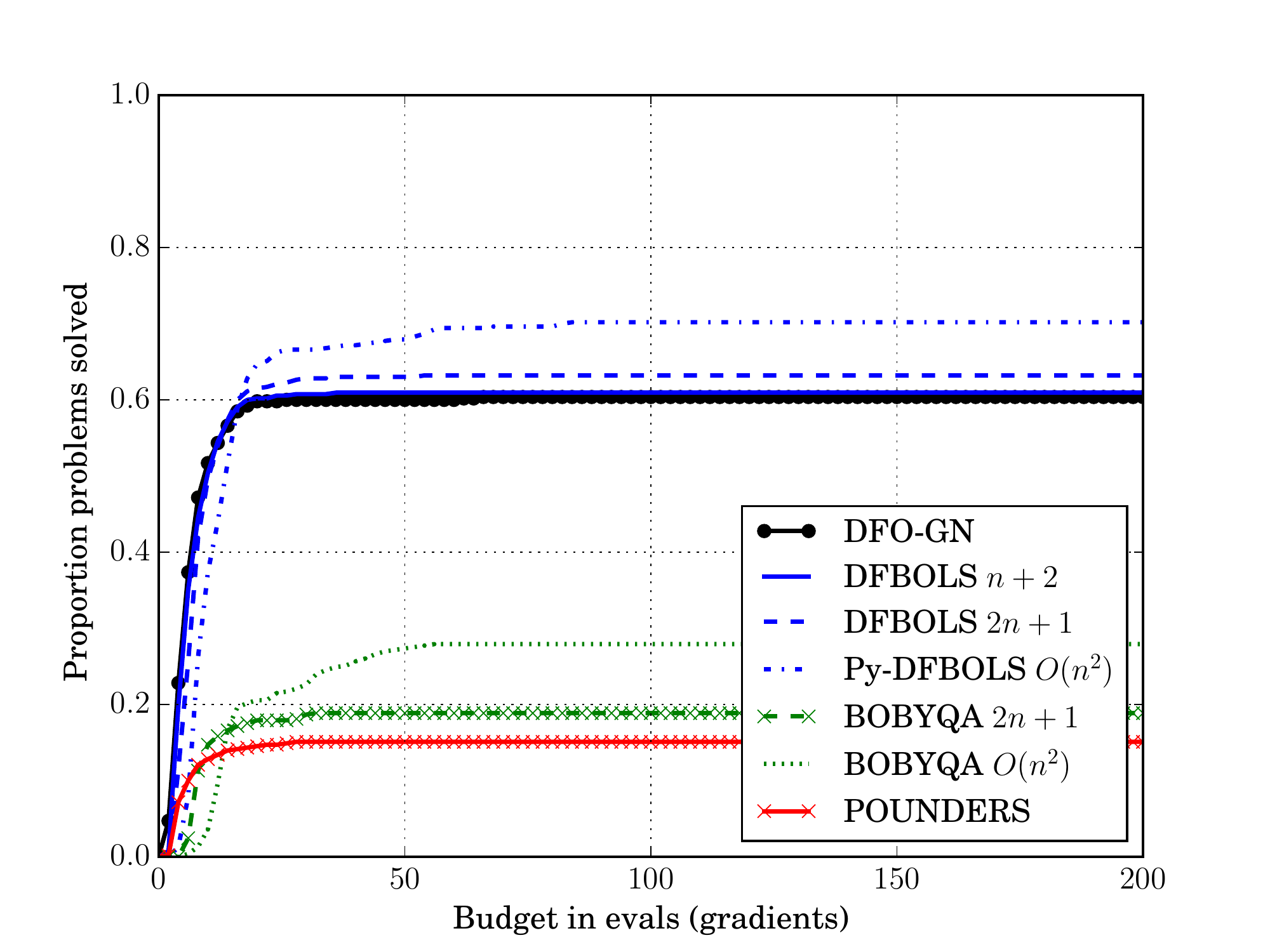}
		\caption{Mult.~Gaussian, data profile}
	\end{subfigure}
	~
	\begin{subfigure}[b]{0.48\textwidth}
		\includegraphics[width=\textwidth]{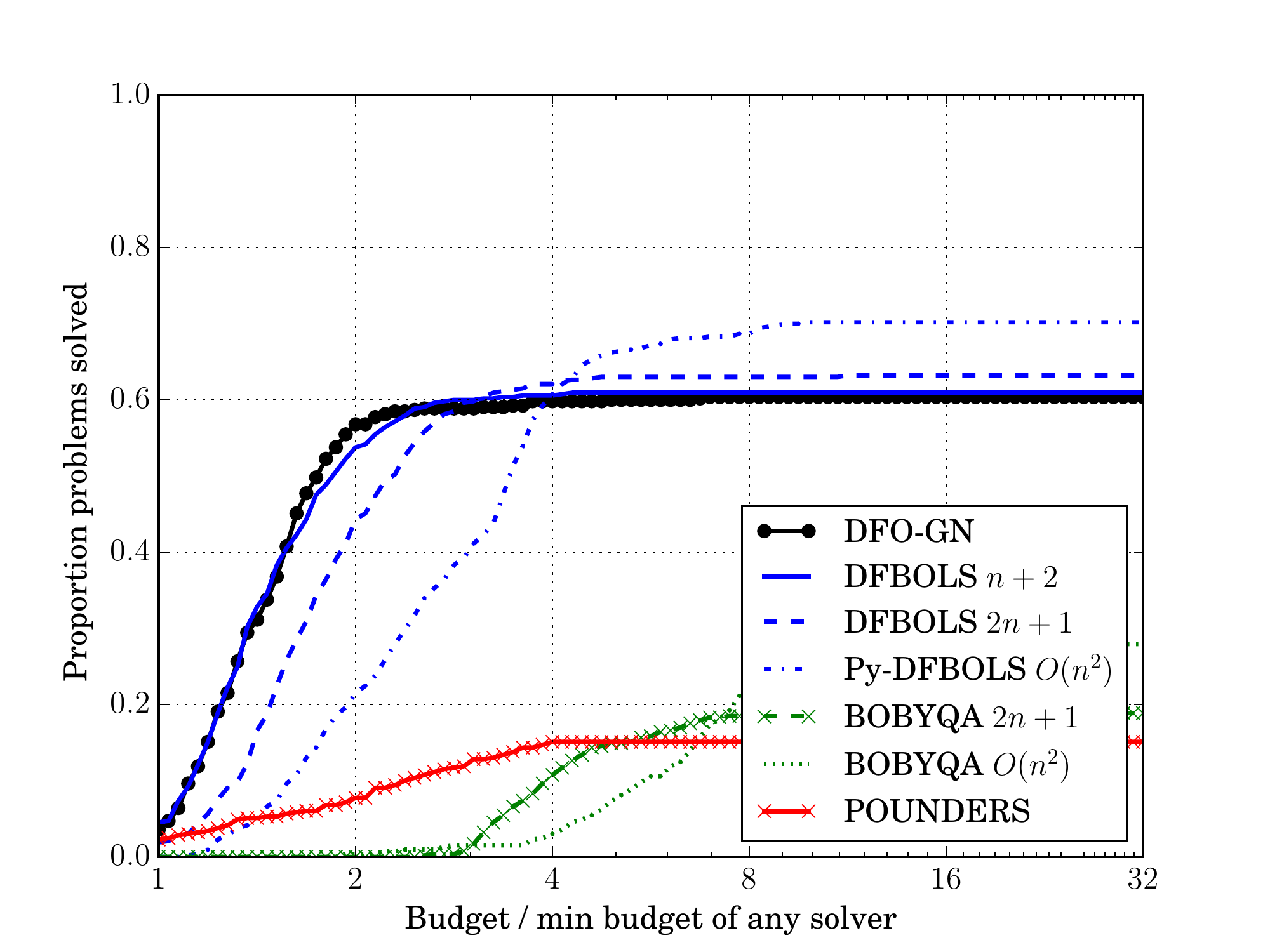}
		\caption{Mult.~Gaussian, performance profile}
	\end{subfigure}
	\\
	\begin{subfigure}[b]{0.48\textwidth}
		\includegraphics[width=\textwidth]{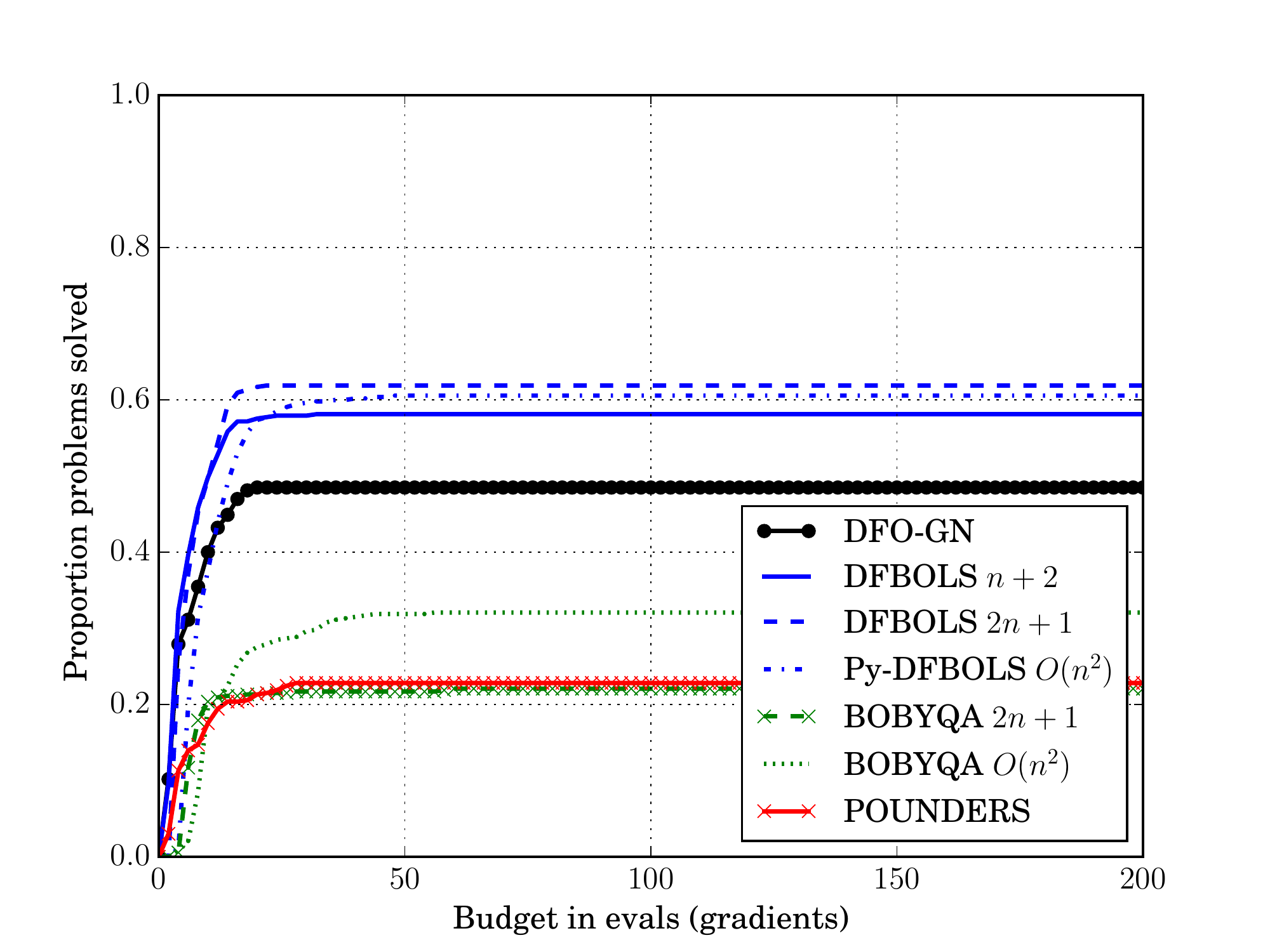}
		\caption{Add.~Gaussian, data profile}
	\end{subfigure}
	~
	\begin{subfigure}[b]{0.48\textwidth}
		\includegraphics[width=\textwidth]{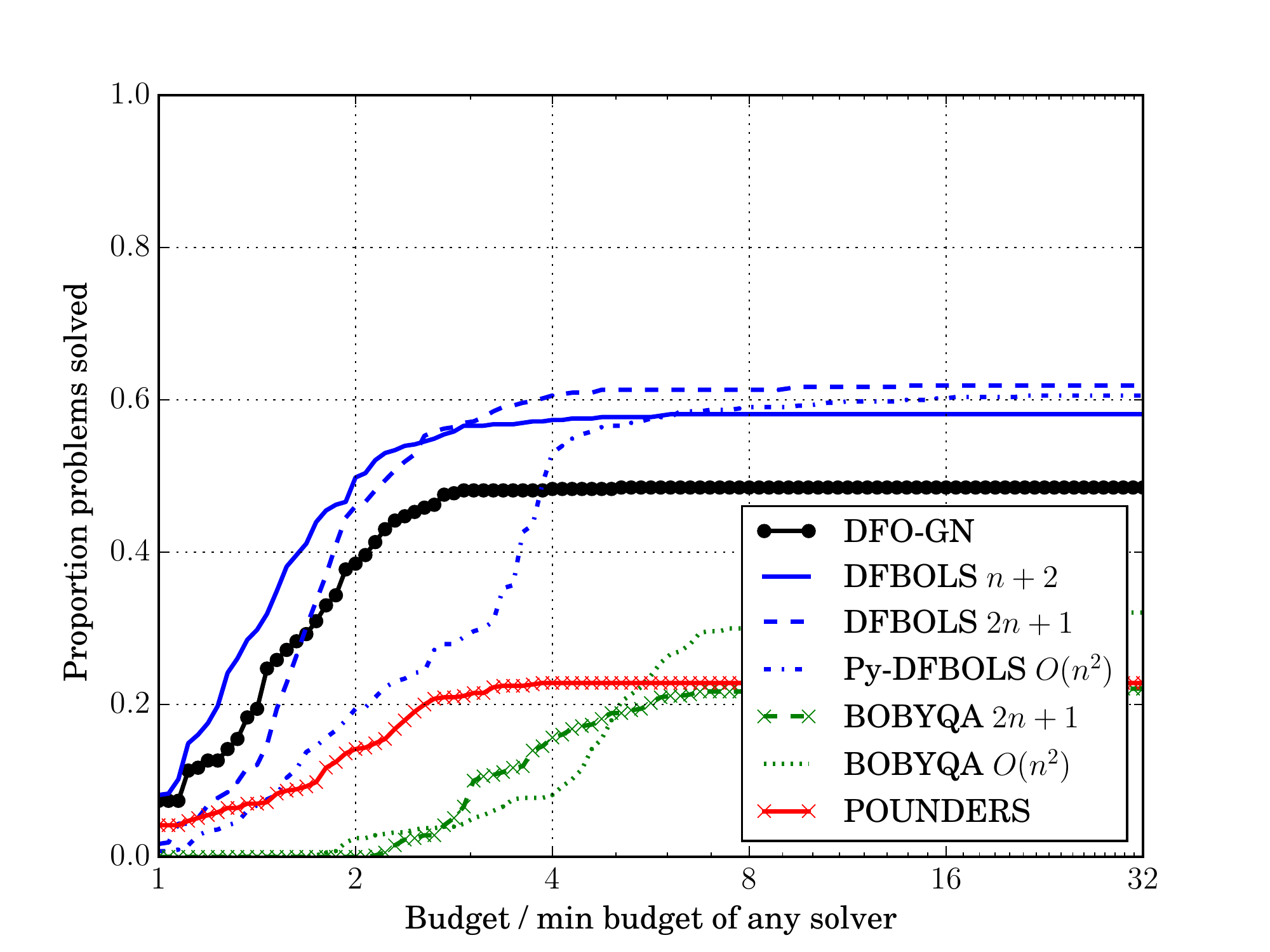}
		\caption{Add.~Gaussian, performance profile}
	\end{subfigure}
	\\
	\begin{subfigure}[b]{0.48\textwidth}
		\includegraphics[width=\textwidth]{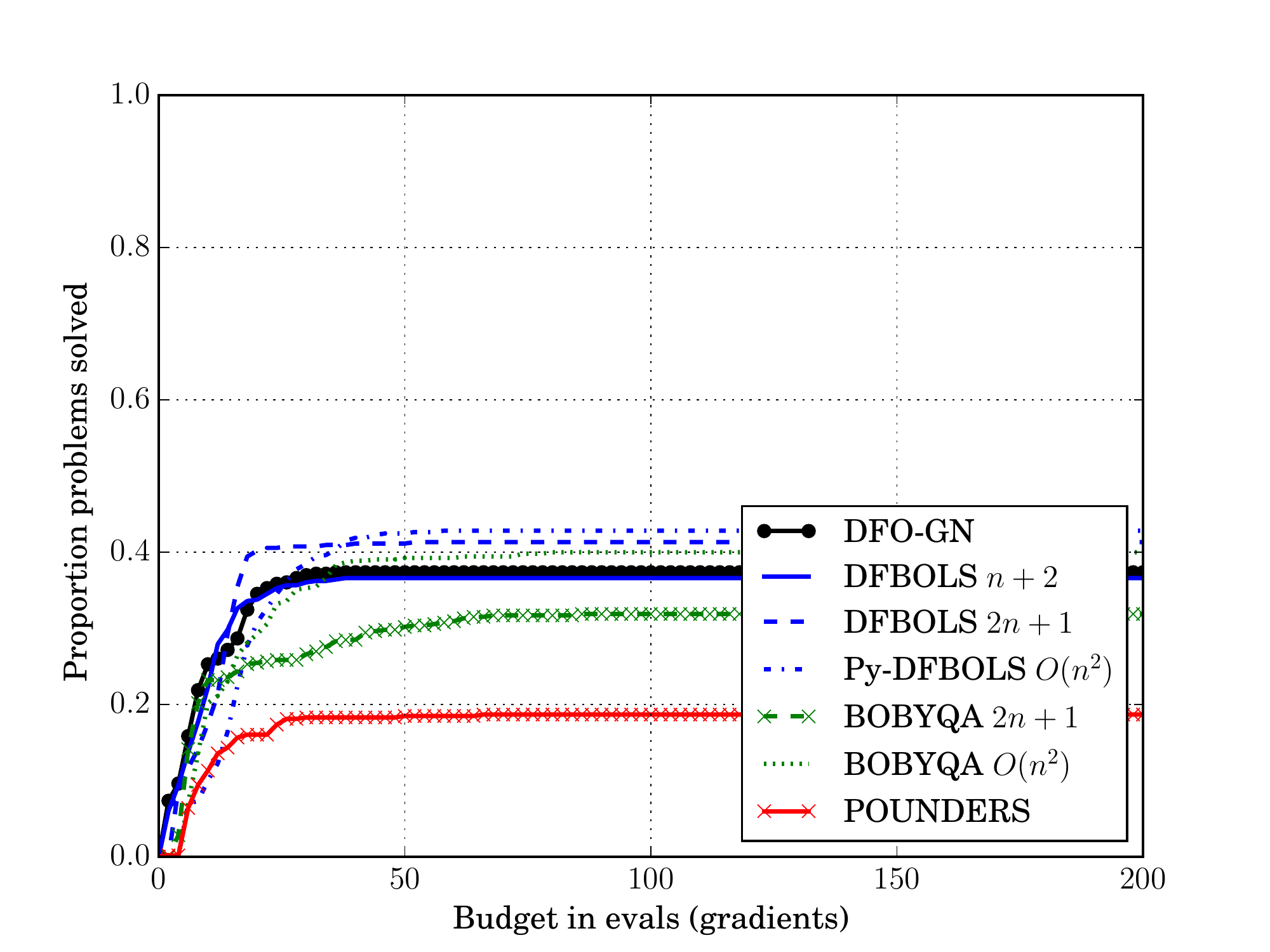}
		\caption{Add.~$\chi^2$, data profile}
	\end{subfigure}
	~
	\begin{subfigure}[b]{0.48\textwidth}
		\includegraphics[width=\textwidth]{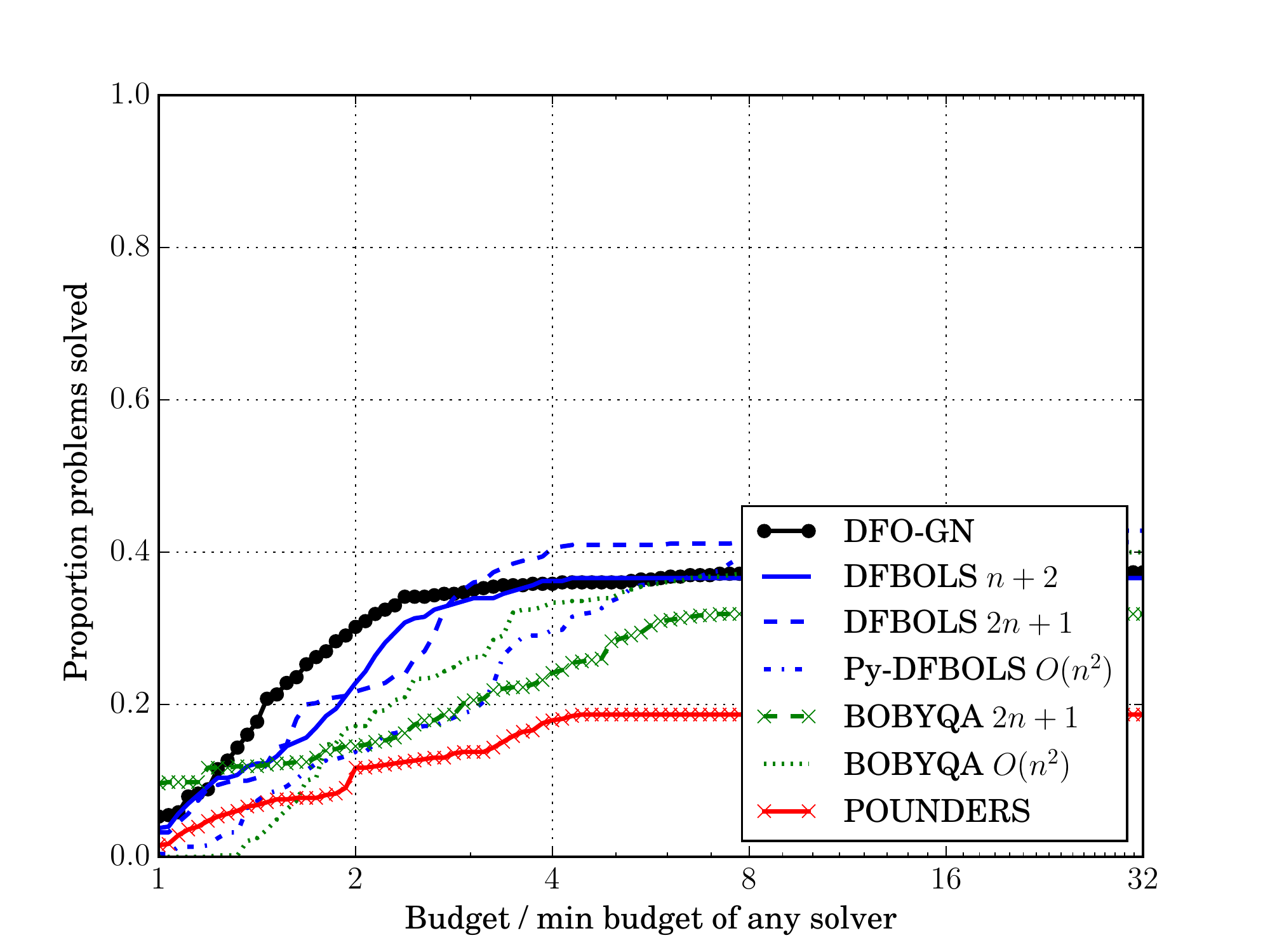}
		\caption{Add.~$\chi^2$, performance profile}
	\end{subfigure}
	\caption{Comparison of DFO-GN with BOBYQA, DFBOLS and POUNDERS for objectives with multiplicative Gaussian, additive Gaussian and additive $\chi^2$ noise with $\sigma=10^{-2}$, to accuracy $\tau=10^{-11}$ (average of 10 runs for each solver). For the BOBYQA and DFBOLS runs, $n+2$, $2n+1$ and $\bigO(n^2)=(n+1)(n+2)/2$ are the number of interpolation points.}
	\label{fig_noisy_high_accuracy_tau11}
\end{figure}

\subsection{Mor\'e \& Wild test set --- nonzero residual problems only (compare to \figref{fig_nonzero})}
\begin{figure}[H]
	\centering
	\begin{subfigure}[b]{0.48\textwidth}
		\includegraphics[width=\textwidth]{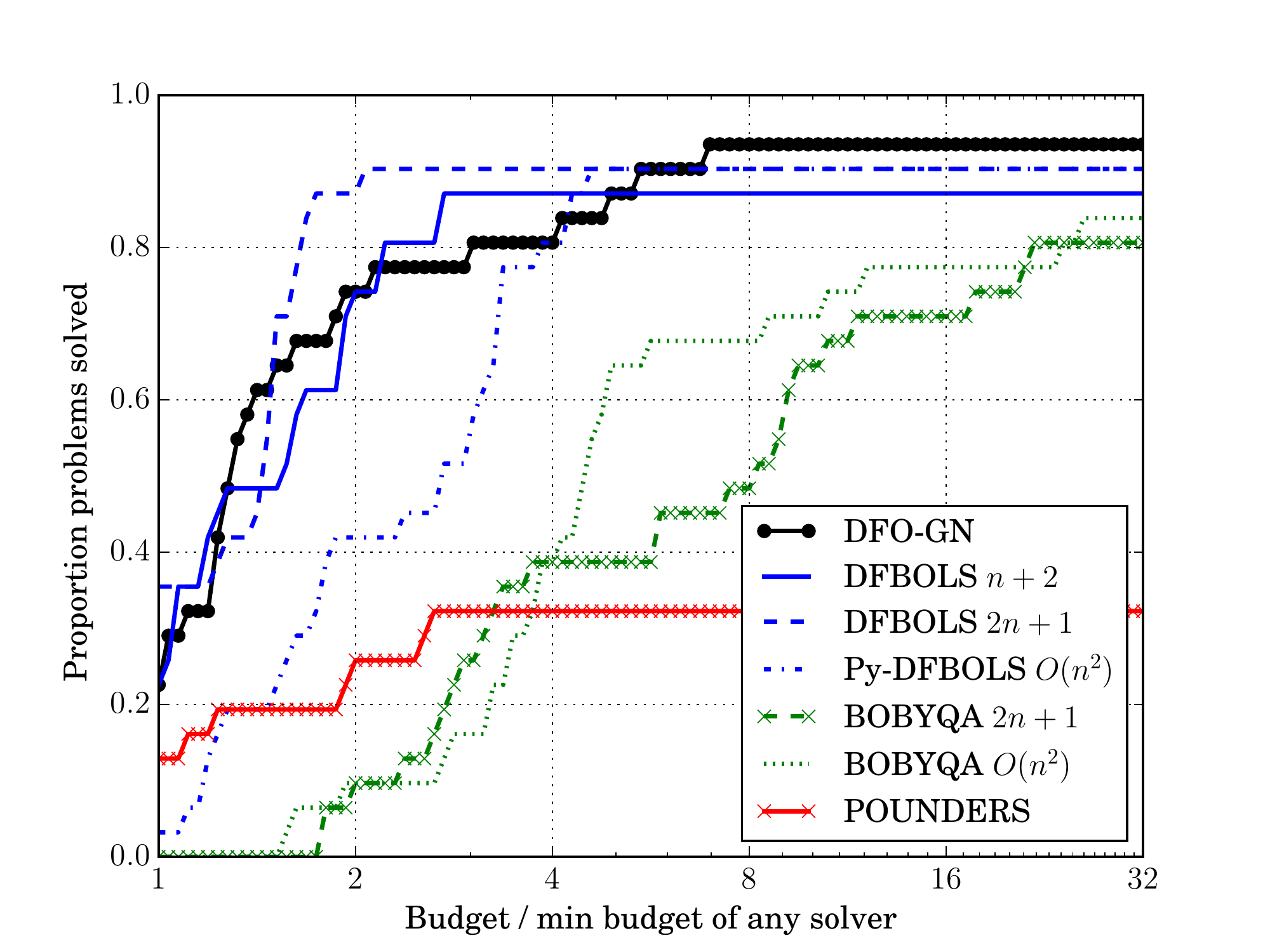}
		\caption{Smooth objective, $\tau=10^{-7}$}
	\end{subfigure}
	~
	\begin{subfigure}[b]{0.48\textwidth}
		\includegraphics[width=\textwidth]{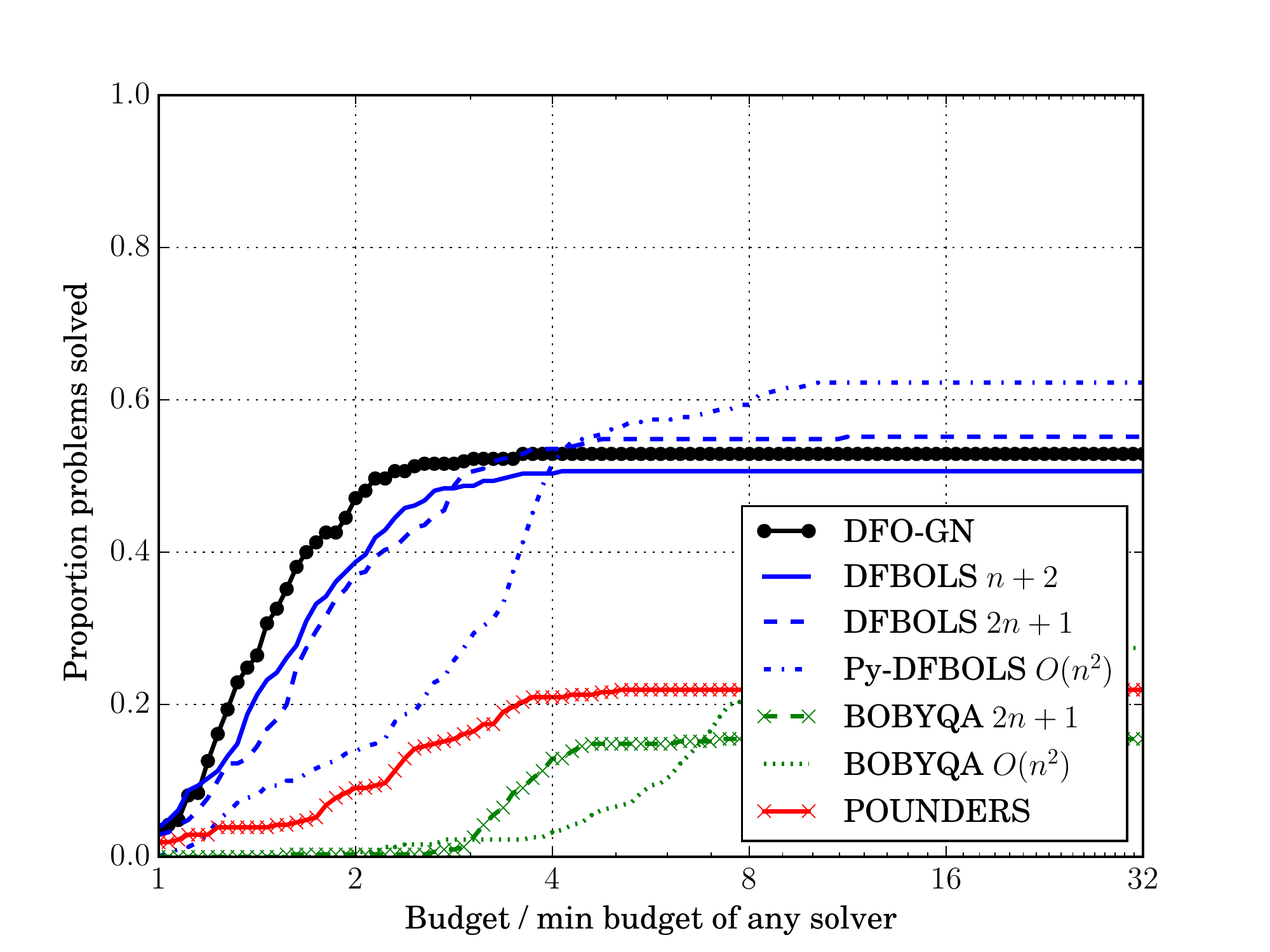}
		\caption{Mult.~Gaussian, $\tau=10^{-7}$}
	\end{subfigure}
	\\
	\begin{subfigure}[b]{0.48\textwidth}
		\includegraphics[width=\textwidth]{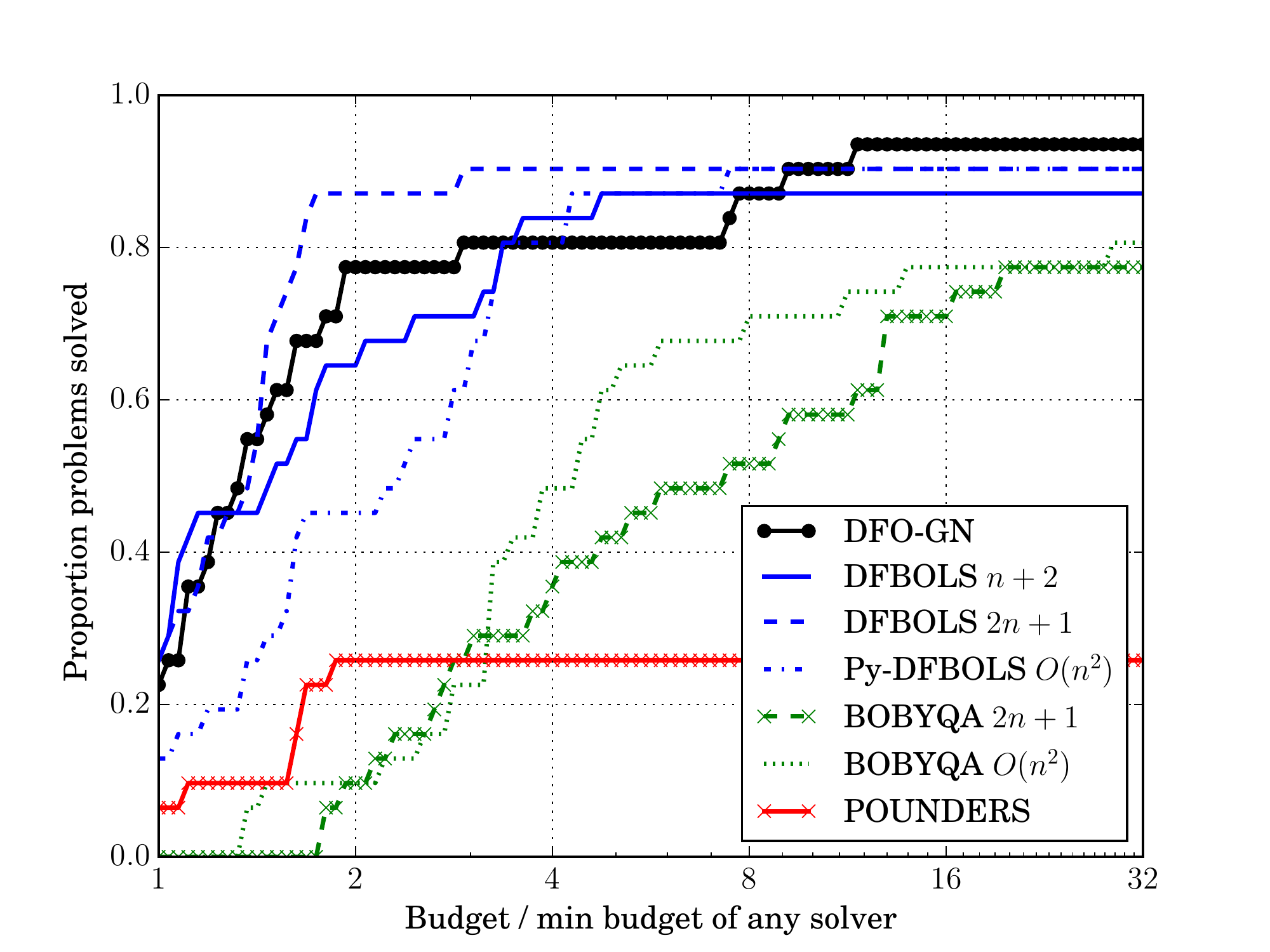}
		\caption{Smooth objective, $\tau=10^{-9}$}
	\end{subfigure}
	~
	\begin{subfigure}[b]{0.48\textwidth}
		\includegraphics[width=\textwidth]{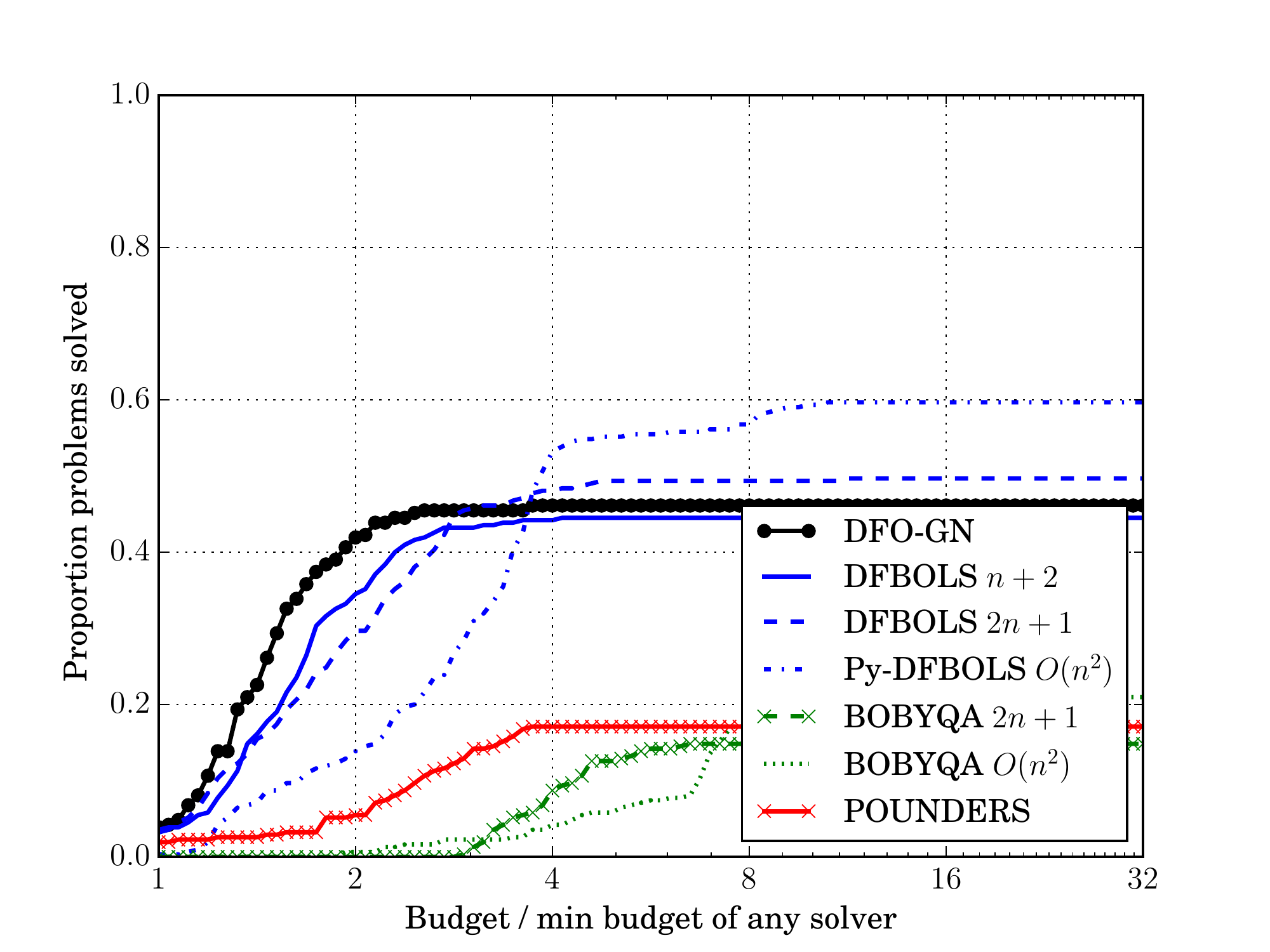}
		\caption{Mult.~Gaussian, $\tau=10^{-9}$}
	\end{subfigure}
	\\
	\begin{subfigure}[b]{0.48\textwidth}
		\includegraphics[width=\textwidth]{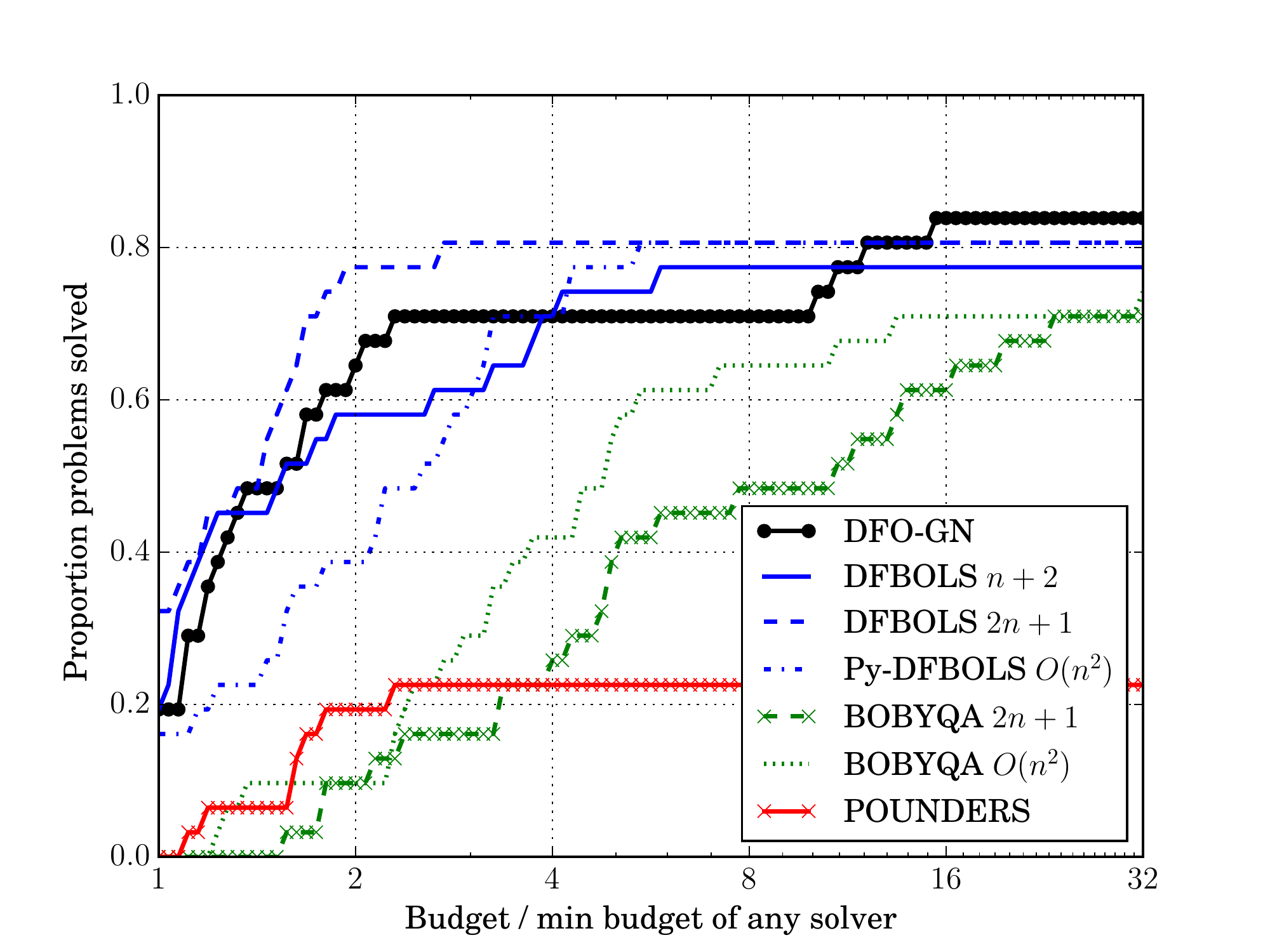}
		\caption{Smooth objective, $\tau=10^{-11}$}
	\end{subfigure}
	~
	\begin{subfigure}[b]{0.48\textwidth}
		\includegraphics[width=\textwidth]{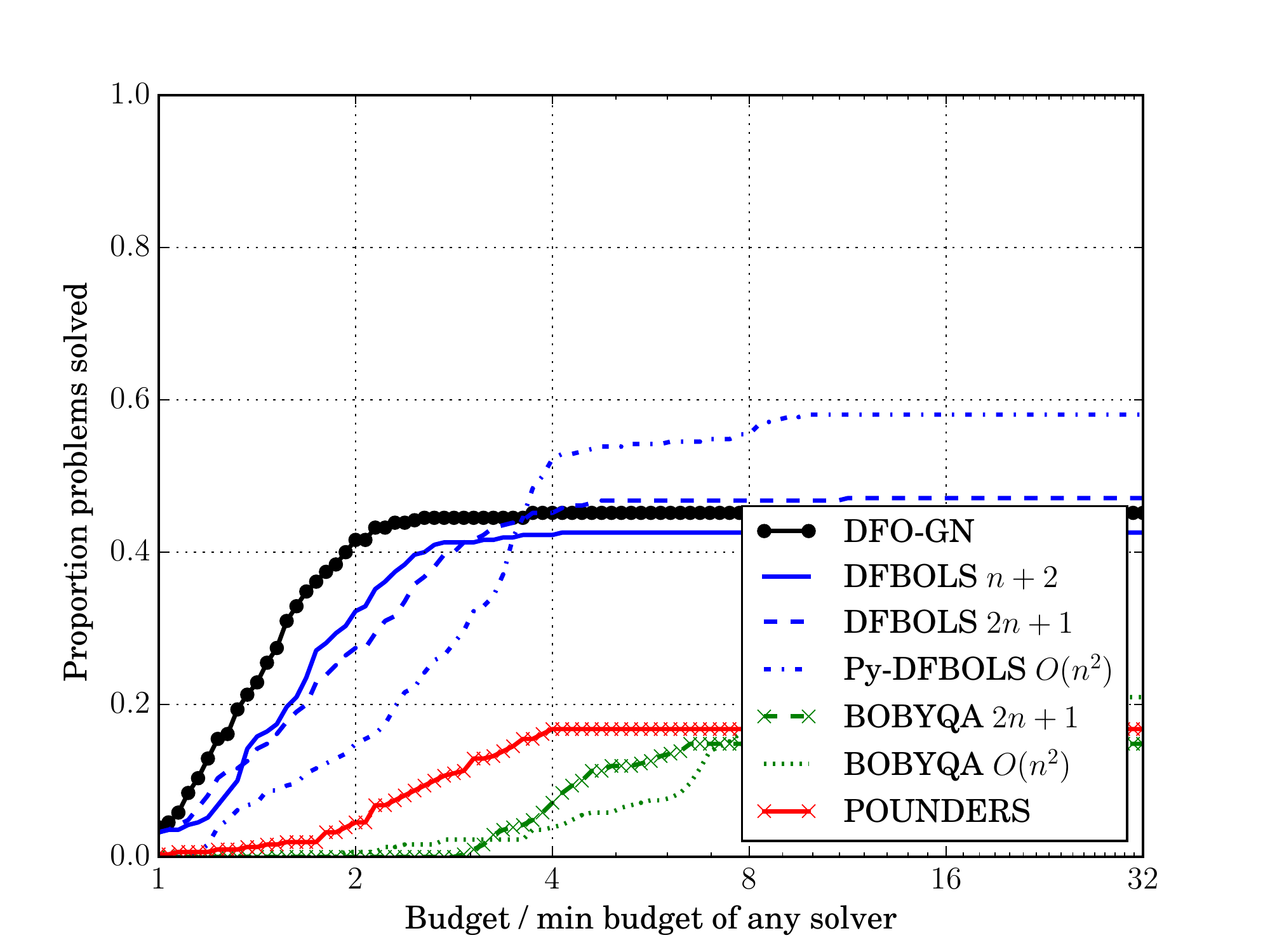}
		\caption{Mult.~Gaussian, $\tau=10^{-11}$}
	\end{subfigure}
	\caption{Performance profile comparison of DFO-GN with BOBYQA, DFBOLS and POUNDERS for nonzero residual problems only, to accuracy $\tau\in\{10^{-5},10^{-9},10^{-11}\}$. For the BOBYQA and DFBOLS runs, $n+2$, $2n+1$ and $\bigO(n^2)=(n+1)(n+2)/2$ are the number of interpolation points. For noisy objectives, results shown are an average of 10 runs for each solver.}
	\label{fig_nonzero_high_accuracy}
\end{figure}

\subsection{CUTEst test problems (compare to \figref{fig_smooth_cutest})}
\begin{figure}[H]
	\centering
	\begin{subfigure}[b]{0.48\textwidth}
		\includegraphics[width=\textwidth]{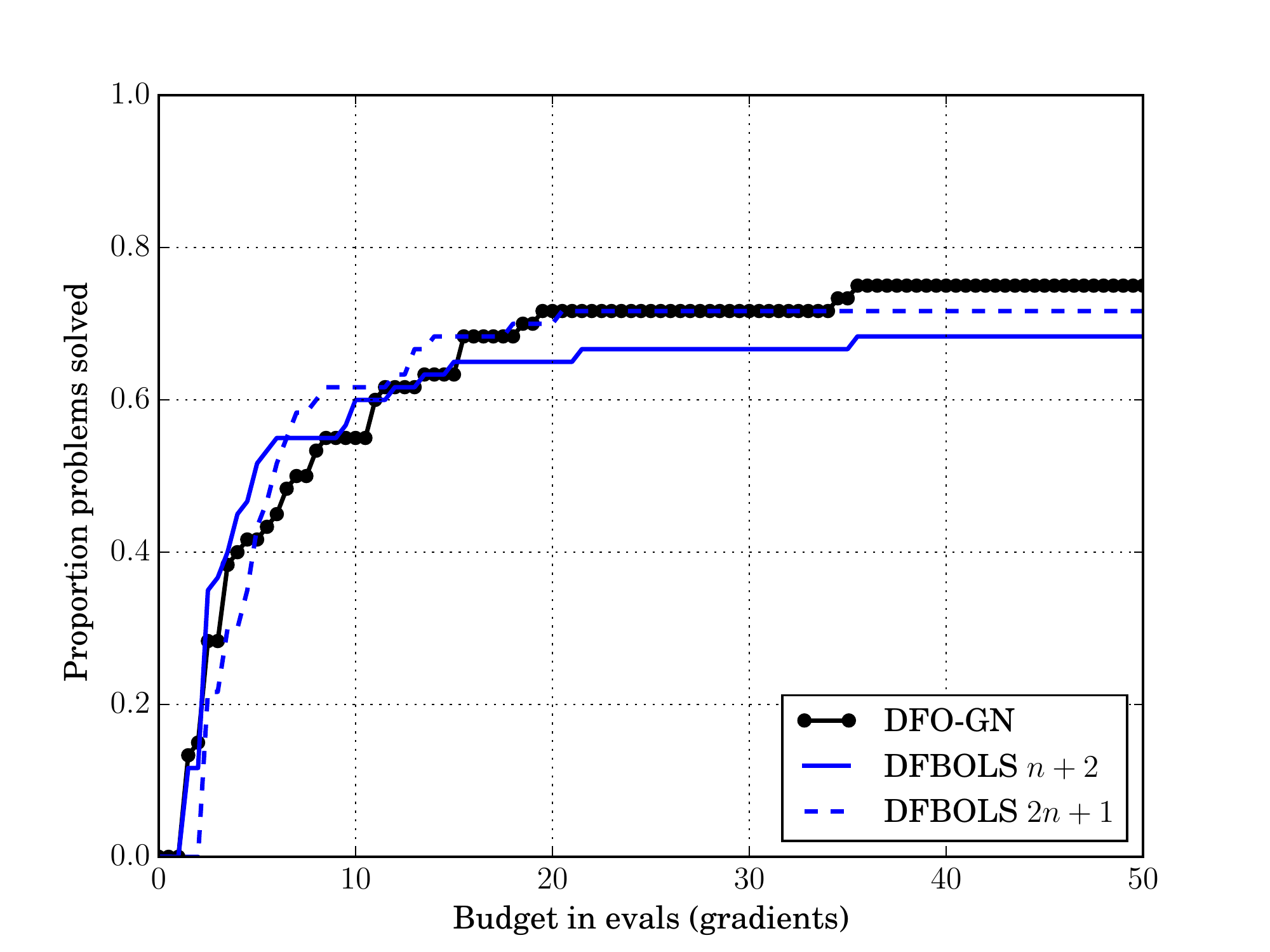}
		\caption{Data profile, $\tau=10^{-7}$}
	\end{subfigure}
	~
	\begin{subfigure}[b]{0.48\textwidth}
		\includegraphics[width=\textwidth]{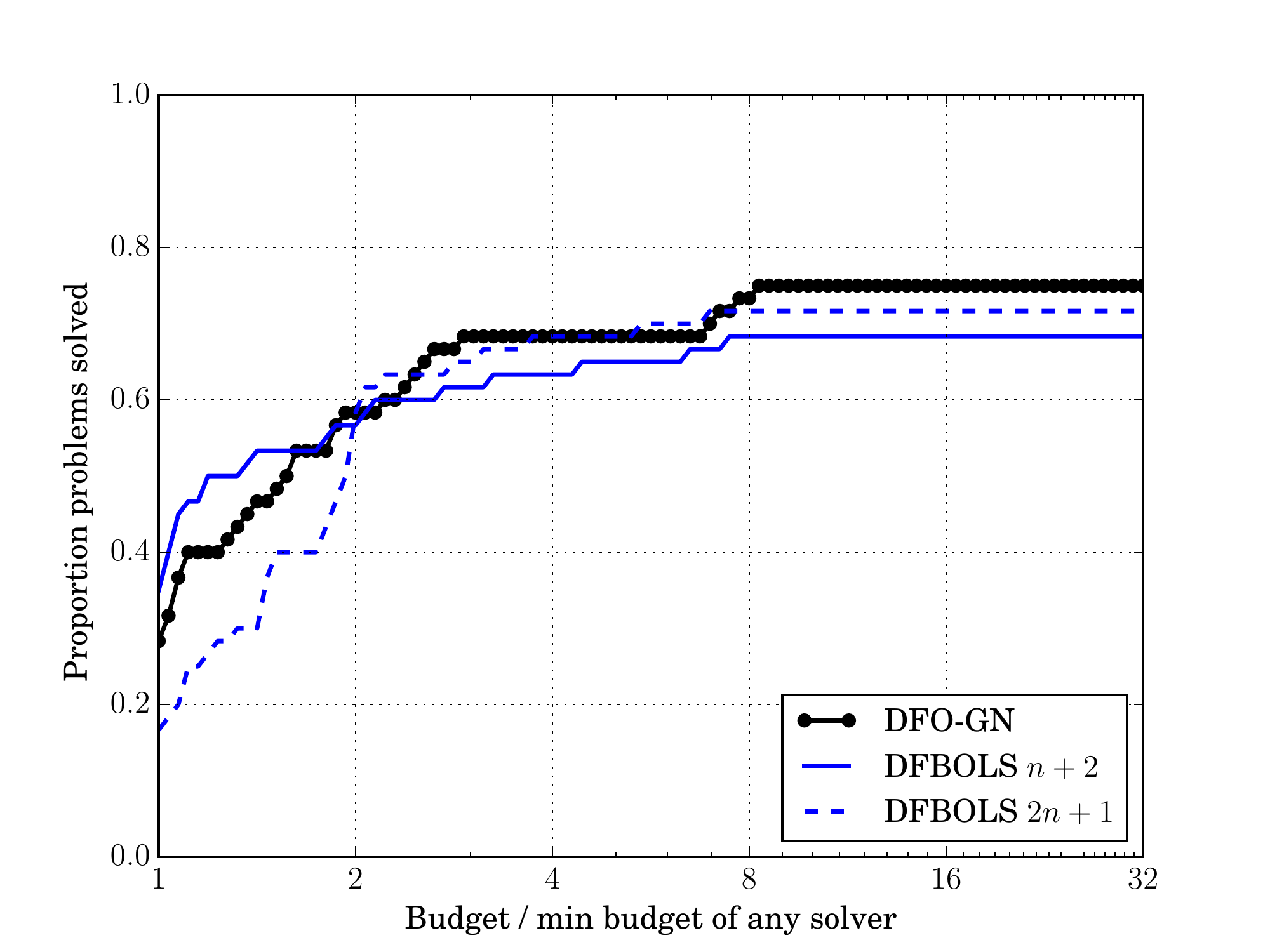}
		\caption{Perf profile, $\tau=10^{-7}$}
	\end{subfigure}
	\\
	\begin{subfigure}[b]{0.48\textwidth}
		\includegraphics[width=\textwidth]{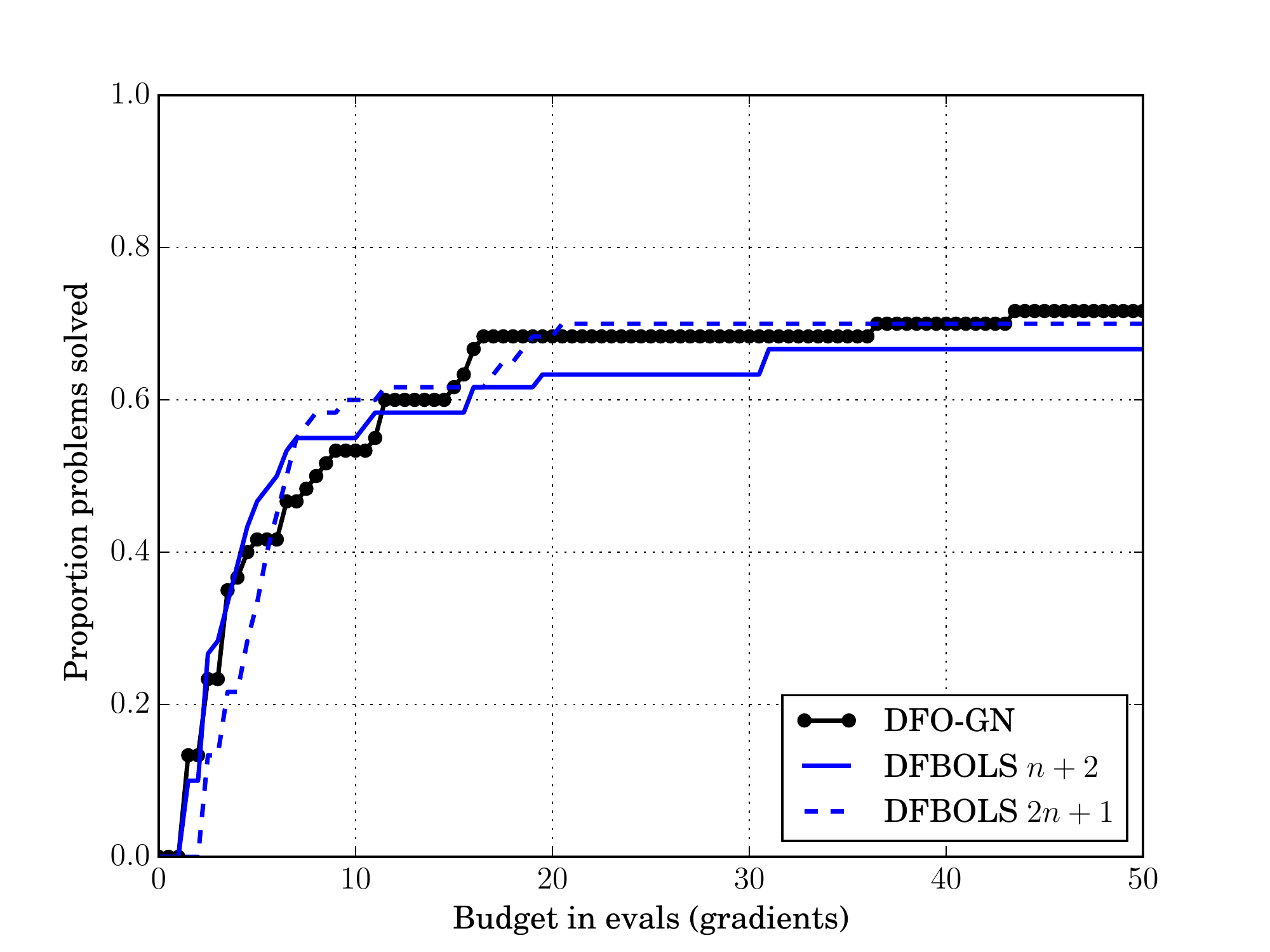}
		\caption{Data profile, $\tau=10^{-9}$}
	\end{subfigure}
	~
	\begin{subfigure}[b]{0.48\textwidth}
		\includegraphics[width=\textwidth]{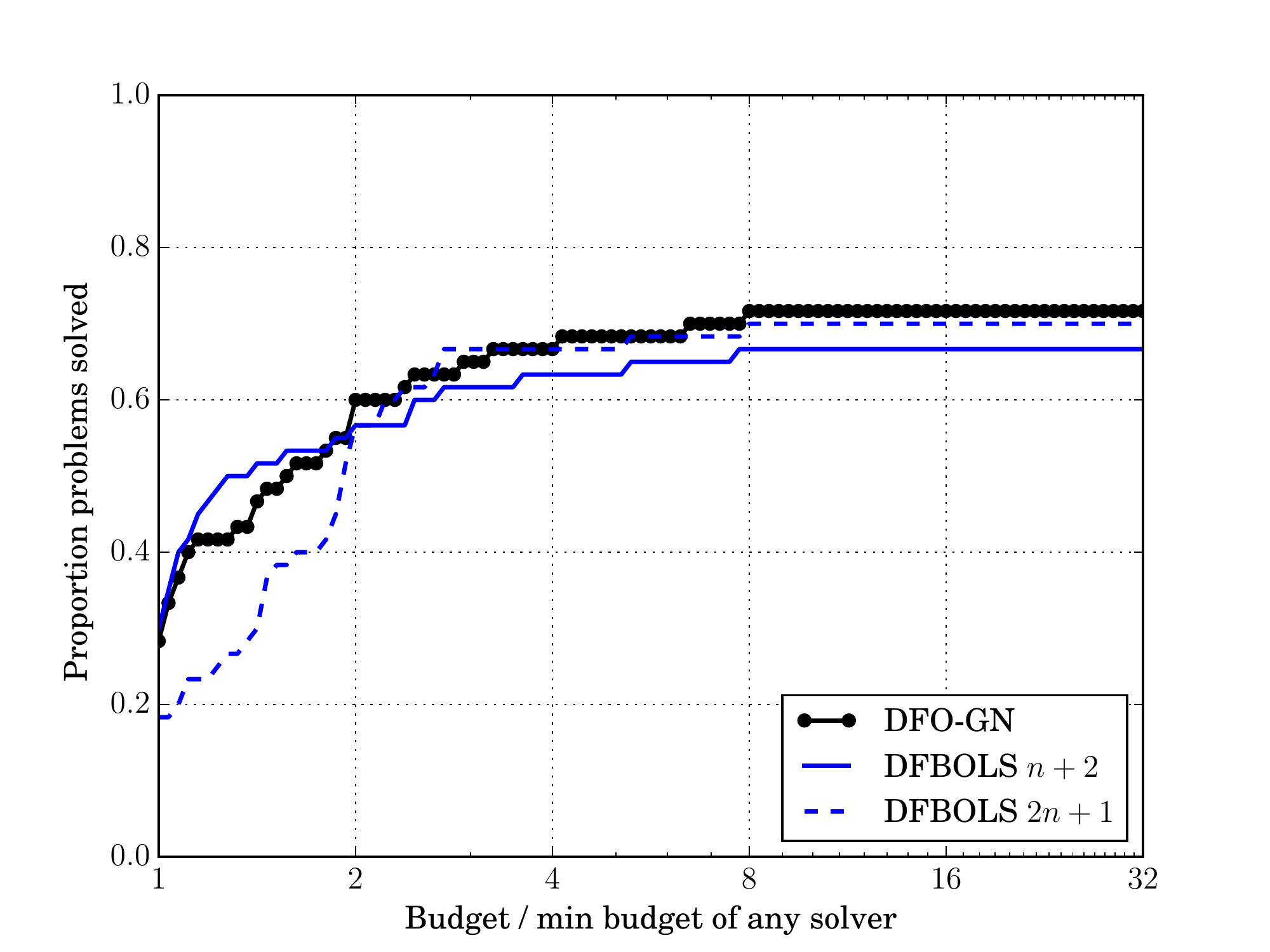}
		\caption{Perf profile, $\tau=10^{-9}$}
	\end{subfigure}
	\\
	\begin{subfigure}[b]{0.48\textwidth}
		\includegraphics[width=\textwidth]{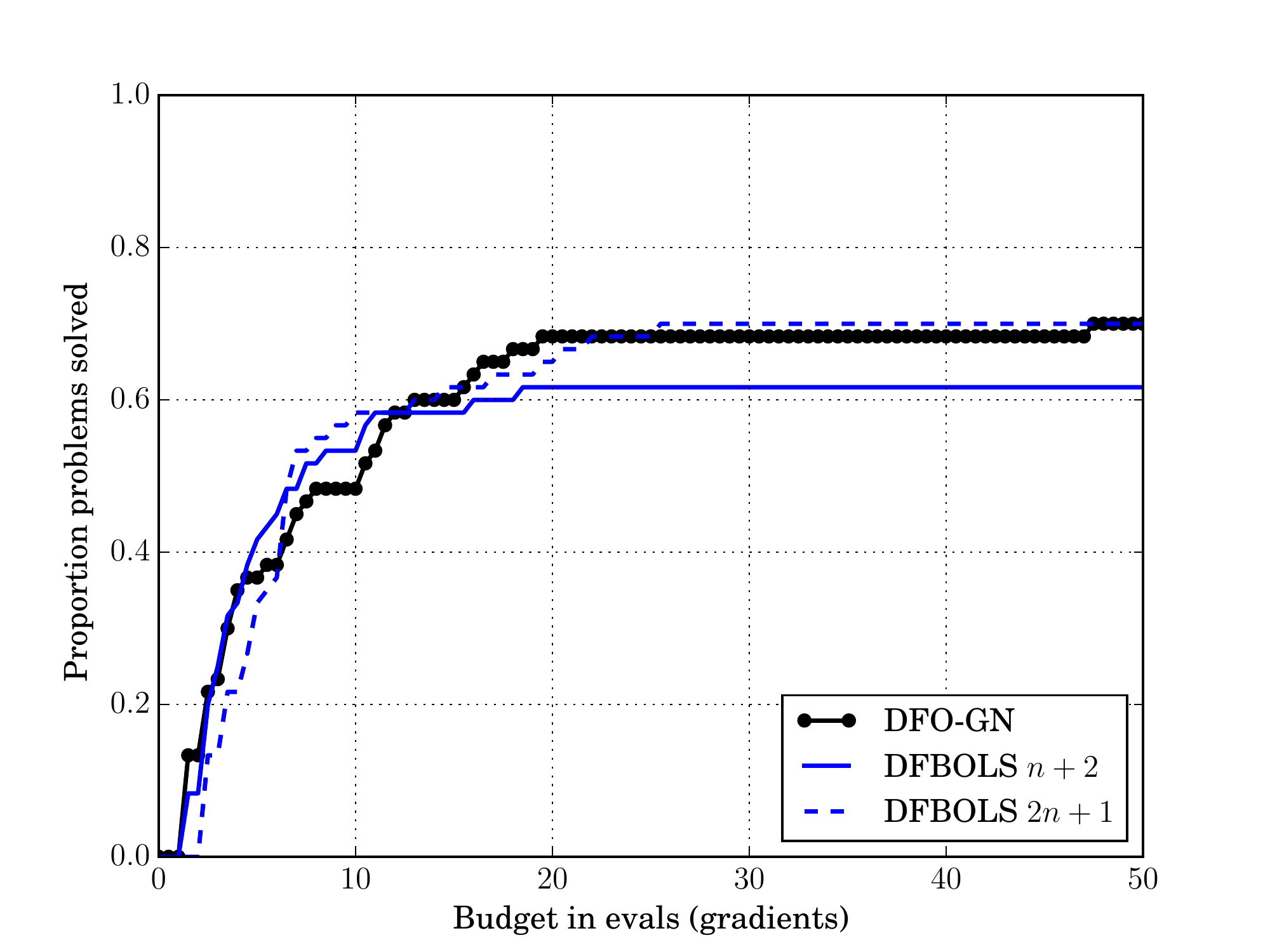}
		\caption{Data profile, $\tau=10^{-11}$}
	\end{subfigure}
	~
	\begin{subfigure}[b]{0.48\textwidth}
		\includegraphics[width=\textwidth]{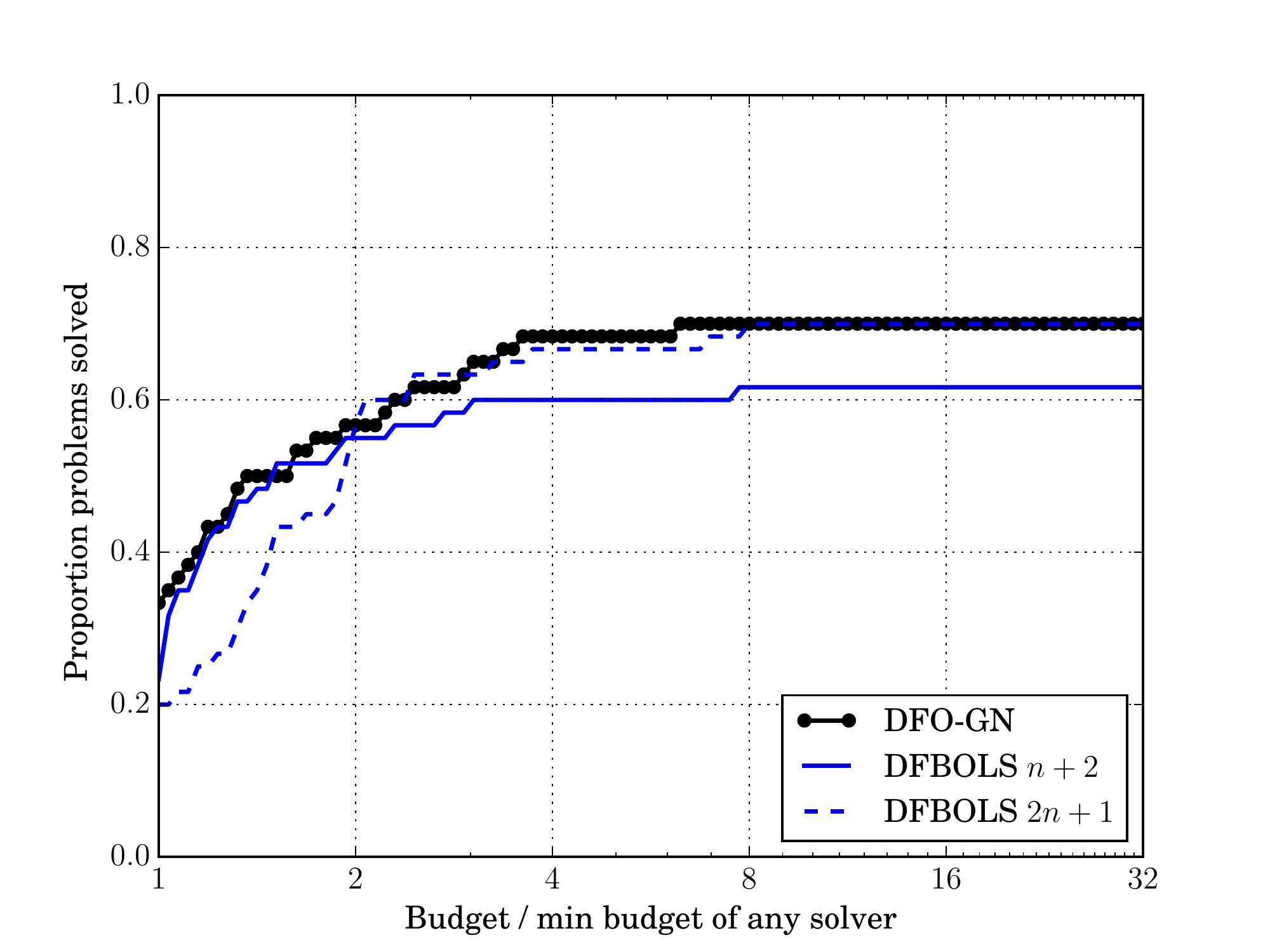}
		\caption{Perf profile, $\tau=10^{-11}$}
	\end{subfigure}
	\caption{Comparison of DFO-GN with DFBOLS for smooth objectives from the set of medium-sized CUTEst problems, to accuracy $\tau\in\{10^{-7},10^{-9},10^{-11}\}$. For the DFBOLS runs, $n+2$, $2n+1$ are the number of interpolation points.}
	\label{fig_smooth_cutest_high_accuracy}
\end{figure}

\end{document}